\documentclass[a4paper,12pt]{amsart}

\title[Euler systems 
of elliptic units and
the pseudo-isomorphism class]{On higher Fitting ideals of 
Iwasawa modules of ideal class groups 
over imaginary quadratic fields and 
Euler systems of elliptic units II}
\author{Tatsuya Ohshita}
\address{Department of Mathematics, 
Graduate School of Science and Engineering,
Ehime University 2--5,
Bunkyo-cho, Matsuyama-shi, Ehime 790--8577, Japan}
\email{ohshita.tatsuya.nz@ehime-u.ac.jp}
\date{\today}

\usepackage{amsmath}
\usepackage{amsmath}
\usepackage{amssymb}
\usepackage{amscd}
\usepackage{enumerate}
\usepackage{mathrsfs}
\usepackage[all]{xy}
\usepackage[dvipdfmx]{graphicx}

\newtheorem{thm}{Theorem}[section]
\newtheorem{prop}[thm]{Proposition}
\newtheorem{cor}[thm]{Corollary}
\newtheorem{lem}[thm]{Lemma}
\newtheorem{conj}[thm]{Conjecture}
\newtheorem{claim}[thm]{Claim}
\theoremstyle{definition}
\newtheorem{dfn}[thm]{Definition}
\newtheorem{exa}[thm]{Example}
\newtheorem{rem}[thm]{Remark}

\def\Gal{\mathop{\mathrm{Gal}}\nolimits}

\def\Fitt{\mathop{\mathrm{Fitt}}\nolimits}
\def\det{\mathop{\mathrm{det}}\nolimits}
\def\Im{\mathop{\mathrm{Im}}\nolimits}
\def\Ker{\mathop{\mathrm{Ker}}\nolimits}

\def\Hom{\mathop{\mathrm{Hom}}\nolimits}

\def\ann{\mathop{\mathrm{ann}}\nolimits}
\def\Fr{\text{\rm Fr}}

\def\Het[#1](#2){H_{\text{\'et}}^{#1}(#2)}

\newcommand{\mf}[1]{{\mathfrak{#1}}}

\newcommand{\bb}[1]{{\mathbb{#1}}}
\newcommand{\mca}[1]{{\mathcal{#1}}}

\begin{document}

\begin{abstract}
In our previous work, 
by using Kolyvagin derivatives 
of elliptic units,
we constructed ideals 
$\mf{C}_{i}^{\mathrm{ell}}$
of the Iwasawa algebra, 
and proved that 
the ideals
$\mf{C}_{i}^{\mathrm{ell}}$
become ``upper bounds" of
the higher Fitting ideals of
the one and two variable 
$p$-adic unramified Iwasawa module $X$ 
over an abelian extension field $K_0$ of 
an imaginary quadratic field $K$.
In this article, 
by using ``non-arithmetic"
specialization arguments, 
we prove that 
the ideals
$\mf{C}_{i}^{\mathrm{ell}}$
also become ``lower bounds" of
the higher Fitting ideals of $X$.
In particular, we show that
the ideals
$\mf{C}_{i}^{\mathrm{ell}}$ determine 
the pseudo-isomorphism class of $X$.
Note that in this article, we also treat
the cases when the $p$-part of 
the equivariant Tamagawa number conjecture 
$(\mathrm{ETNC})_p$
is not proved yet.

In the cases when 
$(\mathrm{ETNC})_p$ is proved,
stronger results have already been obtained by 
Burns, Kurihara and Sano:  
under the assumption of $(\mathrm{ETNC})_p$ 
and certain conditions on 
the character $\psi$ on $\Gal(K_0/K)$, 
they have given a complete description of 
the higher Fitting ideals of the $\psi$-component of $X$
by using Rubin--Stark elements.
In our article, 
we also prove that the $\psi$-part of 
$\mf{C}_{i}^{\mathrm{ell}}$
coincide with  the ideals constructed by 
Burns, Kurihara and Sano 
in certain cases when $(\mathrm{ETNC})_p$ is proved.
As a corollary of this comparison results,
we also deduce that the annihilator ideal of
the $\psi$-part of 
the maximal pseudo-null submodule of $X$
coincides with the initial Fitting ideal
in certain situations.
\end{abstract}

\maketitle


\section{Introduction}\label{secintro}

Let $K$ be an imaginary quadratic field, 
and $K_0/K$ a finite abelian extension. 
We put $\Delta:=\Gal(K_0/K)$. 
For each prime ideal $\mf{l}$ of $K$, 
we denote by $D_{\Delta, \mf{l}}$
the decomposition group of $\Delta$
at $\mf{l}$.
We fix a prime number $p$
not dividing
$ \# \Delta \cdot 
\# (\mca{O}_{K_0}^\times)_{\mathrm{tor}}$.
Let $K_\infty/K$ be an abelian extension 
such that $K_\infty$ contains $K_0$, 
and the Galois group 
$\Gamma:=\Gal(K_\infty/K_0)$
is isomorphic to $\bb{Z}_p$ or $\bb{Z}_p^2$.
We put $\mca{G}:=\Gal(K_\infty / K) 
=\Delta \times \Gamma$, and 
define the completed group ring
$\widetilde{\Lambda}:=\bb{Z}_p[[\mca{G}]]$.

We put $\widehat{\Delta}:=\Hom(\Delta,\overline{\bb{Q}}_p^\times)$.
For any $\psi \in \widehat{\Delta}$, 
we define $\mca{O}_{\psi}:=\bb{Z}_p[\mathrm{Im} \psi]$
to be the $\bb{Z}_p[\Delta]$-algebra 
where $\Delta$ acts via the character $\psi$.
We put $\Lambda_\psi :=\mca{O}_\psi[[\Gamma]]$. 
Since $p \nmid \# \Delta$, 
we have the direct product decomposition
\[
\widetilde{\Lambda}=(\bb{Z}_p[[\Gamma]])[\Delta]
\simeq 
\prod_{G_{\bb{Q}_p}\psi \in G_{\bb{Q}_p} \backslash 
\widehat{\Delta}}
\Lambda_\psi,
\] 
where the group 
$G_{\bb{Q}_p}:=\Gal(\overline{\bb{Q}}_p/\bb{Q}_p)$ 
acts on 
$\widehat{\Delta}$ by 
$\sigma\psi:=\sigma \circ \psi$
for any $\sigma \in G_{\bb{Q}_p}$ 
and $\psi \in \widehat{\Delta}$.
For any $\Lambda$-module $M$, 
we define $M_\psi:=M \otimes_{\Lambda} \Lambda_\psi$.

For any number field $F$, we denote
the ideal class group of $\mca{O}_F$ 
by $\mathrm{Cl}(F)$, and put 
$A(F):=\mathrm{Cl}(F) \otimes_{\bb{Z}} \bb{Z}_p$.
Let $\mca{IF}$ be the set of all intermediate fields $F$ 
of $K_\infty^{\Delta}/K$ satisfying 
$[F:K]< \infty$.
We define the unramified Iwasawa module $X$ by
\(
X:=\varprojlim_{F \in \mca{IF}} A(K_0F)
\).
Note that $X$ is 
a finitely generated torsion $\Lambda$-module.

In this article, we study
the pseudo-isomorphism class of 
the $\Lambda_\psi$-module $X_\psi$.
The pseudo-isomorphism class of
a finitely generated torsion 
$\Lambda_\psi$-module
is determined by the higher Fitting ideals.
In our paper, by using Kolyvagin derivatives 
of elliptic units,
we study the higher Fitting ideals
$\{ \Fitt_{\Lambda_\psi,i}(X_\psi) 
\}_{i \in \bb{Z}_{\ge 0}}$
of the $\Lambda_\psi$-module $X_\psi$.
Note that in this article, we also treat
the cases when the $p$-part of 
the equivariant
Tamagawa number conjecture $(\mathrm{ETNC})_p$
is not proved yet, 
namely 
the cases when $p$ does not split in $K$, 
or the cases when $p$ divides the class number of $K$.
(As we shall see later, 
under assumption of $(\mathrm{ETNC})_p$, 
Burns, Kurihara and Sano 
have already given the description of 
$\Fitt_{\Lambda_\psi,i}(X_\psi)$
by using Rubin--Stark elements.)

In order to state the assertion of 
our results, let us introduce 
the following notation.
Let $I$ and $J$ be ideals of $\Lambda$.
We write $I \prec J$ if and only if 
there exists an ideal $\mca{A}$ of $\Lambda$
of height at least two 
satisfying $\mca{A} I \subseteq J$.
We write $I \sim J$ if and only if 
$I \prec J$ and $J \prec I$.
Note that by the structure theorem, 
the equivalent classes 
of the higher Fitting ideals 
with respect to the relation $\sim$ determine 
the pseudo-isomorphism class of
a finitely generated torsion 
$\Lambda_\psi$-module. 
(For instance, see \cite{Oh2} Remark 2.4.)

In our previous work \cite{Oh1}, 
by using the Kolyvagin derivatives 
of the Euler system of elliptic units,
we defined the ideals $\mf{C}_{i,\psi}^{\mathrm{ell}}$
for each $i \in \bb{Z}_{\ge 0}$.
(For details, see \cite{Oh1} Definition 4.5.
See also Definition \ref{theideal2} 
and Remark \ref{remCiell} in our article.)
By using the Euler system arguments 
developed by Kurihara in \cite{Ku} and \cite{Ku2},
we proved that
\begin{equation}\label{eqoh1thm}
\Fitt_{\Lambda_\psi, i}(X_\psi)
\prec 
\mf{C}_{i,\psi}^{\mathrm{ell}}
\end{equation}
for any $i \in \bb{Z}_{\ge 0}$
under certain assumptions.
(For details, see \cite{Oh1} Theorem 1.1.)
In this article, we shall prove 
the opposite inequality of (\ref{eqoh1thm}).

\begin{thm}\label{thmmainthm}
Let $\psi \in \widehat{\Delta}$ 
be a character
satisfying $\psi \vert_{D_{\Delta,\mf{p}}} \ne 1$
for any prime $\mf{p}$ of $\mca{O}_K$ above $p$,
If $K_0$ contains $\mu_p$, 
we also assume that $\psi 
\ne \omega\psi^{-1}$
and $\psi \vert_{D_{\Delta,\mf{p}}} 
\ne \omega \vert_{{D_{\Delta,\mf{p}}}} $
for any prime $\mf{p}$ of $\mca{O}_K$ above $p$,
where 
\(
\omega \colon 
\Delta \longrightarrow \Gal(K(\mu_p)/K)
\longrightarrow \bb{Z}_p^\times
\) 
denotes the Teichm\"uller character. 
Then, for any $i \in \bb{Z}_{\ge 0}$, 
we have
\[
\mf{C}_{i,\psi}^{\mathrm{ell}} \prec
\Fitt_{\Lambda_\psi, i}(X_\psi).
\]
\end{thm}

Note that 
in the setting of 
Theorem \ref{thmmainthm}, 
the inequality 
(\ref{eqoh1thm})
also holds
since 
under the assumption that
$\psi \vert_{D_{\Delta,\mf{p}}} 
\ne 1 $
for any $\mf{p} \mid p$, 
the ideal $\mca{I}_{T,\psi}$ appearing in 
\cite{Oh1} Theorem 1.1 
is equal to $\Lambda_\psi$.
Hence by combining with Theorem \ref{thmmainthm}
in our article, 
we obtain the following corollary.

\begin{cor}\label{thmpi}
Let $\psi \in \widehat{\Delta}$ 
be as in Theorem \ref{thmmainthm}.
Then, for any $i \in \bb{Z}_{\ge 0}$, we have 
\(
\Fitt_{\Lambda_\psi, i}(X_\psi)
\sim 
\mf{C}_{i,\psi}^{\mathrm{ell}}.
\)
In particular, 
the pseudo-isomorphism class of 
the $\Lambda_\psi$-module
$X_\psi$ is determined by the collection  
$\{
\mf{C}_{i,\psi}^{\mathrm{ell}}
\}_{i \ge 0}$.
\end{cor}

The strategy of the proof of 
Theorem \ref{thmmainthm}
is as follows. 
\begin{enumerate}[(I)]
\item First, by using standard arguments of Euler systems
as in \cite{Ru1} \S 4,
we shall prove Proposition \ref{propgrlev}, 
which is the ``non-variable" version of 
Theorem \ref{thmmainthm}
for general one dimensional Galois representations.  
\item As in \cite{Oh2},
by using the arguments 
in \cite{MR} \S 5.3,
we shall reduce the proof 
of Theorem \ref{thmmainthm}
in the one variable cases, 
namely the cases when $\Gamma \simeq \bb{Z}_p$,
for general one dimensional Galois representations
to the results on the non-variable cases.
(See \S \ref{ssonevar}.)
\item Finally, by specializing at 
a ``good" height one prime, 
we shall reduce the proof of
Theorem \ref{thmmainthm}
for the two variable cases
to that for one variable cases.
Note that in this step, 
use a lemma close to \cite{Oc} Lemma 3.5 
in order to choose a good hight one primes.
(See \S \ref{sstwovar}.)
\end{enumerate}
In the steps (II) and (III), 
we may use ``non-arithmetic" specializations
in the following sense: 
the module obtained after 
the specialization  is not  
``the dual fine Selmer group" of 
(a Galois deformation
containing) a Galois representation coming from a motive.
So, in the steps (I) and (II), 
we need to work in the setting of 
general one dimensional Galois representations.

Note that a certain ``weak specialization compatibility" 
of $\mf{C}_{i,\psi}^{\mathrm{ell}}$
(Lemma \ref{lemredCi} and Lemma \ref{lemredCitodvr})
is a key of the proof of Theorem \ref{thmmainthm}.
By using weak specialization compatibility and 
the class number formula 
in terms of elliptic units, 
we can also obtain the following Theorem \ref{thm+P}, 
which implies that 
the ideal $\mf{C}_{i,\psi}^{\mathrm{ell}}$
is ``close" to $\Fitt_{\Lambda_\psi, i}(X_\psi)$
in the different sense form the relation $\sim$.
(For the proof of Theorem \ref{thm+P}, 
see \S \ref{ssthmP}.)

\begin{thm}\label{thm+P}
We denote by
$I(\Gamma)$ the augmentation ideal
of $\Lambda_\psi=\mca{O}_\psi[[\Gamma]]$.
Let $\psi \in \widehat{\Delta}$ be as in  
Theorem \ref{thmmainthm}.
Then, for any $i \in \bb{Z}_{\ge 0}$, 
we have 
\[
\Fitt_{\Lambda_\psi, i}(X_\psi)+ I(\Gamma)
= \mf{C}_{i,\psi}^{\mathrm{ell}}+ I(\Gamma).
\]
\end{thm}

By Theorem \ref{thm+P}, we immediately obtain 
the following corollary, 
which says that 
the ideals $\mf{C}_{i,\psi}^{\mathrm{ell}}$
determines the cardinality of 
a minimal system of generators of $X_\psi$.

\begin{cor}
Let $\psi \in \widehat{\Delta}$ be as in  
Theorem \ref{thmmainthm}.
Let $r \in \bb{Z}_{\ge 0}$. 
We have $\Fitt_{\Lambda_\psi, r}=\Lambda_\psi$
if and only if $\mf{C}_{r,\psi}^{\mathrm{ell}}=\Lambda_\psi$.
In particular, the cardinality of 
a minimal system of generators of 
the $\Lambda_\psi$-module $X_\psi$ is equal to
the minimum of
\(
\{ r \in \bb{Z}_{\ge 0} \mathrel{\vert}
\mf{C}_{r,\psi}^{\mathrm{ell}}=\Lambda_\psi\}.
\)
\end{cor}

In the cases when  
$(\mathrm{ETNC})_p$ holds, 
stronger assertions were proved 
by Burns, Kurihara and Sano. 
By using (higher order) Rubin--Stark elements, 
they constructed ideals which coincides 
with the higher Fitting ideals 
of (a quotient of) 
ray class group. (See \cite{BKS} Corollary 1.7.)
Bley proved $(\mathrm{ETNC})_p$
over imaginary quadratic fields
in the case when $p$ splits completely in $K/\bb{Q}$,
and does not divide the class number of $K$.
(The case when $p$ does not split, 
equivariant Tamagawa number conjecture
is still open.)
In this paper, 
when $p$ splits completely in $K/\bb{Q}$,
we compare our ideals 
$\mf{C}_{i,\psi}^{\mathrm{ell}}$
Burns--Kurihara--Sano's ideals
$\Theta^{\mathrm{RS}}_{S,T,i}(M/K)_\psi$
(written in our notation, 
see Definition \ref{defTheta}),
and proved that 
our ideals 
coincides with 
the ideal 
$\Theta^{\mathrm{RS}}_{S,T,i}(M/K)_\psi$.
Roughly speaking, the assertion of 
the comparison result is as follows.
(For the precise statement, see 
Theorem \ref{thmTheta=C}.)

\begin{thm}\label{thmTheta=Crough}
Suppose that $p$ splits completely in $K/\bb{Q}$, 
and $p$ is prime to $\# \mathrm{Cl}(K)$.
Let $F \in \mca{IF}$  be any element.
We take  a character $\psi$ and
finite sets $S,T$ of places of $K$ 
suitably.
Then, for any $i \in \bb{Z}_{\ge 0}$, we have
\[
\mf{C}^{\mathrm{ell}}_{i}(F/K)_\psi
=\Theta^{\mathrm{RS}}_{S,T,i}(K_0F/K)_\psi.
\]
\end{thm}

Here, 
$\mf{C}^{\mathrm{ell}}_{i}(F/K)_\psi$
is the finite layer version of 
$\mf{C}_{i,\psi}^{\mathrm{ell}}$.
(See Definition \ref{theideal2}.)
By this comparison, 
the following corollary immediately
follows from 
\cite{BKS} Corollary 1.7.

\begin{cor}\label{corFitt=}
Suppose that $p$ splits completely in $K/\bb{Q}$, 
and $p$ is prime to $\# \mathrm{Cl}(K)$.
Let $\psi \in \widehat{\Delta}$
be a non-trivial faithful character 
such that $\psi \vert_{D_{\Delta,v}} \ne 1$
for any finite place $v$ in $S$.
Moreover we assume that if $K_0$ contains $\mu_p$, 
then $\psi \ne  \omega$.
We have 
\[
\Fitt_{\Lambda_\psi, i}(X_\psi)
= \mf{C}_{i,\psi}^{\mathrm{ell}}
\]
for any $i \in \bb{Z}_{\ge 0}$.
\end{cor}

By combining Theorem \ref{thmTheta=C} 
with our previous result 
(\cite{Oh1} Theorem 1.1),
we obtain a results on 
the structure of the pseudo-null part of
the unramified Iwasawa module
in one-variable cases.
Keep the notation and assumptions as in 
Corollary \ref{corFitt=}. 
Furthermore, we assume that $\Gamma \simeq \bb{Z}_p$.
Let $X_{\mathrm{fin},\psi}$ be the maximal pseudo-null
$\Lambda_\psi$-submodule of  $X_\psi$.
In \cite{Oh1}, we proved that
\[
\ann_{\Lambda_\psi}(X_{\mathrm{fin},\psi})
\Fitt_{\Lambda_\psi, 0}(X_\psi
/X_{\mathrm{fin},\psi})
\subseteq \mf{C}_{0,\psi}^{\mathrm{ell}}.
\]
Since we have
\begin{align*}
\Fitt_{\Lambda_\psi, 0}(X_\psi)
&=\Fitt_{\Lambda_\psi,0}(X_{\mathrm{fin},\psi})
\Fitt_{\Lambda_\psi, 0}(X_\psi
/X_{\mathrm{fin},\psi}) \\
&\subseteq 
\ann_{\Lambda_\psi}(X_{\mathrm{fin},\psi})
\Fitt_{\Lambda_\psi, 0}(X_\psi
/X_{\mathrm{fin},\psi})
\end{align*}
Theorem \ref{thmTheta=C} 
implies the following corollary,
which constrains
the structure of 
the pseudo-null part
$X_{\mathrm{fin},\psi}$.

\begin{cor}\label{corXfin}
Let $(K_0/K,p,\psi)$ be as in Corollary \ref{corFittRS}.
We assume that $\Gamma \simeq \bb{Z}_p$.
Then, we have 
$\Fitt_{\Lambda_\psi, 0}(X_{\mathrm{fin},\psi})
=\ann_{\Lambda_\psi}(X_{\mathrm{fin},\psi})$.
\end{cor}

Since the philosophy of construction 
is very close, 
the proof of Theorem \ref{thmTheta=C}
is given by ``translation" of generators.
The keys of the proof of the comparison results
are Mazur--Rubin--Sano conjecture, 
which is valid 
if $(\mathrm{ETNC})_p$ holds, 
and the explicit description of Stark units.

\begin{rem}
Let $L/\bb{Q}$ be 
any totally real finite abelian extension.
Then, the Stark units of $L$ are given 
by circular units, and 
$(\mathrm{ETNC})_p$ for $L/\bb{Q}$
was proved by Burns and Greither \cite{BG}.
(Moreover, the $2$-part 
$(\mathrm{ETNC})_2$ for $L/\bb{Q}$ 
was proved by Flach \cite{Fl}.)
The arguments in the proof of
Theorem \ref{thmTheta=C} work
when we have an Euler system consisting of 
Stark units, and when $(\mathrm{ETNC})_p$ holds.
So, in particular, similar comparison results 
to Theorem \ref{thmTheta=C}
hold for $L/\bb{Q}$. 
Namely, under suitable conditions, 
the higher cyclotomic ideals
$\mf{C}_{i,\chi}$ defined in \cite{Oh2}
coincide with (the limit of) 
Burns--Kurihara--Sano's ideals.
\end{rem}

Let us see the contents of our article.
In \S \ref{secSelIw}, 
we introduce Selmer groups 
and Iwasawa modules which we study.
In \S \ref{secES}, we recall the definition of 
Euler systems and their Kolyvagin derivatives. 
In \S \ref{secCi},
we construct certain ideals $\mf{C}_i(Z)$ of
the Iwasawa algebra  
by using Kolyvagin derivatives, 
and show some basic properties of $\mf{C}_i(Z)$
including the specialization compatibility.
In \S \ref{secpgl}, we address the step (I) 
in the above strategy, and prove Theorem \ref{thm+P}.
In \S \ref{secpf},
we deal with the steps (II) and (III), 
and completes the proof of 
Theorem \ref{thmmainthm}.
In \S \ref{seccomparison}, 
we compare
our ideals with 
Burns--Kurihara--Sano's ideals.

\subsection*{Notation}

Throughout this paper, we fix
$K_0/K$ and $K_\infty /K$ be the extensions of 
an imaginary quadratic field 
introduced in \S \ref{secintro},
and let define $\mca{G}=\Delta \times \Gamma$
similarly.
We fix an algebraic closure $\overline{\bb{Q}}=
\overline{K}$ of $\bb{Q}$ (and $K$).
In this article,
an algebraic number field 
is an intermediate field of 
$\overline{\bb{Q}}/\bb{Q}$
which is finite over $\bb{Q}$.
For any field $F$, 
we denote by $G_F$ 
the absolute Galois group $F$.

We fix  embeddings
$\overline{\bb{Q}} \hookrightarrow \bb{C}$
and $\overline{\bb{Q}} \hookrightarrow 
\overline{\bb{Q}}_p$, 
and we regard 
$\overline{\bb{Q}} \subseteq \bb{C}$
and $\overline{\bb{Q}} \subseteq 
\overline{\bb{Q}}_p$
via the fixed embedding.

For any number field $F$, we denote
by $P(F)$ the set of all places of $F$.
Let $L/F$ be a finite extension of number fields, and 
$\mf{n}$ a non-zero ideal of $\mca{O}_F$. 
We denote $P(L)^{\mf{n}}$ be the set of all places of $F$
prime to $\mf{n}$, and $\mathrm{Prime}_L(\mf{n})$
be the set of all prime ideals of $\mca{O}_L$
dividing $\mf{n}$.
For simplicity, we write 
$\mathrm{Prime}(\mf{n}):=
\mathrm{Prime}_F(\mf{n})$.

Let $F$ be an algebraic number field, 
and $L/F$ a Galois extension. 
Let $\lambda$ be a prime ideal of $\mca{O}_F$, 
and $\lambda'$ a prime ideal of $L$ above $\lambda$. 
We denote the completion of $F$ at $\lambda$ by $F_{\lambda}$. 
If $\lambda$ is unramified in $L/F$,  
the arithmetic Frobenius 
at $\lambda'$ is denoted by 
$(\lambda', L/K) \in \Gal(L/K)$. 
We fix a collection of embeddings 
$\{\mf{l}_{\overline{K}}\colon\overline{K}\hookrightarrow 
\overline{K}_\mf{l} \}_{\mf{l}: \rm{prime}}$ 
satisfying the condition (Chb) as follows:
\begin{itemize}
\item[(Chb)] {\em For any subfield $F \subset \overline{K}$ 
which is a finite Galois extension of $K$ 
and any element $\sigma \in \Gal(F/K)$, 
there exist infinitely many primes $\mf{l}$ 
such that $\ell$ is unramified in $F/K$ 
and $(\mf{l}_F, F/K)= \sigma$, 
where $\mf{l}_F$ is the prime ideal of $\mca{O}_F$ 
corresponding to the  embedding $\mf{l}_{\overline{K}}\vert_F$.}
\end{itemize}
We can easily show 
the existence of the collection satisfying the condition (Chb) 
by using the Chebotarev density theorem.

Let $R$ be a commutative ring
and $\rho\colon G_K \longrightarrow \mca{O}^\times$ 
a continuous character.
We denote by $R(\rho)$ a free $R$-module  
of rank one with a fixed basis $e_{R,\rho}$
equipped with a continuous action of $G_K$
via $\rho$.
(We often write $e_\rho :=e_{R,\rho}$
for simplicity.)
For any $R$-algebra $S$, 
we identify $S \otimes_R R(\rho)$
with $S(\rho)$ by the isomorphism
defined by $1 \otimes e_{R,\rho} \mapsto e_{S,\rho}$.
We  define a $G_{K}$-equivariant pairing
\[
R(\rho) \times R(\rho^{-1})
\longrightarrow R;\ 
(ae_{\rho},be_{\rho^{-1}})\longmapsto ab.
\]
By this pairing, we identify 
the $R[G_{K}]$-module 
$\mca{O}(\rho)$ with 
the contragredient $R[G_{K}]$-module
$\Hom_{\mca{O}}(\mca{O}(\rho), \mca{O})$.
For two $R^\times$-valued characters
$\rho_1$ and $\rho_2$ of $G_K$,
we identify 
$R(\rho_1)\otimes R(\rho_2)$
with $R(\rho_1\rho_2)$
via the $R[G_K]$-isomorphism sending 
$e_{\rho_1}\otimes e_{\rho_2}$
to $e_{\rho_1\rho_2}$.

For each $N \in \bb{Z}_{>0}$, 
we denote the group of $N$-th roots of unity
contained in $\overline{\bb{Q}}$ by
$\mu_N:=\boldsymbol{\mu}_N(\overline{\bb{Q}})$, 
and put $\zeta_N:=e^{2 \pi i/N} \in 
\overline{\bb{Q}} \subseteq \bb{C}$.
(Recall that we have fixed an embedding
$\overline{\bb{Q}} \hookrightarrow \bb{C}$.)
Then, we obtain a basis 
$(\zeta_{p^{n}})_n \in \varprojlim_n \mu_{p^n}$. 
We denote by $\chi_{\mathrm{cyc}}\colon 
G_K \longrightarrow \bb{Z}_p^\times$
the cyclotomic character.
We identify the $\bb{Z}_p[[G_K]]$-module
$\varprojlim_n \mu_{p^n}$
with $\bb{Z}_p(\chi_{\mathrm{cyc}})$
via the isomorphism defined by 
$(\zeta_p^{n})_n \mapsto e_{\chi_{\mathrm{cyc}}}$.

Let $M,N$ be $R[G_K]$-modules, and 
$f \colon M \longrightarrow N$ 
an $R[G_K]$-linear map.
We write $M\otimes \rho:=M\otimes_{R} R(\rho)$, 
and $m\otimes \rho:=m \otimes e_\rho$
for any  $m \in M$.
We define a homomorphism 
\(
f \otimes \rho \colon 
M \otimes \rho \longrightarrow N \otimes \rho 
\)
by
$m\otimes \rho \mapsto f(m)\otimes \rho$
for any $m \in M$.

Let $R$ be a commutative ring, and  $M$ an $R$-module.
For any $a \in R$, let $M[a]$ be the $R$-submodule of $M$ consisting of 
all $a$-torsion elements.
For each $x \in M$, the annihilator ideal of $x$
is denoted by $\ann_R(x;M)$. 
We denote the ideal of $R$ consisting of all annihilators of $M$
by $\ann_R(M)$. 
Let $G$ be a group, and suppose 
that $G$ acts on the $R$-module $M$.
Then, we denote by $M^G$ the maximal subgroup of $M$ 
fixed by the action of $G$.

\section*{Acknowledgment}
It is a pleasure to thank 
Takamichi Sano for some
useful information on his works, 
as well as 
Takashi Taniguchi 
for his encouragement.
This work is supported by 
JSPS KAKENHI Grant Number 26800011.

\section{Selmer groups and Iwasawa module}\label{secSelIw}

Here, we set the notations of
Selmer groups and Iwasawa modules 
which we study.
Throughout this section, 
we use the same notations
as in the previous section.
For instance, we fix
$K_0/K$ and $K_\infty /K$ be the extensions of 
an imaginary quadratic field 
introduced in \S \ref{secintro},
and put $\mca{G}:=\Gal(K_\infty /K)=\Delta \times \Gamma$.
We denote by $\mca{IF}$ the set of all intermediate fields $F$ 
of $K_\infty^{\Delta}/K$ satisfying 
$[F:K]< \infty$.
We fix a finite extension field $\mca{L}$ of $\bb{Q}_p$,
and define $\mca{O}$ to be the ring of integers in $\mca{L}$.
We fix a uniformizer $\pi \in \mca{O}$, 
and put $k:=\mca{O} /\pi \mca{O}$. 
We define $\Lambda:=\mca{O}[[\Gamma]]$.

In \S \ref{sssel}, we set our notation of Selmer group, 
and in \S \ref{ssImod}, we introduce the Iwasawa modules 
arising from our Selmer groups.

\subsection{Selmer groups}\label{sssel}
Here, we recall the notation of 
Selmer groups, and review
some basic properties briefly.

Let $\rho \colon G_{K} 
\longrightarrow \mca{O}^\times$
be a continuous character.
Let $\Sigma(\rho)$ be the set of primes of $\mca{O}_K$
consisting of all primes above $p$
and all primes above where the character $\rho$ is 
ramified.
We denote by $\bar{\rho} \colon G_{K} 
\longrightarrow k^\times$
the modulo-$\pi$ reduction of $\rho$.
Let $\bar{\chi}_{\mathrm{cyc}} \colon G_{K}
\longrightarrow k^\times$ is 
be the modulo-$\pi$ cyclotomic character.
We assume the following hypotheses:
\begin{itemize}
\item[(C1)] We have $\bar{\rho} \ne \bar{\chi}_{\mathrm{cyc}}\bar{\rho}^{-1}$
as characters on $G_{K}$.
\item[(C2)] For any prime  $\mf{p}$ of $\mca{O}_K$
dividing $p\mca{O}_K$,
we have $\bar{\rho} \vert_{G_{K_{\mf{p}}}} \ne 1$
and $\bar{\rho} \vert_{G_{K_{\mf{p}}}} 
\ne \bar{\chi}_{\mathrm{cyc}} \vert_{G_{K_{\mf{p}}}}$. 
\item[(C3)] If $\bar{\rho}$ is unramified at a place $v \in P_K$,
then $\rho$ is unramified at $v$.
\end{itemize}

Here, we denote by $W$ the $\mca{O}[G_{K}]$-module 
$(\mca{L}/\mca{O})\otimes \chi_{\mathrm{cyc}}\psi^{-1}$
or $(\mca{O}/\pi^N \mca{O})\otimes \chi_{\mathrm{cyc}}\psi^{-1}$
for some $N \in \bb{Z}_{>0}$.

\begin{dfn}
Let $F$ be any algebraic number field containing $K$, and
$\mf{n}$ a non-zero ideal of $\mca{O}_F$.
We define 
\begin{align*}
\mathrm{Sel}^{\mf{n}}(F,W)
:&=\Ker \left(
H^1(F,W) \longrightarrow 
\prod_{v \in P(F)^{\mf{n}}} 
\frac{H^1(F_v,W)}{H^1_f(F_v,W)}
\right), \\
\mathrm{Sel}_{\mf{n}}(F,W)
:&=\Ker \left(
\mathrm{Sel}^{\mf{n}}(F,W)
\longrightarrow 
\prod_{w \in \mathrm{Prime}_F(\mf{n})} 
H^1(F_w,W)
\right)
\end{align*}
where $H^1_f(F_v,W)$ is the finite local condition
on $H^1(F_v,W)$ in the sense of Bloch--Kato \cite{BK} \S 3.
We put 
\begin{align*}
X(F,\rho)_{\mca{O}}:&=
\Hom_{\mca{O}}(\mathrm{Sel}_{p\mca{O}_K}
(F,(\mca{L}/\mca{O})\otimes \rho),
\mca{L}/\mca{O}) \\
X_{N}(F,\rho)_{\mca{O}}:&=
\Hom_{\mca{O}}(\mathrm{Sel}_{p\mca{O}_K}(F,(\mca{O}/\pi^N \mca{O})
\otimes \rho),\mca{O}/\pi^N\mca{O}).
\end{align*}
If no confusion arises, we write 
$X(F,\rho):=X(F,\rho)_{\mca{O}}$ and 
$X_N(F,\rho):=X_N(F,\rho)_{\mca{O}}$.
\end{dfn}

Let $F \in \mca{IF}$. 
For each $N_1,N_2, \in \bb{Z}_{>0}$ 
with $N_1<N_2$, we define
\[
\nu_{N_1,N_2}\colon 
\mca{O}/\pi^{N_1}\mca{O}
\longrightarrow \mca{O}/\pi^{N_2}\mca{O};
a \mapsto a\pi^{N_2-N_1}.
\]
Since $[F:K]$ is a power of $p$, 
the assumption (C2) implies that 
$H^0(K,k\otimes \rho)=0$.
So, the map $\nu_{N_1,N_2}$ induces an injection
\[
H^1(\nu_{N_1,N_2})\colon 
H^1(F,(\mca{O}/\pi^{N_1} \mca{O})
\otimes \rho) \hookrightarrow 
H^1(F,(\mca{O}/\pi^{N_2} \mca{O})
\otimes \rho).
\]

\begin{lem}\label{lemSelsurj}\label{lemSellim}
Let $F \in \mca{IF}$ be any element, and 
$\mf{n}$ a non-zero ideal of $\mca{O}_K$.
\begin{enumerate}[{\rm (i)}]
\item Let $N_1,N_2, \in \bb{Z}_{>0}$ be 
positive integers satisfying $N_1<N_2$.
The restriction of the injection $H^1(\nu_{N_1,N_2})$ 
induces an isomorphism
\[
\mathrm{Sel}^{\mf{n}}(F,(\mca{O}/\pi^{N_1} \mca{O})
\otimes \rho) \xrightarrow{\ \simeq \ } 
\mathrm{Sel}^{\mf{n}}(F,(\mca{O}/\pi^{N_2} \mca{O})
\otimes \rho)[\pi^{N_1}].
\]
{\rm (See \cite{Ru5}  Lemma 1.5.4.)}
\item We have natural isomorphisms
\begin{align*}
\varinjlim_{N} \mathrm{Sel}^{\mf{n}}(F,(\mca{O}/\pi^{N} \mca{O})
\otimes \rho) & \simeq 
\mathrm{Sel}^{\mf{n}}(F,(\mca{L}/\mca{O})
\otimes \rho),  \\
\varinjlim_{N} \mathrm{Sel}_{\mf{n}}(F,(\mca{O}/\pi^{N} \mca{O})
\otimes \rho) & \simeq 
\mathrm{Sel}_{\mf{n}}(F,(\mca{L}/\mca{O})
\otimes \rho),
\end{align*}
\end{enumerate}
where the inductive limits are taken with respect to 
the systems given by the restriction of
the injections $H^1(\nu_{N_1,N_2})$.
{\rm (See \cite{Ru5} Lemma 1.3.6.)}
\end{lem}

\begin{lem}\label{lemH2X}
For any $F \in \mca{IF}$, 
we have 
\(
X(F,\rho)
\simeq 
H^2(K_{\Sigma(\rho)}/ F,\mca{O}(\rho)).
\)
\end{lem}

\begin{proof}
Let $\mf{l} \in {\Sigma(\rho)}$ be any element
prime to $p$, 
and $w$ a place of $F$ above $\mf{l}$. 
Note that the prime $\mf{l}$ is ramified in $K_0/K$.
By the assumption (C3),
we have 
$H^0(F_w^{\mathrm{ur}},
(\mca{L}/\mca{O}) \otimes \rho)=0$.
So, by \cite{Ru5} Lemma 1.3.5 (iii),
we obtain
$H^1_f(F_w,\mca{O}(\rho))=0$.
Moreover, 
if $w' \in P_F$ is a place above $p$, 
or if $w' \in P_F$ is a place where 
$\rho$ is ramified, 
then it follows from the local duality and 
the assumptions (C2) and (C3) 
on $\rho$ that we have
$H^2(F_{w'} ,\mca{O}(\rho) )=0$.
Hence by the Poitou-Tate exact sequence, 
we obtain the isomorphism as desired.
\end{proof}

\begin{exa}
Let $\psi \in \widehat{\Delta}$
be a character as in Theorem \ref{thmmainthm}.
We set $\mca{O}_\psi:=\bb{Z}_p[\Im \psi]$.
Then, 
the character 
$\chi_{\mathrm{cyc}}\psi^{-1} \colon 
G_{K} \longrightarrow \mca{O}^\times$
satisfies the conditions (C1)--(C3).
Let $F \in \mca{IF}$ be any element. 
Namely, let $F$ be an intermediate field of 
$K_\infty^{\Delta} /K$ satisfying $[F:K]< \infty$.
By the global class field theory and
the Poitou--Tate exact sequence, we  have
\(
H^2(F_{\Sigma(\rho)}/F,\mca{O}(\chi_{\mathrm{cyc}}\psi^{-1}))
\simeq X(\chi_{\mathrm{cyc}}\psi^{-1})
\simeq A_{F,\psi}
\).
\end{exa}

\subsection{Iwasawa modules}\label{ssImod}

Now, let us recall some basic facts 
in the Iwasawa theoretical setting.
First, let us define 
the Iwasawa module which we mainly study.

\begin{dfn}
We define 
\[
X(K_\infty^\Delta /K,\rho)_{\mca{O}}:=
\Hom_\mca{O}\left(
\varinjlim_{F \in \mca{IF}}
\mathrm{Sel}_{p\mca{O}_K}
(F,(\mca{L}/\mca{O})\otimes \rho),\mca{L}/\mca{O}\right).
\]
Note that $X(K_\infty^\Delta /K,\rho)$
is a finitely generated $\Lambda$-module.
If no confusion arises, we write 
$X(K_\infty^\Delta /K,\rho):=
X(K_\infty^\Delta /K,\rho)_{\mca{O}}$
for simplicity.
\end{dfn}

Let $\rho' \colon \Gamma
\longrightarrow 1 +\pi \mca{O} 
\subseteq \mca{O}^\times$
be a continuous character.
Since the extension $K_\infty^{\Delta} /K$
is unramified outside $p$, 
we have $\Sigma(\rho)=\Sigma(\rho\rho')$. 
Note that the assumptions (C1)--(C3) on the character $\rho$
imply that the character
$\rho\rho' \colon G_{K}
\longrightarrow \mca{O}^\times$
also satisfies (C1)--(C3).
The following lemma plays
an important role 
in the reduction arguments.

\begin{lem}\label{lemXred}
Let $\rho' \colon \Gamma
\longrightarrow 1 +\pi \mca{O}$
be a continuous character.
\begin{enumerate}[{\rm (i)}]
\item Suppose that $\Gamma \simeq \bb{Z}_p$. 
Let $\gamma$ be a topological generator of $\Gamma$.
Then, we have a natural $\mca{O}$-isomorphism
\[
X(K_\infty^\Delta /K,\rho) \otimes_{\Lambda} 
\Lambda/(\gamma - \rho'(\gamma))
\simeq X(K,\rho\rho').
\]
\item Suppose that $\Gamma \simeq \bb{Z}_p^2$, 
and $\gamma_1,\gamma_2$ are topological generators of $\Gamma$
with $\rho'(\gamma_2)=1$.
For each $i \in \{ 1,2 \}$, 
we denote by $\Gamma_i$ the closed subgroup of $\Gamma$
topologically generated by $\gamma_i$.
We put 
\(
\overline{\Lambda}:=
\Lambda/(\gamma_1 - \rho'(\gamma_1))
\simeq \mca{O}[[\Gamma_2]].
\)
Then, we have a natural 
$\overline{\Lambda}$-isomorphism
\[
X(K_\infty^{\Delta} /K,\rho) \otimes_{\Lambda} 
\overline{\Lambda}
\simeq X(K_\infty^{\Delta \times \Gamma_1} 
 /K,\rho\rho').
\]
\end{enumerate}
\end{lem}

\begin{proof}
We omit the proof of the first 
assertion since it is proved similarly 
to the second one.
Suppose that $\Gamma \simeq \bb{Z}_p^2$.
We denote $\rho$ or $\rho\rho'$ by $\chi$.
Then, by Lemma \ref{lemH2X}, 
Shapiro's lemma and \cite{Ta} Corollary 2.2, 
we have 
\[
X(K_\infty^{\Delta }  /K,\chi)
\simeq \varprojlim_{F \in \mca{IF}}
H^2(K_{\Sigma(\rho)} /F, \mca{O}( \chi))
\simeq H^2(K_{\Sigma(\rho)} /K, \Lambda \otimes \chi).
\]
Since the $p$-cohomological dimension of $G_{K,{\Sigma(\rho)}}$
is $2$, we obtain 
\[
H^2(K_{\Sigma(\rho)} /K, \Lambda \otimes \chi)
\otimes_{\Lambda}\left(
\Lambda/(\gamma -1 )
\right)
\simeq H^2 \left(
K_{\Sigma(\rho)} /K, 
\mca{O}[[\Gal(K_\infty^{\Delta \times \Gamma_1})]](\chi)
\right).
\]
Hence we obtain 
\(
X(K_\infty^{\Delta} /K,\chi ) \otimes_{\Lambda} 
\left(
\Lambda/(\gamma -1 )
\right)
\simeq X(K_\infty^{\Delta \times \Gamma_1} 
 /K,\chi).
\)

Let $\mathrm{tw}\colon 
\Lambda \longrightarrow \Lambda$
be the continuous homomorphism 
of topological $\mca{O}$-algebras
defined by 
$\mathrm{tw}(\sigma)=
{\rho'}(\sigma)^{-1}\sigma$
for any $\sigma \in \Gamma$.
Let $\Lambda'=(\Lambda,\mathrm{tw})$
be the $\Lambda$-algebra 
whose underlying ring is $\Lambda$,
and whose structure map is $\mathrm{tw}$.
Then, we have a $\Lambda'$-isomorphism
\begin{align*}
\left(
\varprojlim_{F \in \mca{IF}}
H^2(K_{\Sigma(\rho)} /F, \mca{O} (\rho) )
\right)\otimes_{\Lambda} 
\Lambda' 
\simeq 
\varprojlim_{F \in \mca{IF}}
H^2(K_{\Sigma(\rho)} /F, \mca{O} (\rho\rho') ),
\end{align*}
where we regard the right hand side as a $\Lambda'$-module
via the equality $\Lambda=\Lambda'$ of the underlying rings.
Hence we obtain
\(
X(K_\infty^{\Delta} /K,\rho ) \otimes_{\Lambda} \Lambda'
\simeq X(K_\infty^{\Delta} /K,\rho \rho')
\).
Since we have 
$\overline{\Lambda} \simeq 
\Lambda' \otimes \Lambda/(\gamma_1-1)$, 
we obtain the isomorphism as desired.
\end{proof}

\section{Euler systems}\label{secES}

In this section, 
we recall the notion 
of Euler systems 
(in the sense of \cite{Ru5})
for one dimensional Galois representations, 
and their Kolyvagin derivatives.
We use the same notations
as in the previous section.
Let $\rho \colon G_{K}
\longrightarrow \mca{O}^\times$
be a continuous character satisfying 
the conditions (C1), (C2) and (C3)
introduced in \S \ref{sssel}.
We fix a finite set $\Sigma$
of primes of $\mca{O}_K$
containing $\Sigma(\rho)$.

In \S \ref{ssdefES}, we recall 
the definition of  Euler systems.
In \S \ref{ssKdiv}, we recall 
the Kolyvagin derivatives of Euler systems.
In \S \ref{ssellunits}, 
we recall the Euler systems of 
elliptic units.

\subsection{Euler systems}\label{ssdefES}

In this section, we recall 
the definition of Euler systems.
We denote by $\mathcal{N}$ the set of 
all ideals of $\mca{O}_K$ decomposed into 
square-free products of prime ideals not contained in $\Sigma$.
For each prime ideal $\mf{l} \in \mca{N}$, 
we denote the $p$-ray class field modulo $\mf{l}$ 
by $K\langle \mf{l}\rangle$. 
Then, for any $\mf{n} \in \mca{N}$ 
with prime ideal decomposition
$\mf{n} = \mf{l}_1  \cdots \mf{l}_r$, 
we define the composite field
$K\langle \mf{n} \rangle :=
K \langle \mf{l}_1\rangle \cdots K\langle \mf{l}_r \rangle$.
For any $\mf{n} \in \mca{N}$ and $F \in \mca{IF}$, 
we define $F\langle \mf{n} \rangle:=F \cdot 
K\langle \mf{n} \rangle$
to be the composite field.

\begin{dfn}\label{defES}
We call a family
\[
\boldsymbol{c}:=\left\{ 
c_F( \mf{n} ) \in H^1(F \langle \mf{n} \rangle 
,\mca{O}(\rho)) 
\right\}_{(F,\mf{n}) \in \mca{IF} \times \mathcal{N}}
\]
of cohomology classes {\it an Euler system for 
$(K_\infty^\Delta/K,\Sigma,\rho)$} 
if $\boldsymbol{c}$ satisfies the following conditions:
\begin{itemize}\setlength{\leftskip}{4mm}
\item[$({\rm ES1})$] For any $F,F' \in \mca{IF}$ with 
$F \subseteq F'$ and 
$\mf{n} \in \mca{N}$, 
we have 
\[
\mathrm{Cor}_{F' \langle \mf{n} \rangle 
/F\langle \mf{n} \rangle}\left(
c_{F'} ( \mf{n} ) \right) = 
c_F ( \mf{n} ).
\]
\item[$({\rm ES2})$] Let $F \in \mca{IF}$ and 
$\mf{n} \in \mca{N}$.
Then, for any prime divisor $\mf{l}$ of $\mf{n}$, we have 
\[
\mathrm{Cor}_{F \langle \mf{n} \rangle / 
F \langle \mf{n}/\mf{\mf{l}} \rangle}\left(
c_{F}(\mf{n}) \right) = 
(1-N(\mf{l})^{-1}\rho(\mathrm{Fr}_{\mf{l}})
\mathrm{Fr}_{\mf{l}}^{-1})
 \cdot c_m(\mf{n}/\mf{l}),
\]
where $\mathrm{Fr}_{\mf{l}} \in \Gal \left( 
F \langle \mf{n}/\mf{l} \rangle / K \right)$ is 
the arithmetic Frobenius element at $\mf{l}$.
\end{itemize}
We denote by 
$\mathrm{ES}_\Sigma(K_\infty^\Delta/K; \mca{O}(\rho))$
the set of all Euler systems for 
$(K_\infty^\Delta/K,\Sigma,\mca{O}(\rho))$.
\end{dfn}

Twist of Euler systems is
a key in the specialization arguments
in \S \ref{secpf}.

\begin{dfn}
Let $\boldsymbol{c}:=\{ 
c_F( \mf{n} ) \}_{F,\mf{n}} 
\in \mathrm{ES}_\Sigma(K_\infty^\Delta/K; \mca{O}(\rho))$
be any Euler system, and 
\(
\rho' \colon \Gamma
\longrightarrow 1+ \pi\mca{O}
\)
a continuous character.
Then, we can define an Euler system 
\(
\boldsymbol{c}\otimes \rho :=\{ 
(c\otimes \rho)_F( \mf{n} ) \} 
\in \mathrm{ES}_\Sigma(K_\infty^\Delta/K; \mca{O}(\rho\rho'))
\)
such that $(c\otimes \rho)_F( \mf{n} )$ is 
equal to 
the image of 
\[
(c_{F'}(\mf{n}))_{F'\in \mca{IF}}\otimes \rho
\in \varprojlim_{F' \in \mca{IF}} 
H^1(F' \langle \mf{n} \rangle, \mca{O}(\rho))\otimes \rho'
\simeq \varprojlim_{F' \in \mca{IF}} 
H^1(F' \langle \mf{n} \rangle, \mca{O}(\rho\rho'))
\]
in $H^1(F \langle\mf{n} \rangle,\mca{O}(\rho\rho'))$
for each $F \in \mca{IF}$ and $\mf{n} \in \mca{N}$.
We call $\boldsymbol{c}\otimes \rho$ 
{\em the twist} of the Euler system $\boldsymbol{c}$
by the character $\rho$.
\end{dfn}

\subsection{Kolyvagin derivatives}\label{ssKdiv}

Let us recall the definition of 
Kolyvagin derivatives.
For any prime $\mf{l}$ of $\mca{O}_K$ not contained in $\Sigma$, 
let $I_{\mf{l}}$ be the ideal of 
$\mca{O}$ generated by $N(\mf{l})-1$ and 
$N(\mf{l})^{-1}\rho(\mathrm{Fr}_{\mf{l}}) -1$. 

Let $\mf{n} \in \mca{N}$ be any element. 
Suppose that  $\mf{n} \in \mathcal{N}_N$  
has prime factorization $\mf{n}=\mf{l}_1 \cdots \mf{l}_r$.
We define the ideal 
$I_{\mf{n}} := \sum_{i=1}^rI_{\mf{l}_i}$ 
of $\mca{O}$.
We define $H_{\mf{n}}:=\mathrm{Gal}\left( 
K \langle \mf{n} \rangle /K \langle \mca{O}_K \rangle \right)$.
We define $H_{\mf{n}}:=\mathrm{Gal}\left( 
K \langle \mf{n} \rangle /K \langle \mca{O}_K \rangle \right)$.
For any $F \in \mca{IF}$, 
we have natural isomorphisms
\[
\mathrm{Gal}(F \langle \mf{n} \rangle /
F \langle \mca{O}_K \rangle) \simeq H_{\mf{n}}
\simeq  H_{\mf{l}_1}\times \dots \times H_{\mf{l}_r}.
\]
These groups are identified via the above natural isomorphisms.

\begin{dfn}
Let $N \in \bb{Z}_{>0}$ and $F \in \mca{IF}$ 
be any element.
We define a set $\mathcal{P}_{F,N}(\rho)$ 
of prime numbers by 
\[
\mathcal{P}_{F,N}(\rho) := \left\{\mf{l} 
\in P(K) \mathrel{\bigg\vert} 
\begin{array}{l}
\text{$\mf{l}$ splits completely in 
$F\langle \mca{O}_K \rangle /K$, } \\
\mf{l} \notin \Sigma, \ \text{and}
I_{\mf{l}} \subseteq \pi^N \mca{O}.
\end{array}
\right\}. 
\]
Then, we define
\[
\mathcal{N}_{F,N}(\rho) := \left\{
\prod_{i=1}^r \mf{l}_i \mathrel{\bigg\vert} 
\begin{array}{l}
\text{$r \in \mathbb{Z}_{>0}$, 
$\mf{l}_i \in \mathcal{P}_{F,N}(\rho)$ 
($i=1,\dots,r$)} \\ 
\text{and $\mf{l}_i\ne \mf{l}_j$ if $i\ne j$} 
\end{array}
\right\} \cup\{ \mca{O}_K \}. 
\]
For simplicity,
we write $\mathcal{P}_{F,N}:=\mathcal{P}_{F,N}(\rho)$
and 
$\mathcal{N}_{F,N}:=\mathcal{N}_{F,N}(\rho)$.
We also write 
$\mathcal{P}_{N}:=\mathcal{P}_{K,N}$
and 
$\mathcal{N}_{N}:=\mathcal{N}_{K,N}$.
\end{dfn}

\begin{rem}\label{remPN}
If $\rho=\chi_{\mathrm{cyc}}\psi^{-1}$, 
then our $\mca{P}_{F,N}$ coincides with
the set
$\mca{S}^{\mathrm{prime}}_N(F)$
defined in \cite{Oh1} \S 3.1. 
\end{rem}

Let $\mf{l} \in \mca{P}_N$ be any element. 
We shall take a generator 
$\sigma_{\mf{l}}$ of  
$H_\mf{l}$ as follows.

\begin{dfn}\label{dfnsigmaell}
Let $N_{\mf{l} }$ be the positive integer satisfying
\(
(N(\mf{l})-1)\bb{Z}_p
=p^{N_{ \mf{l} }}\bb{Z}_p. 
\)
Note that $\mf{l}$ splits completely in 
$K(\mu_{p^{N_{\mf{l}}}})/K$, so
by the fixed embedding 
$\mf{l}_{\overline{K}}\colon 
\overline{K}\hookrightarrow \overline{K}_{\mf{l}}$, 
we can regard 
$\mu_{p^{N_{\mf{l}}}}$ as a subset of $K_\mf{l}$. 
Let $\lambda_0:=\mf{l}_{K \langle \mca{O} \rangle}$ 
be the place of $K \langle \mca{O} \rangle$ 
below $\mf{l}_{\overline{K}}$, and  
$\lambda_1:=\mf{l}_{K \langle \mf{l} \rangle}$ 
be the place of $K \langle \mf{l} \rangle$ 
below $\mf{l}_{\overline{K}}$. 
We identify $\Gal \big(K \langle \mf{l} \rangle_{\lambda_1}/
K \langle \mca{O} \rangle_{\lambda_0} \big)$ with $H_\mf{l}$ 
by the isomorphism induced from the natural embeddings of fields. 
Let $\varpi$ be a uniformizer of $K \langle \mf{l} \rangle_{\lambda_1}$. 
We fix a generator $\sigma_{\mf{l}}$ of $H_{\mf{l}}$ 
such that 
\[
\varpi^{\sigma_\mf{l}-1} \equiv 
\zeta_{p^{N_{\mf{l }}} } 
\pmod{\mf{m}_
{\lambda_1}},
\]
where $\mf{m}_{{\lambda_1}}$ is the maximal ideal of 
$K \langle \mf{l} \rangle_{\lambda_1}$. 
Note that the definition of $\sigma_\mf{l}$ 
does not depend on the choice of $\varpi$.
\end{dfn}

In order to review the definition of Kolyvagin derivatives,
we need to introduce Kolyvagin operators. 

\begin{dfn}\label{Dn}
For $\mf{l} \in \mathcal{P}_N$, we define 
\(
D_\mf{l}:=\sum_{i=1}^{\# H_{\mf{l}}-1}i\cdot \sigma_{\mf{l}}^i 
\in \mathbb{Z}[H_{\mf{l}}].
\)
Let $\mf{n} \in \mathcal{N}_N$
be any element with prime  factorization
$\mf{n} = \prod_{i=1}^r \mf{l}_i $. 
Then, we define {\em the Kolyvagin operator} 
$D_{\mf{n}}\in \mathbb{Z}[H_{\mf{n}}]$ by 
\(
D_{\mf{n}}:=
\prod_{i=1}^r D_{\mf{l}_i} 
\).
\end{dfn}

Take any $F \in \mca{IF}$ and 
$\mf{n} \in \mathcal{N}_{F,N}$.

\begin{lem}[\cite{Ru5} Lemma 4.4.2]\label{lemDninv}
The image of $D_{\mf{n}} c_F(\mf{n})$  
in $H^1(F \langle \mf{n} \rangle, 
(\mca{O}/\pi^N \mca{O})\otimes\rho)$
is fixed by the action of $H_{\mf{n}}$. 
\end{lem}

By (C2)  and
Hochschild--Serre spectral sequence, 
the restriction map
\[
\mathrm{Res}^{(\mf{n})}_{F,N,\rho}\colon 
H^1(F,(\mca{O}/\pi^N \mca{O}) \otimes \rho) 
\longrightarrow
H^1(F \langle \mf{n} \rangle,
(\mca{O}/\pi^N \mca{O}) \otimes \rho)^{
\Gal(F \langle \mf{n} \rangle /F)}
\]
is an isomorphism.

\begin{dfn}[Kolyvagin derivative]\label{defKD}
Let $\widetilde{\mathrm{N}}_F \in \bb{Z}[\Gal(F 
\langle \mf{n} \rangle /F)]$ 
be a lift of the norm element 
\[
\mathrm{N}_F := \sum_{\sigma \in \Gal(F 
\langle \mca{O}_K \rangle
/F)} \sigma
\in \bb{Z}[\Gal(F \langle \mca{O}_K \rangle/F)].
\]
We define
\[
\kappa_{F,N}(\mf{n};\boldsymbol{z}):=
(\mathrm{Res}^{(\mf{n})}_{F,N,\rho})^{-1} \left( 
\widetilde{\mathrm{N}}_F 
D_n \boldsymbol{c}_F(\mf{n}) \right) \in 
H^1(F,(\mca{O}/\pi^N \mca{O}) \otimes \rho). 
\] 
By Lemma \ref{lemDninv},
the definition of 
the cohomology class
$\kappa_{F,N}(\mf{n};\boldsymbol{z}) $
is independent of the choice 
of the lift $\widetilde{\mathrm{N}}_F$.
We call the cohomology class $\kappa_{F,N}(
\mf{n};\boldsymbol{c})$ 
{\em the Kolyvagin derivative}.
\end{dfn}

\begin{prop}[\cite{Ru5} Theorem 6.5.1]\label{lemKDlc}
For any $F \in \mca{IF}$ and 
any $\mf{n} \in \mathcal{N}_{F,N}$, 
we have
\(
\kappa_{F,N}(\mf{n};\boldsymbol{c})
\in \mathrm{Sel}^{p\mf{n}}(F,
(\mca{O}/\pi^N \mca{O}) \otimes \rho).
\)
\end{prop}

\subsection{Elliptic units}\label{ssellunits}

Here, we recall  Euler systems of elliptic units briefly, 
and set some notation.
We consider the following conditions 
on the pair $(\mf{a},\mf{g})$ of ideals of $\mca{O}_K$.
\begin{itemize}
\item[$(\mathrm{I})_1$] The ideal $\mf{a}$ is prime to $6\mf{g}$,
and the natural map
$\mca{O}_K^\times \longrightarrow (\mca{O}_K/\mf{g})^\times$
is injective.
\item[$(\mathrm{I})_2$]  The ideal $\mf{g}$ is 
a divisor of the conductor $\mf{f}$ of $K_0/K$.
\end{itemize}
Suppose that a pair $(\mf{a}, \mf{g})$
satisfies the condition $(\mathrm{I})_1$.
Let $K(\mf{g})$ be the ray class field of $K$
modulo $\mf{g}$.
Then, we have an elliptic unit 
${_\mf{a}}z_{\mf{g}} 
\in K (\mf{g} )^\times$ 
in the sense of \cite{Oh1} Definition 2.3.
The  element 
${_\mf{a}}z_{\mf{g}} $
satisfies the following properties.

\begin{prop}\label{propell}
Let $(\mf{a}, \mf{g})$ and $(\mf{b}, \mf{g})$ be 
pairs the conditions $(\mathrm{I})_1$.
\begin{enumerate}[{\rm (i)}]
\item The element ${_\mf{a}}z_{\mf{g}}$ is a $\mf{f}$-unit of 
$K (\mf{f} )$.  
Moreover, if $\mf{g}$ is not 
a power of a prime of $\mca{O}_K$, 
then, ${_\mf{a}}z_{\mf{g}}$ becomes a unit.
{\rm (See, for instance, \cite{dS} 2.4 Proposition.)}
\item Let $\mf{l}$ be a prime of $\mca{O}_K$ 
not dividing $\mf{a}$, 
then, we have
\[
N_{K (\mf{gl})/ 
K (\mf{g})}({_\mf{a}}z_{\mf{gl}})=
\begin{cases}
{_\mf{a}}z_{\mf{g}} & (\mf{l} \mid \mf{g}), \\
(1-\Fr_{\mf{l}}^{-1})\cdot {_\mf{a}}z_{\mf{g}}
& (\mf{l} \nmid \mf{g}),
\end{cases}
\]
where we write the scalar action of 
$1-\Fr_{\mf{l}}^{-1} \in 
\bb{Z}[\Gal(K (\mf{g} )/K)]$ to $K(\mf{g})^\times$
in the multiplicative way.
{\rm (See, for instance,  \cite{dS} 2.5 Proposition.)}
\item We have 
\[
(N\mf{b}-\tau_{\mf{b}})\cdot 
{_\mf{a}}z_{\mf{g}}
=(N\mf{a}-\tau_{\mf{a}})\cdot 
{_\mf{b}}z_{\mf{g}},
\] 
where for any ideal $\mf{c}$ of $\mca{O}_K$
prime to $\mf{g}$, 
we write $\tau_{\mf{c}}:=(\mf{c}, K (\mf{g} )/K)$.
{\rm (See, for instance,  \cite{dS} 2.4 Proposition, or 
\cite{Ka} the equality (15.4.4) in page 253.)}
\end{enumerate}
\end{prop}

\begin{dfn}[elliptic units]\label{defEU}
We denote by $\Sigma(K_0/K)$ by the set of primes of $\mca{O}_K$
consisting of all primes dividing $p\mca{O}_K$, 
and all ramified primes in $K_0/K$.
Let $\psi \in \widehat{\Delta}$ be any element,
and put $\mca{O}_\psi:=\bb{Z}_p[\mathrm{Im} \psi ]$.
Let $F \in \mca{IF}$, and $(\mf{a}, \mf{g})$ 
a pair of ideals of 
$\mca{O}_K$ satisfying 
the conditions $(\mathrm{I})_1$ and $(\mathrm{I})_2$.
We denote by $\mf{h}$ by the conductor of $F/K$.
For any ideal $\mf{n}$ of $\mca{O}_K$
prime to $\mf{a}$, we define
\[
{c^{\mf{a}}_{\mf{g}}}(F;\mf{n}):=N_{K (p\mf{gn} \mf{h})/ 
(K_0F \langle \mf{n} \rangle
\cap K (p \mf{gn}\mf{h} ))}
({_\mf{a}}z_{p\mf{gnh}} ) \in K_0F 
\langle \mf{n} \rangle^\times.
\]
Then, we obtain an Euler system
\[
\boldsymbol{c}^{\mf{a}}_{\mf{g},\psi}:=
\left\{
{c^{\mf{a}}_{\mf{g}}}(F; \mf{n})_\psi \in 
\left(
\mca{O}_{K_0F \langle \mf{n} \rangle}
\left[ p^{-1}
\right]^\times 
\otimes_{\bb{Z}}\bb{Z}_p
\right)_\psi \subseteq 
H^1\left(
F, \mca{O}_\psi
\left(
\chi_{\mathrm{cyc}}\psi^{-1}
\right)
\right)
\right\}_{F,\mf{n}}
\]
for $(K_\infty^{\Delta}/K,
\Sigma(K_0/K) \cup \mathrm{Prime}(\mf{a}), 
\mca{O}_\psi
\left(
\chi_{\mathrm{cyc}}\psi^{-1}
\right))$.
We define the set $Z^{\mathrm{ell}}$ 
of Euler systems of elliptic units by
\[
Z^{\mathrm{ell}}_\psi :=
\left\{
\boldsymbol{c}^{\mf{a}}_{\mf{g},\psi} 
\mathrel{\vert} 
\text{the pair $(\mf{a}, \mf{g})$ satisfies  
$(\mathrm{I})_1$ and $(\mathrm{I})_2$} 
\right\}.
\]
\end{dfn}

\section{Construction of the ideal $\mathfrak{C}_i(Z)$}\label{secCi}

Let $\rho\colon 
G_{K} \longrightarrow \mca{O}^\times$
be a continuous character satisfying 
the assumptions (C1), (C2) and (C3).
We fix a finite set $\Sigma$
of primes of $\mca{O}_K$
containing $\Sigma(\rho)$.
We fix a non-empty set $Z= \bigcup_{\mf{a}} Z_{\mf{a}}$
of Euler systems, 
where $\mf{a}$ runs through all ideals of $\mca{O}_K$
not divided by primes contained in $\Sigma$, and
\(
Z_{\mf{a}} \subseteq 
\mathrm{ES}_{\Sigma \cup \mathrm{Prime}(\mf{a})}(K_\infty^\Delta/K;
\mca{O}(\rho))
\).

In this section, we shall construct ideals 
$\mf{C}_i(Z)$ of the Iwasawa algebra 
$\Lambda=\mca{O}[[\Gamma]]$
by using Kolyvagin derivatives, 
and show some basic properties of $\mf{C}_i(Z)$.

In \S \ref{ssCi}, we define the ideals $\mf{C}_i(Z)$, 
and in \S \ref{sspropCi}, 
we show some basic properties of 
the ideals $\mf{C}_i(Z)$.

\subsection{Construction of $\mf{C}_i(Z)$}\label{ssCi}
First, we fix  $F,F' \in \mca{IF}$ 
with $F \subseteq F'$, and 
$N \in \bb{Z}_{>0}$. 
We put \(
R_{F,N}:=
\mca{O}/\pi^N\mca{O}[\Gal(F/K)]
\).
We shall construct an ideal 
$\mathfrak{C}_{i,F,N}^{F'}(Z)$
of $R_{F,N}$ for each $i \in  \bb{Z}_{\ge 0}$.
We need the following notion introduced 
by Kurihara \cite{Ku}.

\begin{dfn}
An element $\mf{n} \in\mathcal{N}_{F',N}(\rho)_{\mca{O}}$ 
is {\em well-ordered} if and only if 
$\mf{n}$ has a prime factorization
$\mf{n}= \prod_{i=1}^r \mf{l}_i$
such that $\mf{l}_{i}$ splits completely 
in $K(\prod_{j=1}^{i-1}\mf{l}_j)$ for 
each $i \in \{1,2, \dots, r \}$.
We denote by $\mf{n} \in \mathcal{N}^{\mathrm{wo}}_{F',N}$
the subset of $\mathcal{N}_{F,N}$
consisting of all the well-ordered elements.
\end{dfn}

Let $\mf{n} \in\mathcal{N}^{\mathrm{wo}}_{F,N}(\rho)_{\mca{O}}$ with 
the prime factorization $\mf{n}=\prod_{j=1}^r \mf{l}_j$.  
We write $Z ^{\mf{n}}= \bigcup_{\mf{a}} Z_{\mf{a}}$, 
where $\mf{a}$ runs through all ideals of $\mca{O}_K$ not divided by
any prime contained in $\Sigma \cup \mathrm{Prime}(\mf{n})$.  
Let $\boldsymbol{c} \in Z_{\mf{a}}$ be any element.
We denote the number of prime divisors of $\mf{n}$ 
by $\epsilon (\mf{n})$, that is, $\epsilon (\mf{n}):=r$.
We define an ideal 
$\mathfrak{C}_{F,N}(\mf{n};\boldsymbol{c})_{\mca{O}}$ of $R_{F,N}$ by 
\[
\mathfrak{C}_{F,N}(\mf{n};\boldsymbol{c})_{\mca{O}}:= \left\{
f(\kappa_{F,N}(\mf{n};\boldsymbol{c})) \mathrel{\vert} 
f \in \mathrm{Hom}_{R_{F,N}}(H^1(F,(\mca{O}/\pi^N\mca{O})\otimes \rho),R_{F,N})
\right\}.
\]

\begin{dfn}\label{theideal0}
For any $i \in \bb{Z}_{\ge 0}$,
we denote by $\mathfrak{C}^{F'}_{i,F,N}(Z)_{\mca{O}}$
the ideal of $R_{F,N}$ generated by 
\(
\bigcup_{\mf{n}} 
\bigcup_{\boldsymbol{c} \in Z^{\mf{n}}}
\mathfrak{C}_{F,N}
(\mf{n};\boldsymbol{c})_{\mca{O}},
\)
where $\mf{n}$ runs through all the elements of 
$\mathcal{N}^{\mathrm{wo}}_{F,N}$ 
satisfying $\epsilon (\mf{n})\leq i$, 
and $\mf{a}+\mf{n}=\mca{O}_K$. 
\end{dfn}

Now, we shall vary $F$ and $N$, 
and construct the ideal 
$\mathfrak{C}_{i}(K_\infty^{\Delta}/K,Z)_{\mca{O}}$ of $\Lambda$. 
As \cite{Oh1} Claim 4.4 and \cite{Oh2} Lemma 4.13, 
the following lemma holds.

\begin{lem}\label{liftinghom}\label{lemma}
Suppose that $F_1, F_2 \in \mca{IF}$ and  
$N_1, N_2 \in  \bb{Z}_{>0}$ 
are elements satisfying  
$F_2 \supseteq F_1$ and $N_2 \ge N_1$. 
We put $W_i:=(\mca{O}/\pi^{N_i}\mca{O})\otimes \rho$
for each $i \in \{1,2 \} $. 
\begin{itemize}
\item[(i)] For any $R_{F_2,N_2}$-linear map
\(
f_2 \colon H^1(F_2,W_2)
\longrightarrow R_{F_2,N_2},
\)
there exists an $R_{F_1,N_1}$-linear map
\(
f_1 \colon 
H^1(F_1,W_1) \longrightarrow 
R_{F_1,N_1,}
\)
which makes the diagram
\[
\xymatrix{
H^1(F_2,W_2) \ar[r]^(0.6){f_2} 
\ar[d]_{\mathrm{Cor}_{F_2/F_1}} & 
R_{F_2,N_2} \ar@{->>}[d]  \\
H^1(F_1,W_1) 
\ar@{-->}[r]^(0.6){f_1} & 
R_{F_1,N_1}
}
\]
commute, where the left vertical arrow 
$\mathrm{Cor}_{F_2/F_1}$
is the corestriction map, and
the right one is the natural projection.
\item[(ii)] Assume that $N_1=N_2$, and put $W:=W_1=W_2$. 
Then, For any $R_{F_1,N}$-linear map
\(
g_1 \colon 
H^1(F_1 ,W) \longrightarrow 
R_{F_1,N},
\)
there exists an $R_{F_2,N}$-linear map 
\[
g_2 \colon 
H^1(F_2 ,W) \longrightarrow 
R_{F_2,N}
\]
which makes the diagram
\[
\xymatrix{
H^1(F_2,W) \ar@{-->}[r]^(0.6){g_2} 
\ar[d]_{\mathrm{Cor}_{F_2/F_1}} & 
R_{F_2,N,\psi} \ar@{->>}[d]  \\
H^1(F_1,W) \ar[r]^(0.6){g_1} & 
R_{F_1,N,\psi}
}
\]
commute. 
\end{itemize}
\end{lem}

\begin{proof}
By (C2), 
the restriction map
\(
H^1(F_1,W_2)  \longrightarrow
H^1(F_2,W_2)^{\Gal(F_2/F_1)}
\)
is an isomorphism.
Moreover, by (C2) 
the map
\(
H^1(\times \pi^{N_2-N_1})\colon 
H^1(F_2,W_1)
\longrightarrow
H^1(F_2,W_2)
\)
becomes an injection.
Note that 
the $R_{F_2,N}$-module
$R_{F_2,N}$
is injective 
since the injective 
$R_{F_2,N}$-module
\(
\Hom_{\bb{Z}_p}(R_{F_2,N},
\bb{Q}_p/\bb{Z}_p)
\) 
is free of rank one.
By taking care of these facts, 
we can prove Lemma \ref{liftinghom}
by similar arguments to those
in the proof of \cite{Oh2} Lemma 4.13.
Indeed, the proof of 
Lemma \ref{liftinghom} is reduced to 
the following two cases:
\begin{itemize}
\item[(A)] $F_1=F_2$ and $N_2=N_1+1$,
\item[(B)] $[F_2:F_1]=p$ and $N_1=N_2$.
\end{itemize}
In each case, 
we can deduce the assertions as desired 
similarly to loc.\ cit.\ 
by using the facts noted above.
\end{proof}

Let $F_1,F_2, N_1,N_2$ and $\mf{n}$ be as above. 
Then, Lemma \ref{lemma} and 
the norm compatibility (ES1) of 
the Euler systems imply that
the image of 
$\mathfrak{C}^{F_2}_{i,F_2,N_2}(Z)_{\mca{O}}$ in
$R_{F_1,N_1}$ is contained in 
$\mathfrak{C}^{F_1}_{i,F_1,N_1}(Z)_{\mca{O}}$.
We obtain the projective system 
of the natural homomorphisms 
\[
\left\{
\mathfrak{C}^{F_2}_{i,F_2,N_2}(Z)_{\mca{O}} \longrightarrow 
\mathfrak{C}^{F_1}_{i,F_1,N_1}(Z)_{\mca{O}} \mathrel{\big\vert}
\text{
$F_2 \supseteq F_1$ and $N_2 \ge N_1$}
\right\}
\]
of $\Lambda$-modules.
We define  
$\mathfrak{C}_{i}(Z)$ as follows.

\begin{dfn}\label{theideal1}
For any $i \in \bb{Z}_{\ge 0}$, we define the ideal
$\mathfrak{C}_{i}(K_\infty^\Delta /K, Z)_{\mca{O}}
$ of $\Lambda=\varprojlim R_{F,N}$ 
by the projective limit
\(
\mathfrak{C}_{i}(K_\infty^\Delta /K, Z)_{\mca{O}}:=
\varprojlim_{F,N} \mathfrak{C}^F_{i,F,N}(Z)_{\mca{O}}.
\)
If no confusion arises, we denote 
$\mathfrak{C}_{i}(K_\infty^\Delta /K, Z)_{\mca{O}}$
by $\mathfrak{C}_{i}(Z)$
or $\mathfrak{C}_{i}(K_\infty^\Delta /K, Z)$
for simplicity.
For each $i \in \bb{Z}_{\ge 0}$ and 
$F,F' \in \mca{IF}$ with $F \subseteq F'$,  
we also define 
\[
\mf{C}^{F'}_{i}(F/K,Z)
:= \varprojlim_{N}\mf{C}^{F'}_{i,F,N}(Z)
\subseteq \varprojlim_{N} 
(\mca{O}/\pi^N\mca{O})[\Gal(F/K)]
=\mca{O}[\Gal(F/K)].
\]
We put $\mf{C}_{i}(F/K,Z)
=\mf{C}^{F}_{i}(F/K,Z)$.
When $F=K$,  
we write $\mf{C}^{F'}_{i}(K,Z)
:=\mf{C}^{F'}_{i}(K/K,Z)$, 
and $\mf{C}_{i}(K,Z)
:=\mf{C}^{K}_{i}(K,Z)$.
\end{dfn}

The definition of the ideal
$\mathfrak{C}_{i,\psi}^{\mathrm{ell}}$
is as follows.

\begin{dfn}\label{theideal2}
Let $\psi \in \widehat{\Delta}$ be any element, and
$Z^{\mathrm{ell}}_\psi$ be the set of 
Euler systems of elliptic units
introduced in Definition \ref{defEU}.
We put $\mca{O}_\psi:=\bb{Z}_p[\mathrm{Im} \psi]$.
Then, for any 
$i \in \bb{Z}_{\ge 0}$, we define the ideal
\(
\mathfrak{C}_{i,\psi}^{\mathrm{ell}}:=
\mathfrak{C}_{i}(K_\infty^\Delta /K, Z^{\mathrm{ell}}_\psi
)_{\mca{O}_\psi }.
\)
For each $i \in \bb{Z}_{\ge 0}$ and $F,F' \in \mca{IF}$
with $F \subseteq F'$,  
we define 
\(
\mf{C}^{\mathrm{ell},F'}_{i}(F/K)_\psi
:= \mathfrak{C}^{F'}_{i}(F/K, Z^{\mathrm{ell}}_\psi
)_{\mca{O}_\psi}
\subseteq \mca{O}_\psi [\Gal(F/K)].
\)
We put $\mf{C}^{\mathrm{ell},F'}_{i}(F/K)_\psi
=\mf{C}^{{\mathrm{ell}},F'}_{i}(F/K)_\psi$.
When $F=K$,  
we write $\mf{C}^{\mathrm{ell},F'}_{i}(K)_\psi
:=\mf{C}^{\mathrm{ell},F'}_{i}(K/K)_\psi$, 
and $\mf{C}^{{\mathrm{ell}}}_{i}(K)_\psi
:=\mf{C}^{{\mathrm{ell}},K}_{i}(K)_\psi$.
\end{dfn}

\begin{rem}\label{remCiell}
Under the assumption that
$\psi \vert_{D_{\Delta,\mf{p}}} \ne 1$
for any prime $\mf{p}$ above $p$,
we have 
\(
{c^{\mf{a}}_{\mf{g}}}(F;\mf{n}) \in 
\left(
\mca{O}_{K_0F \langle \mf{n} \rangle}
\left[ p^{-1}
\right]^\times 
\otimes_{\bb{Z}}\bb{Z}_p
\right)_\psi
=\left(
\mca{O}_{K_0F \langle \mf{n} \rangle}^\times 
\otimes_{\bb{Z}}\bb{Z}_p
\right)_\psi
\)
for any $F \in \mca{IF}$, 
any pair $(\mf{a},\mf{g})$ and any $\mf{n}$.
So, the Euler systems of elliptic units which 
we use coincides with those in \cite{Oh1}.
By the fact noted in Remark \ref{remPN},
we can easily show that our ideals 
$\mathfrak{C}_{i,\psi}^{\mathrm{ell}}$
coincides with those defined in 
\cite{Oh1} Definition 4.5.
\end{rem}

\begin{rem}\label{remCigen}
Let $F \in \mca{IF}$, and $N \in \bb{Z}_{>0}$.
Let $\psi \in \widehat{\Delta}$
be a character 
such that $\psi\vert_{D_{\Delta, \mf{l}}}\ne 1$
for any $\mf{l} \in \Sigma(K_0/K)$, and 
$\psi \ne  \omega$ as a character of $G_K$.
In this situation, we can write 
$\mf{C}_{i,F,N}(Z^{\mathrm{ell}}_{\psi} )$
in simpler form. 
Let $\mf{f}$ be the conductor of 
the extension $K_0/K$, and 
$\mf{g}_\psi$ the conductor of $\psi$.
By Chebotarev density theorem,
we can choose a prime $\mf{l}$ of $\mca{O}_K$
satisfying 
$\psi\vert_{G_{K_v}} \ne \omega\vert_{G_{K_v}}$.
Then, we have 
$N(\mf{l})-\psi(\Fr_\mf{l}) \in \mca{O}_\psi^\times$.
Moreover, by the assumption of $\psi$,
we have
$N(\mf{l})-\psi(\Fr_\mf{q}) \in \mca{O}_\psi^\times$
for any prime $\mf{q}$ dividing $p \mf{f} \mf{g}_\psi^{-1}$.
So, by Proposition \ref{propell},  
for any pair $(\mf{a},\mf{h})$ satisfying the condition 
$(\mathrm{I})_1$ and $(\mathrm{I})_2$, 
there exists $x \in \mca{O}_\psi[\Gal(F/K)]$ such that
$c^{\mf{a}}_{\mf{h}}(F;\mca{O}) 
=x\cdot c^{\mf{l}}_{\mf{f}}(F;\mca{O})$.
Hence for any $i \in \bb{Z}_{\ge 0}$ and 
any $N \in \bb{Z}_{\ge 0}$, we have
\(
\mf{C}_{i,F,N}(Z^{\mathrm{ell}}_{\psi} )=
\mf{C}_{i,F,N}( \{\mf{c}^{\mf{l}}_{\mf{f},\psi} \})
\).
\end{rem}

\subsection{Basic Properties of $\mf{C}_i(K_\infty^\Delta/K, Z)$}\label{sspropCi}
Here, we show two basic properties 
of the ideals $\mf{C}_i(K_\infty^\Delta/K, Z)$,
namely the compatibility for scalar extension of
the coefficient ring and
that for specialization via the character of $\Gamma$.

First, let us show the property
on the scalar extension of the coefficient ring.
Let $\mca{L}'$ be a finite extension field of $\mca{L}$,
and $\mca{O}'$ the ring of integers of $\mca{L}'$.
We naturally regard
$\mathrm{ES}_\Sigma(K_\infty^\Delta/K;\rho)_{\mca{O}}$
as a subset of 
$\mathrm{ES}_\Sigma(K_\infty^\Delta/K;\rho)_{\mca{O}'}$.
So, we can define the ideal
$\mathfrak{C}_{i}(K_\infty^\Delta /K, Z)_{\mca{O}'}$
for each $i \in \bb{Z}_{\ge 0}$.

\begin{lem}\label{lemOO'}
For each $i \in \bb{Z}_{\ge 0}$, 
we have 
\(
\mca{O}'\mathfrak{C}_{i}(K_\infty^\Delta /K, Z)_{\mca{O}}
=\mathfrak{C}_{i}(K_\infty^\Delta /K, Z)_{\mca{O}'}.
\)
\end{lem}

\begin{proof}
Let $F \in \mca{IF}$ and $N \in \bb{Z}_{>0}$ 
be any elements.
We denote by $e$ 
the ramification index of $\mca{L}'/\mca{L}$.
By definition, we have 
$\mca{N}_{F,N}^{\mathrm{w.o.}}(\rho)_{\mca{O}}
=\mca{N}_{F,eN}^{\mathrm{w.o.}}(\rho)_{\mca{O}'}$.
Let $\mf{n} \in 
\mca{N}_{F,N}^{\mathrm{w.o.}}(\rho)_{\mca{O}}$
be any element, and 
and
take any 
$\boldsymbol{c} \in Z^{\mf{n}}$.
In order to prove Lemma \ref{lemOO'}, 
it suffices to show that
\(
\mca{O}'\mathfrak{C}_{F,N}
(\mf{n};\boldsymbol{c})_{\mca{O}}
=\mathfrak{C}_{F,eN}
(\mf{n};\boldsymbol{c})_{\mca{O}'}.
\)
By definition, we have
$\mca{O}'\mathfrak{C}_{F,N}
(\mf{n};\boldsymbol{c})_{\mca{O}}
\subseteq \mathfrak{C}_{F,eN}
(\mf{n};\boldsymbol{c})_{\mca{O}'}$
obviously.
So, it is sufficient to show that
$\mca{O}'\mathfrak{C}_{F,N}
(\mf{n};\boldsymbol{c})_{\mca{O}}
\supseteq \mathfrak{C}_{F,eN}
(\mf{n};\boldsymbol{c})_{\mca{O}'}$.

Let $\pi$ be a prime element of $\mca{O}$, 
and $\pi'$ that of $\mca{O}'$. 
We put 
\(
R:=
\mca{O}/\pi^N\mca{O}[\Gal(F/K)]
\) and $R':=
\mca{O}'/{\pi'}^{eN}\mca{O}'[\Gal(F/K)]$.
We identify $R'$ with 
$\mca{O}'\otimes_{\mca{O}}R$.
We put $[\mca{L}':\mca{L}]:=n$, and
fix a basis $B:=\{ b_1, \dots , b_n \}$
of the free $\mca{O}$-module $\mca{O}'$.
Then, the image $\overline{B}:=\{ 
\bar{b_1}, \dots, \bar{b}_n\}$
of $B$ in $R'$ forms a basis of 
the free $R$-module $R'$.
Let $f\colon H^1(F,
(\mca{O}'/\pi^{eN}\mca{O}')\otimes \rho)
\longrightarrow R'$
be any $R'$-linear map, and put
\[
f \left(
\kappa_{F,N}(\mf{n};\boldsymbol{c})_{\mca{O}}
\right)
=f \left(
\kappa_{F,eN}(\mf{n};\boldsymbol{c})_{\mca{O}'}
\right)
=:\sum_{i=1}^n \bar{y}_ib_i,
\]
where $y_i \in R$ for any $i\in \{ 1, \dots, n \}$.
We put 
\[
W:=R \cdot \kappa_{F,N}
(\mf{n};\boldsymbol{c})_{\mca{O}}
\subseteq 
H^1(F,
(\mca{O}/\pi^{N}\mca{O})\otimes \rho).
\] 
Then, for each $i\in \{ 1, \dots, n \}$,
we can define the $R$-linear map
\[
g_i\colon W \longrightarrow R;\ 
\bar{a} \cdot \kappa_{F,N}
(\mf{n};\boldsymbol{c})_{\mca{O}}
\longmapsto \bar{a}y_i.
\]
Since $R$ is an injective $R$-module,
for each $i\in \{ 1, \dots, n \}$,
we the exists an $R$-linear map
\(
\tilde{g}_i\colon H^1(F,
(\mca{O}/\pi^{N}\mca{O})\otimes \rho)
\longrightarrow R
\)
satisfying $\tilde{g}_i \vert_W=g_i$. 
So, we obtain
\[
f \left(
\kappa_{F,N}(\mf{n};\boldsymbol{c})_{\mca{O}}
\right)
=\sum_{i=1}^n b_i
g_i(\kappa_{F,N}
(\mf{n};\boldsymbol{c})_{\mca{O}}) \in 
\mca{O}'\mathfrak{C}_{F,N}
(\mf{n};\boldsymbol{c})_{\mca{O}}.
\]
This completes the proof of Lemma \ref{lemOO'}.
\end{proof}

Next, let us study the (weak) stability of
$\mf{C}_i(K_\infty^{\Delta}/K,Z)$
under the specialization via 
$\mca{O}^\times$-valued continuous characters of $\Gamma$.
For each continuous character
$\rho' \colon \Gamma \longrightarrow 1+ \pi\mca{O}$,
we put 
\(
Z\otimes \rho'
:=\{ \boldsymbol{c}\otimes\rho'
\mathrel{\vert} \boldsymbol{c} \in Z \}.
\)
First, we state 
the two-variable version.

\begin{lem}[The weak specialization compatibility]\label{lemredCi}
Suppose that $\Gamma \simeq \bb{Z}_p^2$, 
and let $\gamma_1,\gamma_2$ be 
topological generators of $\Gamma$.
For each $i \in \{ 1,2 \}$, 
we define $\Gamma_i$ to be 
the closed subgroup of $\Gamma$
topologically generated by $\gamma_i$,
and identify $\Gamma_2$
with $\Gal(K_\infty^{\Delta \times \Gamma_1}/K)$.
Let 
$\rho' \colon \Gamma:= \Gal(K_\infty^{\Delta}/K)
\longrightarrow 1+ \pi\mca{O}$
be a continuous character satisfying
$\Ker \rho' \subseteq \Gamma_2$.
Then, the image of $\mf{C}_i(K_\infty^{\Delta}/K,Z)$ in 
\(
\Lambda/(\gamma_1-\rho'(\gamma_1)) = \mca{O}[[\Gamma_2]]
\)
is contained in 
$\mf{C}_i(K_\infty^{\Delta \times \Gamma_1}/K,Z\otimes\rho')$.
\end{lem}

\begin{proof}
We denote by $\mathrm{ord}_\mca{L} 
\colon \mca{L}^\times \twoheadrightarrow \bb{Z}$
the additive valuation. 
We put $u:=\rho'(\gamma_1)$, and
$m:=\mathrm{ord}_p(u-1) \in \bb{Z}_{>0}$.
For each $n_1,n_2 \in \bb{Z}_{\ge 0}$,
let $K_{n_1,n_1}$ be 
the maximal subfield of $K_\infty^\Delta$
fixed by $\gamma_1^{p^{n_1}}$ and $\gamma_2^{p^{n_2}}$.
Fix $n \in \bb{Z}_{\ge 0}$.
The character $\rho'$ induces a character
\[
\bar{\rho}'_n \colon 
\Gal(K_{n,n}/K)=\Gal(K_{n,0}/K) \times \Gal(K_{0,n}/K) 
\longrightarrow 
(\mca{O}/\pi^{mn}\mca{O})^\times, 
\]
given by $\bar{\rho}'_n(\bar{\sigma})  
= \rho'(\sigma) \mod \pi^{mn}$, 
where $\sigma$ denotes a lift of $\bar{\sigma}$
in $\Gamma$.
We put 
$R_n:=\mca{O}/\pi^{mn}\mca{O}[\Gal(K_{n,n}/K)]$ 
and $\overline{R}_n:=\mca{O}/\pi^{mn}\mca{O}[\Gal(K_{0,n}/K)]$.
By definition, we have 
natural isomorphisms
$\mca{O}[[\Gamma]] \simeq \varprojlim_n R_n$
and $\mca{O}[[\Gamma_2]] \simeq \varprojlim_n \overline{R}_n$.
We can extend the character $\bar{\rho}_n$
to a homomorphism
\[
\mathrm{ev}_{\rho'_n}\colon
R_n=\overline{R}_n[\Gal(K_{n,0}/K)]
\longrightarrow \overline{R}_n
\]
of $\overline{R}_n$-algebras, and obtain 
an $\mca{O}[[\Gamma_2]]$-algebra homomorphism
\[
\mathrm{ev}_{\rho'}:=(\mathrm{ev}_{\rho_n})_n \colon 
\mca{O}[[\Gamma]]=\mca{O}[[\Gamma_2]] [[\Gamma_1]]
\longrightarrow 
\mca{O}[[\Gamma_2]];\ 
\gamma_1 \longmapsto \rho (\gamma_1)
\]
whose kernel is  $(\gamma_1-u)$.
In order to prove Lemma \ref{lemredCi},
it suffices to show that
\[
\mathrm{ev}_{\rho'}(\mf{C}_i(K_\infty^{\Delta}/K,Z))
\subseteq  
\mf{C}_i(K_\infty^{\Delta \times \Gamma_1}/K,Z\otimes\rho').
\]

Note that we have
$\mca{P}_{K_{n,n},mn}(\rho)
= \mca{P}_{K_{n,n},mn}(\rho\rho')$.
Indeed, if a prime $\mf{l}$ of $\mca{O}_K$ 
splits completely in $K_{n,n}/K$, 
then we have $\rho'(\Fr_\mf{l}) \in 1+\pi^{mn}\mca{O}$.
Hence we obtain 
\(
\mca{N}_{K_{n,n},mn}^{\mathrm{w.o.}}(\rho)
=\mca{N}_{K_{n,n},mn}^{\mathrm{w.o.}}(\rho\rho')
\subseteq \mca{N}_{K_{0,n},mn}^{\mathrm{w.o.}}(\rho\rho')
\).

Let $\mf{n} \in \mca{N}_{K_{n,n},mn}^{\mathrm{w.o.}}(\rho)$
be any element.
Take any element
$\boldsymbol{c} \in Z_{\mf{a}}\subseteq Z^{\mf{n}}$.
We fix an $R_n$-linear map
\(
f \colon 
H^1(K_{n,n},(\mca{O}/\pi^{mn}\mca{O})\otimes \rho)
\longrightarrow R_n
\).
In order to prove Lemma \ref{lemredCi},
it is sufficient to prove that
$\mathrm{ev}_{\bar{\rho}'_n}(\mf{C}_{K_{n,n},mn}
(\mf{n};\boldsymbol{c}))
\subseteq  
\mf{C}_{K_{0,n},mn}
(\mf{n};\boldsymbol{c}\otimes\rho')$.
By the definition of $\boldsymbol{c}\otimes\rho'$ and 
Kolyvagin derivatives,
we have
\(
\kappa_{K_{n,n},mn}(\mf{n};\boldsymbol{c}\otimes \rho')
=\kappa_{K_{n,n},mn}(\mf{n};\boldsymbol{c})
\otimes \bar{\rho}'_n
\)
under the natural identification 
\[
H^1(K_{n,n},(\mca{O}/\pi^{mn}\mca{O})\otimes \rho\rho')
=H^1(K_{n,n},(\mca{O}/\pi^{mn}\mca{O})\otimes \rho)
\otimes \bar{\rho}'_n.
\]
By Lemma \ref{liftinghom} (i), 
we have an $\overline{R}_n$-linear map
$g \colon 
H^1(K_{0,n},(\mca{O}/\pi^{mn}\mca{O})\otimes \rho\rho')
\longrightarrow \overline{R}_n$
which makes the diagram
\begin{equation}\label{eqdiag}
\xymatrix{
H^1(K_{n,n},(\mca{O}/\pi^{mn}\mca{O})\otimes \rho) \ar[r]^(0.70)
{f\otimes \bar{\rho}'_n} 
\ar[d]_{\mathrm{Cor}_{K_{n,n}/K_{0,n}}} & 
R_{n}\otimes \bar{\rho}'_n \ar[d]_{\mathrm{ev}_{\bar{\rho}'_n}}  
\ar[r]^(0.6){\mathrm{tw}_{\bar{\rho}'_n}}
& R_n  \ar[dl]^{\mathrm{pr}}  \\
H^1(K_{0,n},(\mca{O}/\pi^{mn}\mca{O})\otimes \rho\rho') 
\ar[r]^(0.75){g} & 
\overline{R}_n &
}
\end{equation}
commute, where
$\mathrm{pr}$ is the projection, and 
$\mathrm{tw}_{\bar{\rho}'_n}$
denotes the homomorphism of 
$\mca{O}$-modules given by
$\mathrm{tw}_{\bar{\rho}'_n}(\sigma \otimes \bar{\rho}'_n)=
{\bar{\rho}'_n}(\sigma)^{-1}\sigma$
for each $\sigma \in \Gal(K_{n,n}/K)$.
So, we obtain
\begin{align*}
\mathrm{ev}_{\bar{\rho}'_n}(
f(\kappa_{K_{n,n},mn}(\mf{n};\boldsymbol{c}))\otimes \bar{\rho}'_n)
&=g(\mathrm{Cor}_{K_{n,n}/K_{0,n}}
(\kappa_{K_{0,n},mn}(\mf{n};\boldsymbol{c} \otimes \rho'))) \\
& \in   \mf{C}_{K_{0,n},mn}
(\mf{n};\boldsymbol{c}\otimes\rho').
\end{align*}
This completes the proof Lemma \ref{lemredCi}.
\end{proof}

Similarly, 
we obtain the following 
one variable version.

\begin{lem}[The weak specialization compatibility]\label{lemredCitodvr}
Suppose that $\Gamma \simeq \bb{Z}_p$, 
and let $\gamma$ be 
a topological generator of $\Gamma$.
Let $\rho' \colon \Gamma:= \Gal(K_\infty^{\Delta}/K)
\longrightarrow 1+ \pi \mca{O}$
be a continuous character.
Then, the image of $\mf{C}_i(K_\infty^{\Delta}/K,Z)$ in 
\(
\Lambda/(\gamma-\rho'(\gamma)) = \mca{O}
\)
is contained in 
$\mf{C}_i(K,Z\otimes\rho')$.
\end{lem}

By the arguments in the proof of Lemma \ref{lemredCi},
we immediately obtain the following.

\begin{cor}\label{corC_iIm}
Let $\Gamma$ be isomorphic to 
$\bb{Z}_p$ or $\bb{Z}_p^2$.
Let $I(\Gamma)$ be the augmentation ideal of 
$\Lambda=\mca{O}[[\Gamma]]$.
Then, the image of $\mf{C}_i(K_\infty^{\Delta}/K,Z)$ 
in $\mca{O}=\Lambda/I(\Gamma)$ coincides with 
$\bigcap_{F \in \mca{IF}} \mf{C}_i^F(K,Z)$.
\end{cor}

\section{Preliminary results on the ground level}\label{secpgl}

In this section, 
we introduce a preliminary result
on non-variable cases, 
namely  Proposition \ref{propgrlev}.
In this section, 
all notations follow to 
those in \ref{secCi}.
In particular, 
we fix a character $\rho\colon 
G_{K} \longrightarrow \mca{O}^\times$
satisfying the assumptions (C1), (C2) and (C3),
and fix a set $Z= \bigcup_{\mf{a}} Z_{\mf{a}}$
of Euler systems for $(K_\infty^\Delta/K,\mca{O}(\rho))$.
The following proposition 
is obtained by standard arguments of 
Euler systems.

\begin{prop}\label{propgrlev}
Suppose $\# X (K,\rho) < \infty$. 
For any $i \in \bb{Z}_{\ge 0}$ 
and  $F \in \mca{IF}$, we have 
\[
\mf{C}^F_{i}(K,Z) \subseteq
\Fitt_{\mca{O},i}(X(K,\rho)).
\]
Moreover, if we have 
$\mf{C}_{0}(KZ) =
\Fitt_{\mca{O},0}(X(K,\rho))$, 
then it holds that 
\[
\mf{C}^F_{i}(K,Z) =
\Fitt_{\mca{O},i}(X(K,\rho)) 
\]
for any $i \in \bb{Z}_{\ge 0}$ 
and  $F \in \mca{IF}$.
\end{prop}

Proposition \ref{propgrlev}
is a well-known fact, which is 
a variant of 
\cite{Ru1} Theorem 4.4
for one dimensional representations of $G_{K}$. 
The proof of Proposition \ref{propgrlev}
is also similar to that of \cite{Ru1} Theorem 4.4.
The author does not know the reference
which treat the assertion 
of Proposition \ref{propgrlev}
in our setting,
we review the proof of 
Proposition \ref{propgrlev}.

In \S \ref{ssl-fs}, we recall
the evaluation maps 
which map elements of the Galois group
to linear forms on  Selmer groups.
In \S \ref{ssCheb}, we prove 
Proposition \ref{Cheb}, 
which becomes a key of 
the induction arguments using Euler systems.
In \S \ref{sspfprpgrlev}, 
we complete the proof of 
Proposition \ref{propgrlev}.
In \S \ref{ssthmP}, we prove 
Theorem \ref{thm+P} by using 
Proposition \ref{propgrlev}
and the analytic class number formula.

\subsection{Evaluation maps}\label{ssl-fs}
In this and the next subsections, we fix 
a positive integer $N \in \bb{Z}_{>0}$.
Here, we introduce some ``evaluation maps" 
from a Galois groups
to the dual of a Galois cohomology groups. 
Note that various important maps 
in the theory of Euler systems, 
like finite singular comparison maps, 
are described in terms of evaluation maps.

Let $\bar{\rho}_N\colon G_{K}
\longrightarrow (\mca{O}/\pi^N\mca{O})^\times$
be the modulo-$\pi^N$ reduction of $\rho$, 
and $\bar{\chi}_{\mathrm{cyc},N}$ 
the modulo-$\pi^N$ reduction of the cyclotomic character.
We define $\Omega_{N}$ to be the composite field
of $K(\mu_{p^N})$ and $(K_{\Sigma})^{\Ker \bar \rho_N}$.
By (C2), the restriction maps
\begin{align*}
H^1(K,
(\mca{O}/\pi^N\mca{O})\otimes \rho)
&\longrightarrow 
\Hom(G_{\Omega_N},(\mca{O}/\pi^N\mca{O})\otimes \rho) \\
H^1(K,
(\mca{O}/\pi^N\mca{O})\otimes
\chi_{\mathrm{cyc}}\rho^{-1})
& \longrightarrow 
\Hom(G_{\Omega_N},(\mca{O}/\pi^N\mca{O})\otimes
\chi_{\mathrm{cyc}}\rho^{-1})
\end{align*}
are injective.
By the identifications 
\[
(\mca{O}/\pi^N\mca{O}) \otimes \rho
= \mca{O}/\pi^N\mca{O}=
(\mca{O}/\pi^N\mca{O})\otimes
\chi_{\mathrm{cyc}}\rho^{-1}
\]
via the choice of basis, 
the restriction maps induce 
the group homomorphisms
\begin{align*}
\mathrm{Ev}_{N} \colon  G_{\Omega_{N}}
& \longrightarrow \Hom_{\mca{O}}\left(H^1(K,
(\mca{O}/\pi^N\mca{O})\otimes \rho), 
\mca{O}/\pi^N\mca{O} \right), \\
\mathrm{Ev}_{N}^* \colon  G_{\Omega_{N}}
& \longrightarrow \Hom_{\mca{O}}\left(H^1(K,
(\mca{O}/\pi^N\mca{O})\otimes
\chi_{\mathrm{cyc}}\rho^{-1}), 
\mca{O}/\pi^N\mca{O} \right)
\end{align*}
called {\em the evaluation maps}.
Note that 
the assumption (C2) implies that
\[
H^1(\Omega_{N}/K,
(\mca{O}/\pi^N\mca{O})\otimes\rho)
=H^1(\Omega_{N}/K,
(\mca{O}/\pi^N\mca{O})\otimes
\chi_{\mathrm{cyc}}\rho^{-1})
=0.
\]
So, the following lemma 
follows from similar arguments to 
those in the proof of \cite{Ru5} Lemma 7.2.4.

\begin{lem}\label{lemevsurj}
For any integer $N_0$ satisfying 
$N \ge N_0$, we have 
\begin{align*}
\mca{O}\mathrm{Ev}_{N_0}(G_{\Omega_{N}})& =
\Hom_{\bb{Z}_p}\left(H^1(K,(\mca{O}/\pi^{N_0}\mca{O})\otimes \rho), 
\mca{O}/\pi^{N_0}\mca{O} \right), \\
\mca{O}\mathrm{Ev}_{N_0,X}^* (G_{\Omega_{N}})
&=X_{N_0}(K,\rho)
\end{align*}
\end{lem}

Let $\mf{l} \in \mca{P}_{K,N}$. 
By the Hochschild--Serre spectral sequence,
we have isomorphisms
\begin{align*}
H^1_f(K_{\mf{l}}, (\mca{O}/\pi^N\mca{O})\otimes \rho) 
& \simeq H^{1}(\mca{O}_K /\mf{l}, (\mca{O}/\pi^N\mca{O})\otimes \rho), \\
H^1_s(K_{\mf{l}}, 
(\mca{O}/\pi^N\mca{O})\otimes \rho) 
& \simeq \Hom(\Gal(K_{\mf{l}}^{\mathrm{ab}}/
K_\mf{l}^{\mathrm{ur}}),
(\mca{O}/\pi^N\mca{O})\otimes \rho),  
\end{align*}
where $K_{\mf{l}}^{\mathrm{ab}}$ 
(resp.\ $K_{\mf{l}}^{\mathrm{ur}}$)
denotes
the maximal abelian 
(resp.\ unramified)
extension field of $K_{\mf{l}}$
So, we have the following lemma.

\begin{lem}[\cite{Ru5} Lemma 1.4.7]\label{fs_free}
Let $\mf{l} \in \mca{P}_{K,N}$. 
Fix a lift $\Fr_\mf{l} \in 
\Gal(\overline{K}_{\mf{l}}/K_{\mf{l}})$
of Frobenius element and 
a lift $\tilde{\sigma}_{\mf{l}} \in 
\Gal(\overline{K}_{\mf{l}}/K_{\mf{l}})$ of 
$\sigma \in 
\Gal(K \langle \mf{l} \rangle_{\mf{l}_{K \langle \mf{l} \rangle}
}/K_{\mf{l}})=H_{\mf{l}}$.
Evaluating cocycles on $\Fr_\mf{l}$ and $\sigma_\mf{l}$
induces isomorphisms 
\begin{align*}
\mathrm{Ev}^{\mf{l},f}_{F,N}(\Fr_\mf{l})\colon 
H^1_f(K_{\mf{l}}, (\mca{O}/\pi^N\mca{O})\otimes \rho) 
& \xrightarrow{\ \simeq \ } 
(\mca{O}/\pi^N\mca{O})\otimes \rho= \mca{O}/\pi^N\mca{O}, \\
\mathrm{Ev}^{\mf{l},s}_{F,N}(\tilde{\sigma}_\mf{l})\colon
H^1_s(K_{\mf{l}}, 
(\mca{O}/\pi^N\mca{O})\otimes \rho) 
& \xrightarrow{\ \simeq \ } 
(\mca{O}/\pi^N\mca{O})\otimes \rho
= \mca{O}/\pi^N\mca{O}
\end{align*}
of $\mca{O}$-modules respectively.
The definition of 
$\mathrm{Ev}_{F,N}(\Fr_\mf{l})$
and $\mathrm{Ev}_{F,N}(\tilde{\sigma}_\mf{l})$
is independent of the choice of the lifts $\Fr_\mf{l}$
and $\tilde{\sigma}_{\mf{l}}$, but 
depends on the choice of $e$.
\end{lem}

\begin{dfn}
Let $\mf{l} \in \mca{N}_{K,N}$ be any element.
We call the composite map 
\begin{align*}
(\cdot)^{\mf{l},s}_{N,}\colon
H^1(K, (\mca{O}/\pi^N\mca{O})\otimes \rho) 
& \xrightarrow{\ \mathrm{res.} \ } 
H^1(K_{\mf{l}}, (\mca{O}/\pi^N\mca{O})\otimes \rho) \\
& \longrightarrow  
H^1_s(K_{\mf{l}}, (\mca{O}/\pi^N\mca{O})\otimes \rho) \\
& \xrightarrow{\mathrm{Ev}_{F,N}^{\mf{l},s}
(\tilde{\sigma}_\mf{l}) } 
\mca{O}/\pi^N\mca{O}
\end{align*}
{\em the localization map}, and the composite map 
\begin{align*}
\phi^{\mf{l}}_{N} \colon 
\Ker(\cdot)^{\mf{l},s}_{N}
& \xrightarrow{\ \mathrm{res.} \ }
H^1_f(K_{\mf{l}}, (\mca{O}/\pi^N\mca{O})\otimes \rho) \\
& \xrightarrow{\mathrm{Ev}^{\mf{l},f}_{N}(\Fr_\mf{l})} 
\mca{O}/\pi^N\mca{O}
\end{align*}
{\em the finite-singular comparison map}.
\end{dfn}

\begin{rem}\label{evlfs}
Let $\mf{l} \in \mca{P}_N(\rho)$.
Then, we have $\tilde{\sigma}_{\mf{l}}, 
\Fr_{\mf{l}} \in G_{\Omega_N}$, and 
the following hold.
\begin{itemize}
\item[{\rm (i)}] For any element 
$x\in H^1(K, (\mca{O}/\pi^N\mca{O})\otimes \rho) $, we have 
\(
(x)^{\mf{l},s}_{N}
= \mathrm{Ev}_{N}(\tilde{\sigma}_\mf{l})(x)
\).
\item[{\rm(ii)}] 
For any element $x\in \Ker (-)^{\mf{l},s}_{N}$, we have
\(
\phi^\mf{l}_{N}(x)
=  \mathrm{Ev}_{N}(\mathrm{Fr}_\mf{l})(x)
\).
\end{itemize}
\end{rem}

\begin{prop}[\cite{Ru5} Theorem 4.5.1 and \cite{Ru5} Theorem 4.5.4]
\label{kappalfs}
Let $\mf{n} \in \mathcal{N}_{N}$.
Then, for any prime divisor $\mf{l}$ of $\mf{n}$, 
we have 
\(
\left( \kappa_{K,N}(\mf{n};\boldsymbol{c}) \right)^{\mf{l},s}_{N}
=\phi^{\mf{l}}_{N} \left( 
\kappa_{K,N}(\mf{n}/\mf{l};\boldsymbol{c}) \right).
\)
\end{prop}

\subsection{Application of the Chebotarev density theorem}\label{ssCheb}

Recall that we have fixed a collection of embeddings
$\{ \mf{l}_{\overline{K}}\colon \overline{K} 
\hookrightarrow \overline{K}_\mf{l} \}_\mf{l}$ 
satisfying the condition (Chb) in \S 1.
We can deduce the following lemma by 
the standard arguments using 
the choice (Chb).

\begin{prop}\label{Cheb}
Let $N$ and $N_0$ be integers satisfying $N \ge N_0 >0$, 
and  $F \in \mca{IF}$ any element.
Let $\mf{n} \in \mathcal{N}^{\mathrm{w.o.}}_{F,N}(\rho )$.
Suppose that the following are given:
\begin{itemize}
\item an element $[c] \in 
\mathrm{Ev}_{N}^* (G_{\Omega_{N}})
\subseteq X_{N}(K,\rho)$.
\item an $\mca{O}/\pi^{N_0}\mca{O}$-submodule $W$ 
of $H^1(K,(\mca{O}/\pi^{N_0}\mca{O})\otimes \rho)$ 
of finite order;
\item an $\mca{O}/\pi^{N_0}\mca{O}$-linear map
$\psi  \colon W \longrightarrow \mca{O}/\pi^{N_0}\mca{O}$.
\end{itemize}
Then, there exist infinitely many 
$\mf{q} \in \mca{P}_{F,N}(\rho)$ 
satisfying the following.
\begin{enumerate}[{\rm (i)}]
\item The prime $\mf{q}$ splits completely
in $F \langle \mf{n} \rangle /K$.
\item We have  
\(
\mathrm{Ev}^*_{N,X}(\mathrm{Fr}_{\mf{q}})
=[c].
\)
\item The group $W$ is contained 
in the kernel of $(-)^{\mf{q},s}_{N_0}$, 
and there exists an element $u \in \mca{O}^\times$
satisfying 
\(
\psi(x)= u\phi^{\mf{q}}_{N_0}(x)
\)
for any $x \in W$.
\end{enumerate}
\end{prop}

\begin{proof}
Let $L_1$ be the maximal subfield of 
$\overline{K}$
fixed by the kernel of the evaluation map
$\mathrm{Ev}_{N,X}^*\colon G_{\Omega_N}\longrightarrow 
X_{N}(K,\rho)$, 
and $L_2$  that
fixed by the kernel of  
\[
\mathrm{Ev}_{N_0,W}\colon
G_{\Omega_N}\longrightarrow  
\Hom_{\bb{Z}_p}\left(W,  
\mca{O}/\pi^{N_0}\mca{O}
\right);\ \sigma \longmapsto 
\mathrm{Ev}_{N_0}(\sigma)\vert_W.
\]
We put $L_3:=F \langle \mf{n} \rangle 
\cdot \Omega_N(\mu_{p^N})$. 
The extensions 
$L_1$, $L_2$ and $L_3$ are Galois over $K$, 
and 
$\Gal(L_1/\Omega_N)$, $\Gal (L_2/\Omega_N)$
and $\Gal (L_3/\Omega_N)$
are abelian $p$-groups equipped with 
the actions of $\Gal(\Omega_N/K)$.
We define the composite field
$L:=L_1\cdot L_2 \cdot F(\mu_{p^N})$.
Note that the extension $L/K$ is Galois.
By (C1) and (C2),
the fields $L_1$, $L_2$ and $L_3$  
are linearly disjoint over $\Omega_N$. 
Indeed, the Jordan--H\"older constituents 
of the $\mca{O}[\Gal(\Omega_N/K)]$-module 
$\Gal(L_1/\Omega_N)\otimes_{\bb{Z}_p} \mca{O}$ 
(resp.\ $\Gal(L_2/\Omega_N)\otimes_{\bb{Z}_p} \mca{O}$ and 
$\Gal(L_3 /\Omega_N)\otimes_{\bb{Z}_p} \mca{O}$) 
are isomorphic to 
$k(\bar{\chi}_{\mathrm{cyc}}\bar{\rho}^{-1})$
(resp.\ $k(\bar{\rho})$ and $k$).
Hence by Lemma \ref{lemevsurj},
we have an element 
$\sigma \in \Gal(L/L_3 )$ satisfying
$\mathrm{Ev}_{N,X}^*(\sigma)=[c]$ and 
$u\mathrm{Ev}_{N,W}(\sigma)=\psi$
for some $u \in \mca{O}^\times$.
(Note that since $W$ is a cyclic $\mca{O}$-module, 
there exists a unit $u \in \mca{O}^\times$
such that $u^{-1}\psi \in \mathrm{Ev}_{N_0}(G_{\Omega_{N}})$.)
Let $\Sigma'$ be a finite set of 
prime ideals of $\mca{O}_K$ contained in $\Sigma$
consisting of that contained in $\Sigma$,
that dividing $\mf{n}$, 
and that ramified in $L/K$.
By the choice of the collection
$\{ \mf{l}_{\overline{K}}\colon \overline{K} 
\hookrightarrow \overline{K}_\mf{l} \}_\mf{l}$, 
there exist infinitely many primes $\mf{q}$
not contained in $\Sigma'$ 
such that the arithmetic Frobenius 
at $\mf{q}_L$ in $\Gal(L/K)$ coincides with $\sigma$.
Such primes $\mf{q}$  
are contained in $\mca{P}_{F,N}(\rho)$, 
and satisfy (i), (ii) and (iii).
\end{proof}

\subsection{Proof of Proposition \ref{propgrlev}}\label{sspfprpgrlev}

Now let us complete the proof of 
Proposition \ref{propgrlev}.
We assume that the order of $X(K,\rho)$ is finite, 
and let $E$ be the length of the $\mca{O}$-module $X$.
Fix $F \in \mca{IF}$.
We put $N> 2E$. For each pair 
$(m_1,m_2) \in \bb{Z}^2_{>0}$ 
with $m_1 \ge m_2$, 
let $\mathrm{pr}_{m_1,m_2} \colon 
\mca{O}/\pi^{m_1}\mca{O} 
\longrightarrow \mca{O}/\pi^{m_2}\mca{O}$
be the projection.

Fix $i \in \bb{Z}_{\ge 0}$.
Let $\mf{n} \in \mca{N}_{F,N}^{\mathrm{w.o}}(\rho)$
be an element satisfying $\epsilon (\mf{n})=i$, and
take any element 
$\boldsymbol{c} \in Z_{\mf{a}}\subseteq Z^{\mf{n}}$.
We denote by $Y$ the quotient $\mca{O}$-module of 
$X(K,\rho)$
by the submodule generated by the subset
$\{ \mathrm{Ev}_{N}^*(\mathrm{Fr}_{\mf{l}})
\mathrel{\vert} \mf{l} \in \mathrm{Prime}( \mf{n} ) \}$.
The $\mca{O}$-module $Y$ 
is decomposed into a direct sum
of finitely many cyclic $\mca{O}$-submodules 
of finite order. We write
\(
Y=\bigoplus_{j=1}^r \mca{O} \bar{\mf{c}}_j 
\)
and $\mca{O} \bar{\mf{c}}_j \simeq \mca{O}/\pi^{e_j}\mca{O}$.
We may assume that 
$e_j \ge e_{j+1}$ for each $j$.
By Lemma \ref{lemevsurj}, 
we may assume that there exists 
the sequence $\{ \mf{c}'_j \}_{j=1}^r \subseteq 
\mathrm{Ev}_{N,X}^* (G_{\Omega_{N}})$ 
satisfying the following property.
\begin{itemize}
\item[(A)] For any $j$, the image of $\mf{c}'_j$ 
in $Y/\sum_{\nu=1}^{j-1} \mca{O} \bar{\mf{c}}_\nu$
coincides with that of $\bar{ \mf{c} }_j$.
\end{itemize}
Take any $\mca{O}$-linear map 
$f \colon H^1(F,(\mca{O}/\pi^N\mca{O})\otimes \rho)
\longrightarrow \mca{O}/\pi^{N}\mca{O}$. 
In order to prove the first assertion of 
Proposition \ref{propgrlev},
it suffices to show that 
\begin{equation}\label{eqgrfitt}
(\mathrm{pr}_{N,N-E} \circ f)
(\kappa_{K,N}(\mf{n};\boldsymbol{c}))
\in 
\Fitt_{\mca{O},0}(Y)\cdot \mf{C}_{r+i,K,N-E}(Z).
\end{equation}
Indeed, by the definition of higher Fitting ideals, 
we can easily show the following.

\begin{lem}\label{lemFittquot}
Let $R$ be a commutative ring, 
and $M$ a finitely presented $R$-module.
Let $s \in \bb{Z}_{>0}$ be any positive integer, and $M'$ 
an $R$-submodule of $M$
generated by $s$ elements $x_1, \dots , x_s \in M$.
Then, we have
\(
\Fitt_{R,0}(M/M') \subseteq 
\Fitt_{R,s}(M).
\)
\end{lem}

We put $W_{0}:=\mca{O} \cdot \kappa_{K,N}(\mf{n};\boldsymbol{c})$. 
By Proposition \ref{Cheb}, 
there exists an element $\mf{l}_{1} \in \mca{P}_{F,N}(\rho)$ 
prime to $\mf{a}$ 
satisfying $\mathrm{Ev}_{N}^*(\mathrm{Fr}_{\mf{l}_{1}})=\mf{c}_{1}$
and $f\vert_{W_0} = u_{1}\phi^{\mf{l}_{1}}_{N}\vert_{W_0}$
for some $u_{1} \in \mca{O}^\times$.

Let us take 
$\mf{l}_{2}, \dots , \mf{l}_{r+1} \in \mca{P}_{F,N}(\rho)$ 
not dividing $\mf{a}$ 
satisfying the following properties.
\begin{enumerate}[(a)]
\item For each $j \in \{ 2, \dots, r+1 \}$,
the prime  $\mf{l}_j$ splits completely  
in $F \langle \mf{n}_{j-1} \rangle /K$, 
where we put 
$\mf{n_{j-1}}:=\mf{n} \prod_{\nu=1}^{j-1} \mf{l}_\nu$.
\item For each $j \in \{ 2, \dots, r+1 \}$, 
we have $\mathrm{Ev}_{N}^*(\mathrm{Fr}_{\mf{l}_{j}})=\mf{c}'_{j}$.
(Here, we put $\mf{c}'_{r+1}:=0$.)
\item Let $j \in \{ 2, \dots, r+1 \}$.
Put $\mf{n}_{j-1}:=\mf{n} \prod_{\nu=1}^{j-1}\mf{l}$
and $N_{j-1}:=N - \sum_{\nu= 1}^{j-1} e_\nu$.
We define 
$W_{j-1}:=\mca{O} \kappa_{K,N_{j-1}}(\mf{n}_{j-1};\boldsymbol{c})$.
Then, there exists an element $u_{j} \in \mca{O}^\times$ such that
\[
(\cdot )^{\mf{l}_{j-1},s}_{N_{j-1}}
=
u_j\pi^{e_{j-1}}\phi^{\mf{l}_{j}}_{N_{j-1}}\vert_{W_{j-1}}.
\]
\end{enumerate}
In order to take such a sequence 
$\{ \mf{l}_j \}_{j=1}^{r+1}$, 
we need to prove the following lemma.

\begin{lem}\label{lemdivES}
Let $j$ be an integer with $2 \le j \le r$, and
$\mf{l}_{2}, \dots , \mf{l}_{j}
\in \mca{P}_{F,N}(\rho)$
primes satisfying the conditions
{\rm (a)}, {\rm (b)} and {\rm (c)} above.
Then, there exists an $\mca{O}$-linear map
$\bar{\psi}_j \colon W_j \longrightarrow 
\mca{O}/\pi^{N_j}\mca{O}$
satisfying 
$(\cdot)^{\mf{l}_{j},s}_{N_j}\vert_{W_j}
=\pi^{e_j}\bar{\psi}_j$.
\end{lem}

\begin{proof}
Put $x:=\kappa_{K,N_{j-1}}(\mf{n}_{j};\boldsymbol{c})$
and $\bar{x}:=\kappa_{K,N_{j}}(\mf{n}_{j};\boldsymbol{c})$.
Let 
\[
\nu \colon \mca{O}/\pi^{N_{j}}\mca{O}
\xrightarrow{\ \simeq \ } \pi^{e_j}\mca{O}/\pi^{N_{j-1}}\mca{O}
\]
be the isomorphism of $\mca{O}$-modules defined by 
$a \mapsto a\pi^{e_j}$.
We shall define the map 
$\bar{\psi}_j \colon 
\overline{W} \longrightarrow 
\mca{O}/\pi^{N_j}\mca{O}$
by $\bar{\psi}_j(a \bar{x})
=\nu^{-1}((ax)^{\mf{l}_j,s}_{N_j})$
for any $a \in \mca{O}$.
Once such $\bar{\psi}_j$ is defined, 
the map $\psi$ clearly satisfies 
the desired condition.
In order to prove Lemma \ref{lemdivES},
it is sufficient to show that
$\bar{\psi}_j$ is well-defined.

Suppose that an element
$a\in \mca{O}$ satisfies  
$a\bar{x}= 0$.
Then, we have 
\[
\pi^{e_j}ax=
H^1(\nu\otimes \rho)(a \bar{x})=0,
\]
where
\(
H^1(\nu\otimes \rho)\colon
H^1(K,(\mca{O}/\pi^{N_{j}}\mca{O})\otimes \rho)
\longrightarrow 
H^1(K,(\mca{O}/\pi^{N_{j-1}}\mca{O})\otimes \rho)
\)
is the map induced by 
$\nu\otimes \rho \colon (\mca{O}/\pi^{N_{j}}\mca{O})\otimes \rho
\longrightarrow (\mca{O}/\pi^{N_{j-1}}\mca{O})\otimes \rho$.
Let $d \in \bb{Z}_{>0}$ be the integer such that
\(
\ann_{\mca{O}}(x;
H^1(K,(\mca{O}/\pi^{N_{j-1}}\mca{O})\otimes \rho))
= \pi^{d} \mca{O}
\).
Then we obtain 
\begin{equation}\label{eqaind-ej}
a \in \pi^{d-e_j}\mca{O}.
\end{equation}
Put $m:=\max \{ d-N_j,0 \}$.
Then, by definition, we have 
$0 \le m \le e_j$.
By Lemma \ref{lemSelsurj},
we have a surjection 
\[
\xymatrix{
H^1(\nu\otimes \rho) \colon
\mathrm{Sel}^{p\mf{n_j}}(K,(\mca{O}/\pi^{N_{j}}\mca{O})\otimes \rho)
\ar@{->>}[r] & 
\mathrm{Sel}^{p\mf{n_j}}(K,(\mca{O}/\pi^{N_{j-1}}\mca{O})\otimes \rho)
[\pi^{N_j}].
}
\]
There exists an element 
$\bar{y} \in \mathrm{Sel}^{p\mf{n_j}}
(K,(\mca{O}/\pi^{N_{j}}\mca{O})\otimes \rho))$
such that $H^1(\nu\otimes \rho)(\bar{y})=\pi^{m}x$. 
We obtain 
\(
\nu((\bar{y})^{\mf{l}_j,s}_{N_j})
=(H^1(\nu\otimes \rho)(\bar{y}))^{\mf{l}_j,s}_{N_{j-1}}
=\pi^{m}(x)^{\mf{l}_j,s}_{N_{j-1}}
\).

For each $\nu \in \{j-1,j \}$, let $Y_\nu$ be the 
$(\mca{O}/\pi^{N_{j}}\mca{O})$-submodule 
of $X_{N_j}(K,\rho)=X(K,\rho)$ generated by 
$\{ \mathrm{Ev}_{N}^*(\mathrm{Fr}_{\mf{l}})
\mathrel{\vert} \mf{l} \mid \mf{n}_\nu \}$.
Since $Y_j/Y_{j-1}$ is 
generated by $\mf{c}_j$, 
we have 
\[
\ann_{(\mca{O}/\pi^{N_{j}}\mca{O})}(Y_j/Y_{j-1})
=\ann_{(\mca{O}/\pi^{N_{j}}\mca{O})}
(\bar{\mf{c}}_j;Y)
=\pi^{e_j}\mca{O}/\pi^{N_{j}}\mca{O}.
\]
Since 
$\bar{y} \in \mathrm{Sel}^{p\mf{n_j}}
(K,(\mca{O}/\pi^{N_{j}}\mca{O})\otimes \rho)$, 
the global duality of Galois cohomology
implies that 
\(
(\bar{y})^{\mf{l}_j,s}_{N_j} \in 
\pi^{e_j} \mca{O}/\pi^{N_{j}}
\). 
(For instance, see Theorem 1.7.3.)
So, we have 
\[
\pi^{m}(x)^{\mf{l}_j,s}_{N_{j-1}} 
\in \pi^{2e_j} \mca{O}/\pi^{N_{j-1}} .
\]
Since $d-N_{j}\le m \le e_j <2e_j$, we have 
\[
(x)^{\mf{l}_j,s}_{N_{j-1}} \in 
\pi^{2e_j+ N_j-d} \mca{O}/\pi^{N_{j-1}}
=\pi^{e_j-d+ N_{j-1}} \mca{O}/\pi^{N_{j-1}}.
\]
Hence by (\ref{eqaind-ej}), 
we obtain $(ax)^{\mf{l}_j,s}_{N_{j-1}}
=a(x)^{\mf{l}_j,s}_{N_{j-1}}=0$.
This implies that the map $\bar{\psi}_j$ 
is well-defined.
\end{proof}

\begin{proof}[Proof of Proposition \ref{propgrlev}]
Let $j \in \{ 1,2,\dots, r \}$.
Suppose that we have taken 
primes $\mf{l}_{1}, \dots, \mf{l}_{j}
\in \mca{P}_{F,N}(\rho)$
satisfying (a), (b) and (c).
Let us take
the next prime $\mf{l}_{j+1}$.
By Proposition \ref{Cheb}, 
we can take a prime ideal $\mf{l}_{j+1} \in \mca{P}_{F,N}(\rho)$ 
splitting completely in $F \langle \mf{n}_j \rangle/K$, and
satisfying 
$\mathrm{Ev}_{N}^*(\mathrm{Fr}_{\mf{l}_{j+1}})=\mf{c}'_{j+1}$
and 
\(
\bar{\psi}
=
u_{j+1}\phi^{\mf{l}_{j+1}}_{N_{j}}\vert_{W_{j}}
\)
for some $u_{j+1} \in \mca{O}^\times$, 
where 
$\bar{\psi}_j \colon W_j \longrightarrow 
\mca{O}/\pi^{N_j}\mca{O}$
is the map introduced in Lemma \ref{lemdivES}.
This prime $\mf{l}_{j+1}$ clearly
satisfies (a), (b) and (c).
By using the sequence  $\{\mf{l}_j\}_j$,
we have 
\begin{align*}
(\mathrm{pr}_{N,N-E} \circ f)
(\kappa_{K,N}(\mf{n};\boldsymbol{c}))
&=u_1 \cdots u_{r+1} \pi^{e_1+ \cdots + e_r} 
\phi^{\mf{l}_{r+1}}_{N-E}
(\kappa_{K,N-E}(\mf{n}_r;\boldsymbol{c})) \\
&\in \pi^{e_1+ \cdots + e_r}\mf{C}_{i+r,K,N-E}(Z) \\
&=\Fitt_{\mca{O},0}(Y)\cdot \mf{C}_{i+r,K,N-E}(Z).
\end{align*}
Hence we obtain (\ref{eqgrfitt}), and 
complete the proof of the first assertion of 
Proposition \ref{propgrlev}.

Let us show the second assertion. 
Suppose that the ideal
$\mf{C}_{0}(K, Z)$ is equal to 
$\Fitt_{\mca{O},0}(X(K,\rho))$. 
We take an integer $N$ satisfying $N>3E$.
Since any ideal of $\mca{O}/\pi^N \mca{O}$ 
is principal, 
we have an element
$\boldsymbol{c} \in Z$ and  
a homomorphism
\[
f_0 \colon H^1(F,(\mca{O}/\pi^N\mca{O})\otimes \rho)
\longrightarrow \mca{O}/\pi^{N}\mca{O}
\]
of $\mca{O}$-modules satisfying 
\[
f_0 (\kappa_{K,N}(\mca{O}_K;\boldsymbol{c}))
\cdot \mca{O}/\pi^{N}\mca{O}
= \Fitt_{\mca{O},0}(X(K,\rho)).
\] 
We put $\mf{n}:=\mca{O}_K$, 
and $Y:=X(K,\rho)$.
We newly write
\(
X(K,\rho)=\bigoplus_{j=1}^r \mca{O} \bar{\mf{c}}_j 
\)
and $\mca{O}\bar{\mf{c}}_j \simeq \mca{O}/\pi^{e_j}\mca{O}$.
We may assume that there exists a sequence 
$\{ \mf{c}'_j \}_{j=1}^r
\subseteq 
\mathrm{Ev}_{N,X}^* (G_{\Omega_{N}})$ 
satisfying the condition (A) for the present 
$\{ \bar{\mf{c}}_j \}_{j=1}^r$.
We take a sequence of 
primes $\{ \mf{l}_j \}_{j=1}^{r+1} \subseteq \mca{P}_{F,N}(\rho)$
satisfying the conditions (a)--(c)
for the new tripe $(\mf{n}(=\mca{O}),\{\mf{c}_j \}_j,f_0)$.
Then, for any $i \in \bb{Z}$ with $0 \le i \le r$, 
we have an equality
\begin{equation}\label{eqkappa}
(\mathrm{pr}_{N_i,N-E} \circ \phi_{N_i}^{\mf{l}_{i+1}})
(\kappa_{K,N}(\mf{n_i};\boldsymbol{c}))
=u'_i \pi^{\sum_{j=i+1}^r e_j}
\phi^{\mf{l}_{r+1}}_{N-E}
(\kappa_{K,N-E}(\mf{n}_r;\boldsymbol{c}))
\end{equation}
in $\mca{O}/\pi^{N-2E}\mca{O}$,
where $u'_i$ is a unit of $\mca{O}$.
Since we assume that $\mf{C}_{0}(K,Z)$ is equal to 
$\Fitt_{\mca{O},0}(X(K,\rho))$, 
it follows from the equality (\ref{eqkappa}) for $i=0$ 
that the element 
$\phi^{\mf{l}_{r+1}}_{N-E}
(\kappa_{K,N-E}(\mf{n}_r;\boldsymbol{c}))$
belongs to $(\mca{O}/\pi^N\mca{O})^\times$.
So, by the equality (\ref{eqkappa}), we obtain 
\[
\mf{C}_{i}(K, Z) \supseteq
\Fitt_{\mca{O},i}(X(K,\rho)) 
\]
for any $i \in \bb{Z}_{\ge 0}$.
Hence by combining with the first assertion, 
we obtain the second assertion of 
Proposition \ref{propgrlev}.
\end{proof}

\subsection{Proof of Theorem \ref{thm+P}}\label{ssthmP}

Let $\psi \in \widehat{\Delta}$ 
be a character
satisfying $\psi \vert_{D_{\Delta,\mf{p}}} \ne 1$
for any prime $\mf{p}$ above $p$.
If $K_0$ contains $\mu_p$, 
we also assume that $\psi 
\ne \omega\psi^{-1}$
and $\psi \vert_{D_{\Delta,\mf{p}}} 
\ne \omega \vert_{{D_{\Delta,\mf{p}}}} $
for any prime $\mf{p}$  above $p$.
Recall that we put 
$\mca{O}_\psi:=\bb{Z}_p[\mathrm{Im} \psi]$ 
and  $X_\psi=X(\chi_{\mathrm{cyc}}\psi^{-1})$
in \S \ref{secintro}.
Here, we shall prove Theorem \ref{thm+P},
namely the assertion that 
for any $i \in \bb{Z}_{\ge 0}$, 
we have 
\[
\Fitt_{\Lambda_\psi, i}(X_\psi)+ I(\Gamma)
= \mf{C}_{i,\psi}^{\mathrm{ell}}+ I(\Gamma).
\]

First, let us recall a result related 
to the analytic class number formula.
Let $C$ be the $\bb{Z}_p[\Delta]$-subgroup of 
$E:=\mca{O}_{K_0}[p^{-1}]^\times 
\otimes_{\bb{Z}} \bb{Z}_p$
consisting of elliptic units.
Namely, $C_\psi$ is an $\bb{Z}_p[\Delta]$-submodule
generated by all the roots of unity contained in $K_0$
and ${c^{\mf{a}}_{\mf{g}}}(K;\mca{O})_\psi$ 
for any pair $(\mf{a},\mf{g})$ satisfying 
$(\mathrm{I})_1$ and $(\mathrm{I})_2$.
Note that we have a natural isomorphism 
$E_\chi \simeq (\mca{O}_{K_0}^\times 
\otimes_{\bb{Z}} \bb{Z}_p)_\psi$, 
which is written by $\mca{E}(K_0)^{\psi}$ 
in \cite{Ru3},
since $\psi \vert_{D_{\Delta,\mf{p}}}$ is non-trivial
for any prime $\mf{p}$ above $p$. 
We also note that our $C_\psi$ 
coincides with the $\mca{O}_\psi$-submodule
$\mca{C}(K_0)^{\psi}$ of $\mca{E}(K_0)^{\psi}$
in the notation of \cite{Ru3}.
Recall that we write  
$A(K_0)_{,\psi}:=X(K,\chi_{\mathrm{cyc}}\psi^{-1})$
in the notation of \S \ref{secintro}.
By Proposition \ref{propgrlev} for 
$i=0$ and $\rho=\chi_{\mathrm{cyc}}{\psi'}^{-1}$
(with general non-trivial 
$\psi' \in \widehat{\Delta}$)
combined with the analytic class number formula
described in terms of elliptic units, 
we have 
\begin{equation}\label{eqcnf}
\Fitt_{\mca{O}_\psi,0}(A(K_0)_\psi)=\mf{C}_{0}(K, Z^{\mathrm{ell}}_\psi).
\end{equation}
(See \cite{Ru3} Theorem 1.)

\begin{proof}[Proof of Theorem \ref{thm+P}]
Let $i \in \bb{Z}_{\ge 0}$. 
By Lemma \ref{lemXred} (for $\rho=1$), 
the image of $\Fitt_{\Lambda_\psi,i}(X_\psi)$ 
in $\mca{O}_\psi=\Lambda_\psi/I(\Gamma)$ 
coincides with $\Fitt_{\mca{O}_\psi,0}(A(K_0)_\psi)$.
By Corollary \ref{corC_iIm}, 
the image of $\mf{C}_{i,\psi}^{\mathrm{ell}}$
in $\Lambda_\psi/I(\Gamma)$ 
coincides with 
$\bigcap_{F \in \mca{IF}}\mf{C}_{i}^{\mathrm{ell},F}(K)_\psi$.
Note that 
by the second assertion of Proposition \ref{propgrlev}
and (\ref{eqcnf}), we have
\(
\Fitt_{\mca{O}_\psi,i}(A(K_0)_\psi)=
\mf{C}_{i}^{\mathrm{ell},F}(K)_\psi
\)
for any $F \in \mca{IF}$. 
Hence the image of $\mf{C}_{i,\psi}^{\mathrm{ell}}$
in $\Lambda/I(\Gamma)$ 
coincides with that of 
of $\Fitt_{\Lambda_\psi,i}(X_\psi)$.
\end{proof}

\section{Proof of Theorem \ref{thmmainthm}}\label{secpf}
Here, let us complete 
the  proof of Theorem \ref{thmmainthm}.
Let $\psi \in \widehat{\Delta}$ 
be a character
satisfying $\psi \vert_{D_{\Delta,\mf{p}}} \ne 1$
for any prime $\mf{p}$ above $p$.
If $K_0$ contains $\mu_p$, 
we also assume that $\psi 
\ne \omega\psi^{-1}$
and $\psi \vert_{D_{\Delta,\mf{p}}} 
\ne \omega \vert_{{D_{\Delta,\mf{p}}}} $
for any prime $\mf{p}$  above $p$.
We put $\mca{O}_\psi:=\bb{Z}_p[\mathrm{Im} \psi]$, 
and $\Lambda_\psi:=\mca{O}_\psi[[\Gamma ]]$. 
Then, for any $i \in \bb{Z}_{\ge 0}$,
we have defined an ideal
$\mathfrak{C}_{i,\psi}^{\mathrm{ell}}$
of $\Lambda_\psi$ in Definition \ref{theideal2}. 
Our goal is the proof of the inequality
$\mathfrak{C}_{i,\psi}^{\mathrm{ell}} \prec
\Fitt_{\Lambda_\psi, i}(X_\psi)$.

In \ref{ssonevar}, we prove 
Theorem \ref{thmmainthm} for one variable cases, 
and in \ref{sstwovar}, 
Theorem \ref{thmmainthm} for two variable cases.

\subsection{One variable cases}\label{ssonevar}
In this subsection, suppose that $\Gamma \simeq \bb{Z}_p$, 
and fix a topological generator $\gamma$ of $\Gamma$.
Let $\mca{L}$ be 
a finite extension field  of $\bb{Q}_p$,
and $\mca{O}$ the ring of integers in $\mca{L}$.
We identify the Iwasawa algebra 
$\Lambda=\mca{O}[[\Gamma]]$
with the ring $\mca{O}[[T]]$
of formal power series 
by the isomorphism 
$\Lambda \simeq \mca{O}[[T]]$
of $\mca{O}$-algebras 
given by $\gamma \mapsto 1+T$.
Let $\rho \colon \mca{G}
\longrightarrow \mca{O}^\times$
be a character satisfying the assumptions (C1), (C2) and (C3).
We put 
\(
X:=X(K_\infty^{\Delta}/K, \rho)_{\mca{O}}
\).
We fix a finite set $\Sigma$
of primes of $\mca{O}_K$
containing $\Sigma(\rho)$.
Let $Z$ be a non-empty finite subset of 
$\mathrm{ES}_{\Sigma}(K_\infty^\Delta; \rho)_\mca{O}$.
Theorem \ref{thmmainthm} for one variable cases 
follows from 
the following theorem on general characters, 
which is a goal of this subsection.

\begin{thm}\label{thmonevar}
Suppose that the Iwasawa
$\mu$-invariant of $X$ is $0$.
Let $i \in \bb{Z}_{\ge 0}$, and  
$\mf{P}$ a height one prime ideal of $\Lambda$
containing $\Fitt_{\Lambda,i}(X)$. 
We define two integers 
$\alpha=\alpha_i(\mf{P})$ and $\beta=\beta_{i}(\mf{P})$ by
$\Fitt_{\Lambda_{\mf{P}},i}
(X_{\mf{P}} )
 =\mf{P}^{\alpha}\Lambda_{\mf{P}}$ 
and 
$\mf{C}_{i}(K_\infty^\Delta /K, Z)\Lambda_{\mf{P}}
 =\mf{P}^{\beta}\Lambda_{\mf{P}}$
respectively.
Then, we have 
$\beta_{i}({\mf{P}})\ge \alpha_{i}({\mf{P}})$.
\end{thm}

\begin{proof}
We shall prove Theorem \ref{thmonevar}
by using the method developed in \cite{MR} \S 5.3. 
Since the $\mu$-invariant of $X$ is $0$,
we may assume that 
$\mf{P} \ne \pi \Lambda$, namely, 
the ideal $\mf{P}$
is a principal ideal 
generated by 
an irreducible distinguished polynomial
$f(T) \in \mca{O}[T]$.
Let $c \in \overline{\mca{L}}$ be 
a root of $f(T)$, and put $\mca{L}':=\mca{L}(c)$. 
By Lemma \ref{lemOO'}, 
we may replace $\mca{L}$ by $\mca{L}'$, 
and let $f(T)=T-c$.
For each $n \in \bb{Z}_{>0}$, 
we put $f_n(T)=T-c-\pi^n$, and 
$\mf{P}_n:= f_n \Lambda$.

\begin{dfn}
Let $\{ x_n \}_{n \in \bb{Z}_{\ge 0}}$ and 
$\{ y_n \}_{n \in \bb{Z}_{\ge 0}}$
be sequences of real numbers.
We write $x_n \prec y_n$ 
if and only if 
$\limsup_{n \to \infty} (y_n -x_n) <\infty$.
We write $x_n \sim y_n$ if and only if 
we have $x_n \prec y_n$  and 
$y_n \prec x_n$.
\end{dfn}

By the structure theorem of 
finitely generated torsion $\Lambda$-modules, 
we can deduce the following lemma immediately.  

\begin{lem}\label{lemasymineq}
Let $Y$ be a finitely generated $\Lambda$-module, 
and $C$ a non-negative integer satisfying
$\Fitt_{\Lambda_{\mf{P}},i}(Y)=\mf{P}^C \Lambda_{\mf{P}}$.
For each $n \in \bb{Z}_{>0}$, we define 
$C(n) \in \bb{Z}_{\ge 0}$ by
\[
\Fitt_{\mca{O},i}(Y \otimes_{\Lambda} \Lambda/\mf{P}_n)
=\pi^{C(n)}\mca{O}.
\]
Then, we have $C(n) \sim Cn$.
\end{lem}

Let $\rho'_n\colon 
\Gamma \longrightarrow 1+ \pi \mca{O}$
the unique continuous character 
satisfying $\rho'_n(\gamma)=1+c+\pi^n$.
By Lemma \ref{lemXred} (i), 
we have
\(
X\otimes_{\Lambda} (\Lambda /\mf{P}_n )
\simeq X(K,\rho\rho'_n)
\)
and by Lemma \ref{lemredCitodvr}, 
the image of 
$\mf{C}_i(K^\Delta_\infty/K, Z)$
in $\Lambda/(T-c)=\mca{O}$ 
is contained in 
$\mf{C}_i(K, Z \otimes \rho'_n)$. 
We put 
\begin{align*}
\Fitt_{\Lambda, i}(X(K,\rho\rho'_n))
&=\pi^{\bar{\alpha}(n)}, \\
\mf{C}_i(K, Z \otimes \rho'_n)
&=\pi^{\bar{\beta}(n)}.
\end{align*}
By Proposition \ref{propgrlev} and 
Lemma \ref{lemasymineq},
we have 
\[
\alpha_{i}({\mf{P}}) n \sim 
\bar{\alpha}(n) \le 
\bar{\beta}(n)
\prec 
\beta_{i}({\mf{P}})n.
\]
Hence we obtain 
$\alpha_{i}({\mf{P}})
\le \beta_{i}({\mf{P}})$.
This completes the proof.
\end{proof}

\subsection{Two variable cases}\label{sstwovar}

Here, let us consider the two-variable cases, that is, 
the case when
$\Gamma \simeq \bb{Z}_p^2$.
Here, we put $\mca{O}:=\mca{O}_\psi=\bb{Z}_p[\mathrm{Im} \psi]$,
and 
\(
X:=X_\psi=X(K_\infty^\Delta /K, \chi_{\mathrm{cyc}}\psi^{-1}).
\)
Let $Z:=Z^{\mathrm{ell}}_\psi$ be the set of 
Euler systems of elliptic units
introduced in Definition \ref{defEU},
and we write
\(
\mf{C}_{i}:=\mathfrak{C}_{i,\psi}^{\mathrm{ell}}
= \mf{C}_i(K_\infty^\Delta /K, Z)_{\mca{O}}
\)
for each $i \in \bb{Z}_{\ge 0}$.
In order to prove Theorem \ref{thmmainthm}, 
it is sufficient to show that
\(
\mf{C}_i
\prec \Fitt_{\Lambda,i}(X)
\)
for any $i \in \bb{Z}_{\ge 0}$.

We fix topological generators $\gamma_1$ and $\gamma_2$ of 
the group $\Gamma \simeq \bb{Z}_p^2$.
We define  
\[
\mca{E}:=\{ (a_1,a_2) \in \bb{Z}_p^2 \mathrel{\vert} 
a_1\bb{Z}_p+a_2\bb{Z}_p=\bb{Z}_p \}.
\]
By definition, for any $(a_1,a_2) \in \mca{E}$,
the closed subgroup $H$ topologically generated by 
$\gamma_1^{a_1}\gamma_2^{a_2}$  satisfies 
$H \simeq \bb{Z}_p$ and 
$\Gamma/H \simeq \bb{Z}_p$.
Conversely, a closed subgroup $H$
satisfying $H \simeq \bb{Z}_p$ and 
$\Gamma/H \simeq \bb{Z}_p$
is generated by an element 
$\gamma_1^{a_1}\gamma_2^{a_2}$
with $(a_1,a_2) \in \mca{E}$.

\begin{lem}\label{lemforlemforred}
Let $(a_1,a_2)$ and $(b_1,b_2)$ 
be any elements of $\mca{E}$.
Suppose that 
the ideal $(\gamma_1^{a_1}\gamma_2^{a_2} -1,\pi)$
of $\Lambda$ coincides with 
$(\gamma_1^{b_1}\gamma_2^{b_2} -1,\pi )$.
Then, there exists an element $e \in \bb{Z}_p^\times$
satisfying $(b_1,b_2)= (ea_1,ea_2)$.
\end{lem}

\begin{proof}
Take topological generators $\gamma'_1,\gamma'_2$ of $\Gamma$
satisfying $\gamma'_1=\gamma_1^{a_1}\gamma_2^{a_2}$.
Let $(b'_1,b'_2) \in \mca{E}$ be the unique element
satisfying ${\gamma'_1}^{b'_1}{\gamma'}_2^{b'_2}
=\gamma_1^{b_1}\gamma_2^{b_2}$.
Then, we have 
\[
\Lambda/(\gamma'_1 -1 ,\pi)
=\Lambda/(\gamma_1^{a_1}\gamma_2^{a_2} -1 , 
\gamma_1^{b_1}\gamma_2^{b_2} -1,\pi )
=\Lambda/(\gamma'_1-1, {\gamma'_2}^{b'_2}-1,\pi).
\]
Let $\Gamma_2$ be the closed subgroup of $\Gamma$
topologically generated by $\gamma'_2$.
Then, we have  
\[
\Lambda/(\gamma'_1 -1,\pi )
\simeq k[[\Gamma_2]]
\simeq k[[x]];\ \gamma'_2 \longleftrightarrow 1+x.
\] 
On the other hand, 
the $\mca{O}$-algebra
\(
\Lambda/(\gamma'_1-1, {\gamma'_2}^{b'_2}-1,\pi)
\simeq k[[\Gamma_2]]/({\gamma'_2}^{b'_2}-1)
\)
is isomorphic to $k[[\Gamma_2]]$
if and only if $b'_2=0$. 
Hence we obtain 
$(b_1,b_2)= (b'_1a_1,b'_1a_2)$.
Moreover, since $(b'_1,0) \in \mca{E}$, 
we also have $b'_1 \in \bb{Z}_p^\times$.
This completes the proof of 
Lemma \ref{lemforlemforred}.
\end{proof}

\begin{lem}\label{lemforred}
Let $\mca{I}$ and $\mca{J}$ 
be ideals of $\Lambda$ 
whose heights at least two.
Then, there exists an element
$(a_1,a_2; u) \in \mca{E} \times (1 + \pi \mca{O})$
such that 
the images of $\mca{I}$ and $\mca{J}$ 
in $\Lambda /(\gamma_1^{a_1}\gamma_2^{a_2} -u ) $
are ideals of 
$\Lambda /(\gamma_1^{a_1}\gamma_2^{a_2} -u ) $
whose heights are at least two.
\end{lem}

\begin{proof}
In order to prove Lemma \ref{lemforred}, 
we shall introduce some notations.
Let $I$ be an ideal of $\Lambda$.
We denote by $\mathrm{Assoc}(I)$
be the set of associated ideals of
the $\Lambda$-module $\Lambda/I$
whose height is two.
Since the ring $\Lambda$ is Noetherian, 
the sets $\mathrm{Assoc}(I)$
is a finite set consisting of 
all minimal prime ideals containing $I$. 
We define the subset  $\mathrm{Assoc}^0(I)$
of $\mathrm{Assoc}(I)$ 
to be the collection of all the elements 
written in the form $(\gamma_1^{a_1}\gamma_2^{a_2}-1, \pi)$
for some $(a_1,a_2) \in \mca{E}$.

For each $\mf{p} \in \mathrm{Assoc}^0(\mca{I}) \cup 
\mathrm{Assoc}^0(\mca{I})$, 
we fix an element $(a_1(\mf{p}),a_2(\mf{p})) \in \mca{E}$ 
satisfying $\mf{p} =
(\gamma_1^{a_1(\mf{p})}\gamma_2^{a_2(\mf{p})}-1, \pi)$.
Since $\# \bb{P}^1(\bb{Z}_p)= \infty$, we can 
take an element $(a_1,a_2 ) \in \mca{E}$
not contained in 
\[
\bigcup_{\mf{p} \in \mathrm{Assoc}^0(\mca{I}) \cup 
\mathrm{Assoc}^0(\mca{J})} 
\bb{Z}_p^\times \cdot
(a_1(\mf{p}),a_2(\mf{p})).
\]
Let $\mf{q} \in \mathrm{Assoc}(\mca{I}) \cup 
\mathrm{Assoc}(\mca{J})$
be any element, and put 
\[
S(\mf{q}):=\{ u \in 1+ \pi \mca{O}
\mathrel{\vert} 
\gamma_1^{a_1}\gamma_2^{a_2}-u \in \mf{q}
\}.
\]
Let us show the following claim.

\begin{claim}\label{claimSle1}
We have  $\# S(\mf{q}) \le 1$.
\end{claim}

\begin{proof}[Proof of Claim \ref{claimSle1}]
Suppose that there exist 
two elements $u_0,u_1 \in 1+\pi\mca{O}$ satisfying 
$\gamma_1^{a_1}\gamma_2^{a_2}-u_i \in \mf{q}$ for 
each $i=1,2$.
Then we have $u_1 -u_0 \in \mf{q}$.
If $u_1 \ne u_0$, then we have
$\mf{q} \in \mathrm{Assoc}^0(\mca{I}) \cup 
\mathrm{Assoc}^0(\mca{J})$. 
However, by Lemma \ref{lemforlemforred}
and the choice of $(a_1,a_2) \in \mca{E}$,  
the prime ideal $\mf{q}$ never
belongs to $\mathrm{Assoc}^0(\mca{I}) \cup 
\mathrm{Assoc}^0(\mca{J})$.
Hence we deduce that $u_1 = u_0$. 
\end{proof}

Let us complete the proof of Lemma \ref{lemforred}.
We put 
\[
S:= \bigcup_{\mf{q} \in \mathrm{Assoc}(\mca{I}) \cup 
\mathrm{Assoc}(\mca{J})} 
S(\mf{q}).
\]
It follows from Claim \ref{claimSle1}
that $S$ is a finite set.
Let $u \in (1+ \pi \mca{O}) \setminus S$
be any element. 
Then, by definition, for any $\mf{q} \in \mathrm{Assoc}(\mca{I}) \cup 
\mathrm{Assoc}(\mca{J})$, 
we have $\gamma_1^{a_1}\gamma_2^{a_2}-u \notin \mf{q}$. 
This implies that
the images of $\mca{I}$ and $\mca{J}$ 
in $\Lambda /(\gamma_1^{a_1}\gamma_2^{a_2} -u ) $
are ideals of 
$\Lambda /(\gamma_1^{a_1}\gamma_2^{a_2} -u ) $
whose heights are at least two.
Hence we obtain Lemma \ref{lemforred}.
\end{proof}

\begin{proof}[Proof of Theorem \ref{thmmainthm}]
Let $i \in \bb{Z}_{\ge 0}$.
We fix  $F,G \in \Lambda$
satisfying 
$\Fitt_{\Lambda,i(X_\psi)} \sim F \Lambda$
and $\mf{C}_{i} \sim G \Lambda$.
By the inequality (\ref{eqoh1thm})
following from \cite{Oh1} Theorem 1.1,
there exists an element $A \in \Lambda$
satisfying $F=AG$.
Put $\mca{I}:=F^{-1}\Fitt_{\Lambda,i(X_\psi)}$
and $\mca{I}:=G^{-1}\mf{C}_{i,\psi}$.
Let $(a_1,a_2;u) \in \mca{E} \times (1+ \pi \mca{O})$
be as in the the assertion of Lemma \ref{lemforred},
and let $\bar{F}$ (resp.\ $\bar{G}$ and $\bar{A}$) 
denotes the image of $F$ (resp.\ $G$ and $A$) in 
$\overline{\Lambda}:=
\Lambda/(\gamma_1^{a_1}\gamma_2^{a_2}-u)$.
We have 
$\mca{I}\overline{\Lambda} 
\sim \mca{J}\overline{\Lambda} 
\sim  \overline{\Lambda}$.
We fix topological generators 
$\gamma'_1,\gamma'_2$ of $\Gamma$
satisfying $\gamma'_1=\gamma_1^{a_1}\gamma_2^{a_2}$.
There exists a unique continuous character
$\rho' \colon \Gamma \longrightarrow 1+\pi \mca{O}
\subseteq \mca{O}^\times$
satisfying $\psi'(\gamma'_1)=u$ and $\psi'(\gamma'_2)=1$.
By Lemma \ref{lemXred} (ii) and
Lemma \ref{lemredCi} imply that we have 
\begin{align*}
\bar{F}\overline{\Lambda} & \sim
\Fitt_{\overline{\Lambda},i}
(X(K_\infty^{\Delta\times \Gamma_2}/K, 
\chi_{\mathrm{cyc}}\psi^{-1}\rho')), \\
\bar{G}\overline{\Lambda}
& \prec
\mf{C}_i(K_\infty^{\Delta\times \Gamma_2}/K, 
Z\otimes \rho') .
\end{align*}
By Theorem \ref{thmonevar} for 
$\rho:=\chi_{\mathrm{cyc}}\psi^{-1}\rho'$, 
there exist an element 
$\bar{B} \in \overline{\Lambda}$
satisfying 
$\bar{G}=\bar{B} \bar{F}$.
So, we obtain 
$\bar{G}=\bar{B}\bar{A}\bar{G}$.
Since $\overline{\Lambda}$ is 
an integral domain, we have 
$\bar{B}\bar{A}= \bar{1}$.
This implies that $A \in \Lambda^\times$.
Hence $F\Lambda= G \Lambda$.
\end{proof}

\begin{rem}
In Ochiai's article \cite{Oc},
he used  linear elements (of the ring of power series)
in the specialization arguments.
Instead of linear elements, 
we use continuous characters of 
$\Gamma \simeq \bb{Z}_p^2$
in order to keep the Iwasawa algebra 
to be the completed group ring 
after the specialization.
\end{rem}

\begin{rem}
The inequality (\ref{eqoh1thm}) 
allows us to simplify the arguments 
in this subsection significantly.
In general, 
without the inequality like (\ref{eqoh1thm}), 
we need to take infinitely many specializations
in order to reduce a two variable problem to 
one variable problems. 
(For instance, see \cite{Oc} Proposition 3.6.)
However, we have just seen that 
thanks to the inequality (\ref{eqoh1thm}),
we may take {\em only one} good height one prime 
in order to reduce the problem in the two-variable case
to that in the one-variable case.
\end{rem}

\section{Burns--Kurihara--Sano's ideals and $\mf{C}_i^{\mathrm{ell}}$}\label{seccomparison}

All the notation of this section follows 
\S \ref{secintro}.
We fix an imaginary quadratic field $K$, 
and the extension $K_\infty /K$.

Under the assumption of 
the equivariant Tamagawa number conjecture
in a certain form,
by using Rubin--Stark elements, 
Burns, Kurihara and Sano
constructed certain  ideals 
which is equal to the higher Fitting ideals
of an $S$-truncated, $T$-modified Selmer group.
In this section, we compare our ideals 
$\mf{C}_i^{\mathrm{ell}}(F/K)$ with 
the ideals constructed by 
Burns--Kurihara--Sano.
When $p$ is a prime number splitting in $K/\bb{Q}$, 
and not dividing $\# \mathrm{Cl}(K)$, 
we shall prove that 
$\mf{C}_i^{\mathrm{ell}}(F/K)_\psi$
coincides with the $\psi$-part of 
Burns--Kurihara--Sano's ideal 
for certain characters $\psi \in \widehat{\Delta}$.

In \S \ref{ssconj}, we introduce
the $p$-part of 
three conjectures 
(restricted to our setting): 
the leading term conjecture $\boldsymbol{\mathrm{LT}}(M/K)_p$, 
the Rubin--Stark conjecture $\boldsymbol{\mathrm{RS}}(M/K,S,T,V)_p$ and 
the Mazur--Rubin--Sano conjecture $\boldsymbol{\mathrm{MRS}}(L/M/K,T,W)_p$, 
where $M$ and $L$ are certain abelian extension fields of $K$
satisfying $M \subseteq L$, and 
$S,T,V$ and $W$ are subsets of $P_K$ satisfying certain conditions. 
Note that $\boldsymbol{\mathrm{LT}}(M/K)_p$ is a conjecture  
equivalent to the $p$-part of 
the equivariant Tamagawa number conjecture 
(\cite{BF} Conjecture 4 (iv))
for the pair $(h^0(\mathrm{Spec}(M), \bb{Z}[\Gal(M/K)]))$. 
As we review later, 
Burns, Kurihara and Sano proved that
the conjectures $\boldsymbol{\mathrm{RS}}(M/K,S,T,V)_p$ and 
$\boldsymbol{\mathrm{MRS}}(L/M/K,T,W)_p$ 
follow from $\boldsymbol{\mathrm{LT}}(L/K)_p$.
(See \cite{BKS} Theorem 5.11 and
Theorem 5.15.)
In \S \ref{ssETNC/K}, we recall 
Bley's result (\cite{Bl} Theorem 4.2) 
on the equivariant 
Tamagawa number conjecture over
imaginary quadratic fields.
In \S \ref{ssFitBKS}, we recall 
the results (\cite{BKS} Corollary 1.7) 
on higher Fitting ideals by 
Burns, Kurihara and Sano.
In \S \ref{sspfC=Theta}, 
we compare our ideals 
$\mf{C}_i^{\mathrm{ell}}(F/K)$ with 
the ideals constructed by 
Burns--Kurihara--Sano.

\subsection{Equivariant Tamagawa number conjecture and related conjectures}\label{ssconj}

Here, we set the general notation 
used in this section.
Let $M/K$ be a finite abelian extension.
We put $G:=\Gal(M/K)$.
For any subset $Q$ of $P_K$, 
we denote by $Q_M$ the set of all places of $M$
lying above a place contained in $Q$.
Let $S$ be a finite set of places of $K$ containing
all infinite places and all ramified places in $M/K$.
Let $T$ be a non-empty finite set 
of finite places of $K$
such that 
\[
\mca{O}_{M,S,T}^\times :=\{ 
u \in \mca{O}_{M,S}^\times \mathrel{\vert}
u \equiv 1 \ \mathrm{mod}\ w \
\text{for all $w \in T_M$}
\}
\] 
is a torsion free $\bb{Z}$-module.
We denote by $\mathrm{Cl}^T_S(M)$ 
the quotient group 
of the ray class group of $M$ modulo $\prod_{w \in T_M} w$
by the subgroup generated by all the classes 
of the prime ideals in $S_M$.

We define $Y_{M,S}:=\bigoplus_{w \in S_M} \bb{Z} \cdot w$ 
to by a free abelian group generated by $S_M$. 
We put 
\(
X_{M,S}:=\Ker (Y_{M,S} \xrightarrow{\Sigma} \bb{Z})
\), 
where $\Sigma$ is the total sum of coefficients.
(Do not confuse $X_{M,S}$ with the Iwasawa module $X_\psi$.)
The abelian groups $X_{M,S}$ and $Y_{M,S}$
have $\bb{Z}[G]$-module structure naturally.
After taking $\bb{R} \otimes_{\bb{Z}} (-)$, 
we have an isomorphism 
\[
\lambda_{M,S}\colon \bb{R}\mca{O}_{M,S,T}^\times 
\xrightarrow{\ \simeq \ } \bb{R}X_{M,S};\ 
u \longmapsto - \sum_{w \in S_M} \log \left| u \right|_w
\] 
called the regulator map.

Let $A$ be a commutative ring.
We denote by $\mca{P}(A)$
the category of graded invertible $A$-modules 
in the sense of \cite{KM} Chapter I.
We denote by $\mca{D}^{pis}(A)$
be the subcategory of the derived category
$\mca{D}(\mca{Mod}_R)$ of $R$-module
whose objects are perfect complexes, and
whose morphisms are quasi-isomorphisms.
In \cite{KM}, the determinant functor
\(
\det_A \colon \mca{D}^{pis}(A) \longrightarrow 
\mca{P}(A).
\)
is defined.
For details, see \cite{KM} Chapter I, 
in particular Definition 1.

Let us recall the definition of 
$L$-functions which we consider.
Let $\widehat{G}:=\Hom(G, \overline{\bb{Q}})$
be any character, 
and we regard $\chi$ as a $\bb{C}$-valued character
via the fixed embedding 
$\overline{\bb{Q}} \hookrightarrow \bb{C}$.
Then, we define 
\[
L_{K,S,T}(\chi,s):=
\prod_{v \in T}(1-\chi(\Fr_v)Nv^{1-s})
\prod_{v \notin S}(1-\chi(\Fr_v)Nv^{-s})^{-1}.
\]
Note that this product absolutely converge 
in the domain defined by $\mathrm{Re}(s)>1$.
The function 
$L_{K,S,T}(\chi,s)$ 
can be meromorphically continued to
the whole plane $\bb{C}$, and holomorphic 
on $\bb{C} \setminus \{ 1 \}$.
Let $r_{\chi,S}$ be the vanishing order of
$L_{K,S,T}(\chi,s)$ at $s=0$.
For each positive integer $r \le r_{\chi,S}$, 
we define
\[
L^{(r)}_{K,S,T}(\chi,0):= \lim_{s \to 0} 
s^{r_{\chi,S}} L_{K,S,T}(\chi,s).
\]
Then, we put
\[
\theta^{(r)}_{M/K,S,T}(0):= \sum_{\chi \in \widehat{G}}
L^{(r)}_{K,S,T}(\chi^{-1},0) e_\chi \in \bb{C}[G],
\]
where $e_\chi$
is the idempotent of $\bb{C}[G]$ 
corresponding to 
the $\chi$-component.
We define `the leading coefficient' 
$\theta^*_{M/K,S,T}(0)$ at $s=0$ by
\[
\theta^*_{M/K,S,T}(0):= \sum_{\chi \in \widehat{G}}
L^{(r_{\chi,S})}_{K,S,T}(\chi^{-1},0) e_\chi \in \bb{C}[G].
\]
Note that 
$\theta^{(r)}_{M/K,S,T}(0)$ and 
$\theta^*_{M/K,S,T}(0)$ belongs to 
$\bb{R}[G]$.

\subsubsection{Leading term conjecture}
First, we introduce a variant of Tamagawa number conjecture called
the leading term conjecture.

Let $R\Gamma_{c,T}((\mca{O}_{M,S})_{\mca{W}},\bb{Z})$
be the complex of $\bb{Z}[G]$-modules called 
the `Weil-\'etale cohomology' complex 
in the sense of \cite{BKS} Proposition 2.4.
We define 
\[
R\Gamma_{T}((\mca{O}_{M,S})_{\mca{W}},\bb{G}_m)
:= R\Hom(
R\Gamma_{c,T}((\mca{O}_{M,S})_{\mca{W}},\bb{Z}),
\bb{Z})[-2].
\]
Here, we endow this complex 
with the contragredient action of $G$.
In our setting, the complex 
$R\Gamma_{T}((\mca{O}_{M,S})_{\mca{W}},\bb{G}_m)$
is  perfect, 
and concentrated in degrees zero and one:
\[
H^i_T((\mca{O}_{M,S})_{\mca{W}},\bb{G}_m) =
\begin{cases}
\mca{O}_{M,S,T}^\times & (i=0), \\
\mca{S}_{S,T}^{\mathrm{tr}}(\bb{G}_m/K) & (i=1),
\end{cases}
\]
where $\mca{S}_{S,T}^{\mathrm{tr}}(\bb{G}_m/M)$
is a $\bb{Z}[G]$-module which has a canonical exact sequence
\[
0 \longrightarrow 
\mathrm{Cl}^T_S(M) 
\longrightarrow 
\mca{S}_{S,T}^{\mathrm{tr}}(\bb{G}_m/M)
\longrightarrow X_{M,S} \longrightarrow 0.
\]
(For details, 
see \cite{BKS} Proposition 2.4 (iii), (iv) and 
Remark 2.7.)
In particular, after taking $\bb{R} \otimes_{\bb{Z}} (-)$, 
we have the regulator map
\[
\lambda_{M,S} \colon 
\bb{R}H^0_T((\mca{O}_{M,S})_{\mca{W}},\bb{G}_m) 
=\bb{R}\mca{O}_{M,S,T}^\times
\xrightarrow{\ \simeq \ }
\bb{R}X_{M,S}=
\bb{R}H^1_T((\mca{O}_{M,S})_{\mca{W}},\bb{G}_m). 
\]
By using this isomorphism, we can define the map
\[
\vartheta_{\lambda_{M,S}} \colon 
\bb{R}\det_G(R\Gamma_{T}((\mca{O}_{M,S})_{\mca{W}},\bb{G}_m))
\longrightarrow \bb{R}[G]  
\]
to be the composite 
\begin{align*}
\bb{R}\det_G(R\Gamma_{T}((\mca{O}_{M,S})_{\mca{W}},\bb{G}_m))
\xrightarrow{\ \simeq \ } & \bigotimes_{j \in \bb{Z}} 
\det_{\bb{R}[G]}^{(-1)^j}
\bb{R}H^j_T((\mca{O}_{M,S})_{\mca{W}},\bb{G}_m) \\
=\hspace{-1mm} =&  \det_{\bb{R}[G]}
\bb{R}H^0_T((\mca{O}_{M,S})_{\mca{W}},\bb{G}_m) \\
& \hspace{3mm} 
\otimes \det_{\bb{R}[G]}^{-1}
\bb{R}H^1_T((\mca{O}_{M,S})_{\mca{W}},\bb{G}_m) \\
\xrightarrow{\ \simeq \ } 
& \det_{\bb{R}[G]}
\bb{R}H^1_T((\mca{O}_{M,S})_{\mca{W}},\bb{G}_m) \\
& \hspace{3mm} 
\otimes \det_{\bb{R}[G]}^{-1}
\bb{R}H^1_T((\mca{O}_{M,S})_{\mca{W}},\bb{G}_m) \\
\xrightarrow{\ \simeq \ } 
& \bb{R}[G],
\end{align*}
where we regard $\bb{R}[G]$ as a module 
of degree zero.

Let 
$z_{M/K,S,T} \in 
\bb{R}\det_G(R\Gamma_{T}((\mca{O}_{M,S})_{\mca{W}},\bb{G}_m))$
be the unique element satisfying 
\[
\vartheta_{\lambda_{M,S}}(z_{M/K,S,T})=\theta^*_{M/K,S,T}(0).
\]
The element $z_{M/K,S,T}$
is called the zeta element of $\bb{G}_m$
for $(M/K,S,T)$. 
(See \cite{BKS} Definition 3.5.)
The following conjecture $\boldsymbol{\mathrm{LT}}(M/K)$ 
is the $p$-part of the leading term conjecture.

\begin{conj}[\cite{BKS} Conjecture 3.6, $\boldsymbol{\mathrm{LT}}(M/K)_p$] 
We have
\[
\bb{Z}_{(p)}[G] \cdot z_{M/K,S,T}=
\bb{Z}_{(p)}\det_{\bb{Z}[G]}
(R\Gamma_T((\mca{O}_{M,S})_{\mca{W}},\bb{G}_m)). 
\]
\end{conj}

\begin{rem}
The validity of $\boldsymbol{\mathrm{LT}}(M/K)_p$
does not depends on $S$ and $T$.
(See \cite{BKS} Proposition 3.4.)
\end{rem}

\begin{rem}
If $\boldsymbol{\mathrm{LT}}(M/K)_p$ holds
then $\boldsymbol{\mathrm{LT}}(M'/K')_p$ also holds
for any extension $M'/K'$ 
satisfying $K \subseteq K' \subseteq M' \subseteq M$.
(See \cite{BF} Proposition 4.1.)
\end{rem}

\subsubsection{Rubin--Stark conjecture}

We denote by $P_K(M)$ be the set of all the places of $K$
splitting in $M/K$.
We fix a subset $V=\{ v_1, \dots, v_r \} 
\subseteq S \cap P_K(M)$, and 
put $S=\{ v_0, v_1, \dots, v_r \}$.
For each $v_i \in S$, we fix a place $w_i \in S_M$ 
above $v_i$.
The regulator map $\lambda_{M,S}$ induces an isomorphism
\[
\wedge^r\lambda_{M,S} \colon
\bb{R} \bigwedge_{\bb{R}[G]}^r \mca{O}_{M,S,T}^\times
=\bigwedge_{\bb{R}[G]}^r \bb{R}\mca{O}_{M,S,T}^\times
\longrightarrow \bigwedge_{\bb{R}[G]}^r \bb{R}X_{M,S}
=\bb{R}\bigwedge_{\bb{R}[G]}^r X_{M,S}.
\]
We define the $r$-th order Rubin-Stark element 
to be the unique element  
\[
\epsilon^V_{M/K,S,T} \in
\bb{R} \bigwedge_{\bb{R}[G]}^r \mca{O}_{M,S,T}^\times
\]
such that
\[
(\wedge^r\lambda_{M,S})
(\epsilon^V_{M/K,S,T})
=\theta^{(r)}_{M/K,S,T}(0) 
\bigwedge_{i=1}^r(w_i - w_0)
\in \bb{R}\bigwedge_{\bb{R}[G]}^r X_{M,S}.
\]

In order to review the statement of
Rubin--Stark conjecture, 
we need to introduced by the notion of 
``exterior power biduals".

\begin{dfn}[\cite{BS} Definition 2.1.]\label{defEPB}
Let $R$ be a commutative ring, 
and  $X$ an $R$-module.
Then for each $i \in \bb{Z}$, we define
\[
\bigcap_{R}^i X:= 
\Hom_R \left( 
\bigwedge_{R}^i 
\Hom_R(X,R) , R 
\right)
\]
The $R$-module
$\bigcap_{R}^i X$
is called {\em the $i$-th exterior biduals of $X$}. 
Note that we have a natural map
\[
\xi_X^i \colon 
\bigwedge_{R}^i X \longrightarrow \bigcap_{R}^i X;\ 
\bigwedge_{\nu=1}^i x_\nu \longmapsto \left( 
\bigwedge_{\nu=1}^i \Phi_\nu
\mapsto \det (\Phi_\mu(x_\nu))_{1 \le \mu,\nu \le i}
\right).
\]
\end{dfn}

\begin{rem}
Let $R$ be a Noetherian commutative ring, 
and  $X$ a finitely generated $R$-module.
If $X$ is a projective $R$-module, 
then the map $\xi_X^i$ becomes an isomorphism.
{\rm (See \cite{BS} Lemma A.1.)}
In particular, If $R$ isomorphic to the group ring 
$\bb{Z}[H]$ of an abelian group $H$, then 
$\bigcap_{R}^i X$ is regarded as 
an $R$-lattice of 
$\bigwedge_{\bb{Q}R}^i \bb{Q} X\ 
(=\bb{Q}\bigwedge_{R}^i X)$
containing the image of 
$\bigwedge_{R}^i X$
via the natural embedding.
The lattice $\bigcap_{R}^i X$
is called \textit{Rubin lattice}.
\end{rem}

\begin{rem}
The $R$-module 
$\bigcap_{R}^i X$
is first introduced 
by Rubin  in \cite{Ru4}
in a different manner from 
Definition \ref{defEPB} 
when $R$ is an $O$-order of 
a semisimple $Q$-algebra
where $O$ is a Dedekind domain, 
and $Q$ is the quotient field of $O$.
(Note that $\bigcap_{R}^i X$ is denoted by 
$\bigwedge_{0}^i X$ in \cite{Ru4}.)
In such cases, 
Rubin defined it as an $O$-lattice of 
$Q\bigcap_{R}^i X$.
In \cite{BS}, 
Burns and Sano generalized
Rubin' lattice over general commutative rings
by using ``exterior power biduals".
In this paper, 
we follow Burns--Sano's definition.
Note that later, we need to treat the case when 
$R$ is a group ring over $\bb{Z}/p^N\bb{Z}$.
\end{rem}

The following conjecture
$\boldsymbol{\mathrm{RS}}(M/K,S,T,V)_p$
is called Rubin--Stark conjecture.
 
\begin{conj}[\cite{BKS} Conjecture 5.1, 
$\boldsymbol{\mathrm{RS}}(M/K,S,T,V)_p$]
We have
\[
\epsilon^V_{M/K,S,T} 
\in \bb{Z}_{(p)}\bigcap_{\bb{Z}[G]}^r \mca{O}_{K,S,T}^\times.
\] 
\end{conj}

Burns--Kurihara--Sano proved that
the $p$-part of the leading term conjecture
implies the $p$-part of the Rubin--Stark conjecture.

\begin{thm}[\cite{BKS} Theorem 5.11]
If the conjecture $\boldsymbol{\mathrm{LT}}(M/K)_p$
holds, then the conjecture
$\boldsymbol{\mathrm{RS}}(M/K,S,T,V)_p$
also holds for any $(S,T,V)$.
\end{thm}

\subsubsection{Mazur--Rubin--Sano conjecture}\label{sssMRS}
Here, we shall introduce the assertion of 
(the $p$-part of) Mazur--Rubin--Sano conjecture
in our setting.
Recall that we have fixed an extension $K_\infty/K$ in \S \ref{secintro}.
Let $F \in \mca{IF}$ be any element. 
We put $M:=K_0 F\langle \mca{O}_K \rangle$, 
and $G:=\Gal(M/K)$.
We take any $N \in \bb{Z}_{>0}$ and 
any $\mf{n} \in 
\mca{N}_{F,N}^{\mathrm{wo}}(\chi_{\mathrm{cyc}}\psi^{-1})
\setminus \{ \mca{O}_K \}$.
We write $\mf{n}=\mf{l}_1 \cdots \mf{l}_i$, 
where $\mf{l}_\nu$ is a prime of $\mca{O}_K$ 
for each $\nu$, and put
$W:=\mathrm{Prime}_K(\mf{n})$
We define $L:=M \langle \mf{n} \rangle$. 
We put $\widetilde{G}:=\Gal(L/K)$, and 
\[
H:=\Gal(L/M)=H_{\mf{n}}
=H_{\mf{l}_1} \times \cdots H_{\mf{l}_i}.
\] 
Suppose that $S$ 
consists of 
all infinite places and all ramified places in $M/K$.
We put $S':=S \cup W$.
We also put $U:=\{ \infty \}$, and $V:= U \cup W$.
Note that we have $U =S' \cap P_K(L) $, 
and $V =S' \cap P_K(M) $.
Here, we assume that both
$\boldsymbol{\mathrm{RS}}(M/K,S',T,V)_p$ and 
$\boldsymbol{\mathrm{RS}}(L/K,S',T,U)_p$ hold.

Recall that for each prime $\mf{l}$, 
we fixed an embedding 
$\mf{l}_{\overline{K}}\colon \overline{K}
\hookrightarrow \overline{K}_{\mf{l}}$, 
and denoted by $\mf{l}_M$ 
the prime of $\mca{O}_M$ corresponding to
$\mf{l}_{\overline{K}} \vert_M$.
For each $\mf{l}_\nu \in W$, 
we write $D_{\mf{l}_\nu}$ be 
the decomposition subgroup of $\widetilde{G}$ 
at $\mf{l}_\nu$.
Since $\mf{n}$ is well-ordered, we may assume that
\[
H_{\mf{l}_\nu} \subseteq D_{\mf{l}_\nu} 
\subseteq \prod_{j=\nu}^{i} H_{\mf{l}_j}
\subseteq H.
\]
For each subgroup $G' \subseteq \widetilde{G}$,
we denote by $I(G')$ the augmentation ideal $\bb{Z}[G']$.
We write $\mca{I}_{\mf{l}_\nu}:= 
I(D_{\mf{l}_\nu})\bb{Z}_p[\widetilde{G}]$ and
$\mca{I}_H := I(H)\bb{Z}_p[\widetilde{G}]$.
We write $I_{\mf{l}_\nu}:= 
I(D_{\mf{l}_\nu})\bb{Z}_p[H]$ and
$I_H := I(H)\bb{Z}_p[H]$.
We define an ideal $J_W$ of $\bb{Z}[H]$ by 
\(
J_W:= \prod_{\nu=1}^i I_{\mf{l}_\nu}
\), and 
$\mca{J}_W:=J_W\bb{Z}_p[\widetilde{G}]$.
Let $\mathrm{rec}_{(\mf{l}_\nu)_M}\colon 
M_{(\mf{l}_\nu)_M}^\times 
\longrightarrow D_{\mf{l}_\nu}$ 
be the local reciprocity map, 
and define the $\bb{Z}[G]$-linear map
\[
\mathrm{Rec}_{\mf{l}_\nu} \colon \mca{O}_{M,S',T}^\times 
\otimes_{\bb{Z}} \bb{Z}_p
\longrightarrow (\mca{I}_{\mf{l}_\nu})_H
:=\mca{I}_{\mf{l}_\nu}/\mca{I}_H
\mca{I}_{\mf{l}_\nu};\ 
a \longmapsto \sum_{\tau \in G/H} \tau^{-1}
(\mathrm{rec}_{(\mf{l}_\nu)_M}(\tau a)-1). 
\]
Then, the maps $\mathrm{Rec}_{\mf{l}_\nu}$ 
induce the $\bb{Z}[\widetilde{G}]$-linear map
\[
\mathrm{Rec}_W \colon \bb{Z}_p
\bigcap_{\bb{Z}[G]}^{i+1} \mca{O}_{M,S',T}^\times 
\longrightarrow 
\bigcap_{\bb{Z}[G]}^{1} \mca{O}_{M,S',T}^\times 
\otimes_{\bb{Z}[G]} \mca{J}_W/\mca{I}_H\mca{J}_W
=\mca{O}_{M,S,T}^\times 
\otimes_{\bb{Z}} (J_W)_H,
\]
where we regard $(J_W)_H:=J_W/I_HJ_W$
as a trivial $G$-module.
We have a natural injective $\bb{Z}[G]$-linear map
\[
\nu \colon \mca{O}_{M,S',T}^\times 
\otimes_{\bb{Z}} (J_W)_H 
\hookrightarrow 
\mca{O}_{L,S',T}^\times 
\otimes_{\bb{Z}} \bb{Z}_p[H]/I(H)J_W,
\]
where we declare that 
the action of $G$ on $\bb{Z}_p[H]/I(H)J_W$ is also trivial.
We define 
a $\bb{Z}[G]$-linear map
\[
\mca{N}_H \colon 
\mca{O}_{M,S',T}^\times \otimes_{\bb{Z}} \bb{Z}_p
\longrightarrow 
\mca{O}_{M,S',T}^\times 
\otimes_{\bb{Z}} \bb{Z}_p[H]/I(H)J_W
\]
by $\mca{N}_H(a) = 
\sum_{\sigma \in H} \sigma a \otimes \sigma^{-1}$.

In our setting, the assertion of 
the Mazur--Rubin--Sano conjecture is as follows.

\begin{conj}[\cite{BKS} Conjecture 5.4, 
$\boldsymbol{\mathrm{MRS}}(L/M/K,T,W)_p$]
We have 
\[
\mca{N}_H \left( \epsilon^V_{M/K,S',T} \right)
\in \mathrm{Im}(\nu),
\]
and it holds that 
\[
\mca{N}_H \left( \epsilon^U_{L/K,S',T} \right)
=(-1)^{i} \cdot 
\nu\left(
\mathrm{Rec}_W 
\left( \epsilon^{V}_{M/K,S',T} \right)
\right).
\]
\end{conj}

The following lemma implies that 
the map $\mca{N}_H$ is closely related to 
the Kolyvagin operator $D_{\mf{n}} \in \bb{Z}[H]$.

\begin{lem}[\cite{BS} Lemma 4.27 and Lemma 4.28]\label{lemsnDn}
We put 
\[
J_{W}^{\circ}:= 
\left(
\prod_{\nu=1}^i I(H_\nu) 
\right)
\bb{Z}_p[H].
\]
For each ideal $\mf{d}$ of $\mca{O}_K$
dividing the ideal $\mf{n}$, we denote by 
\[
\pi_{\mf{d}}\colon \bb{Z}_p[H]/
I_HJ_{W}
\longrightarrow \bb{Z}_p[H]/
I_HJ_{W}
\]
the endomorphism induced by 
the projection map
$H=H_{\mf{n}} \longrightarrow H_{\mf{d}} \subseteq H$.
We define the endomorphism $s_{\mf{n}}$ 
on $\bb{Z}_p[H]/I_HJ_W$ by 
\(
s_{\mf{n}}:= \sum_{\mf{d} \mid \mf{n}}
(-1)^{\epsilon (\mf{n}/\mf{d})} \pi_{\mf{d}}
\).
Then, the following hold.
\begin{enumerate}[{\rm (i)}]
\item The endomorphism $s_{\mf{n}}$ 
is a projector 
of $J_{W}^{\circ}/
I_HJ^{\circ}_{W}$
in the following sense: 
the image of  $s_{\mf{n}}$
coincides with 
$J^{\circ}_{W}/I_HJ^{\circ}_{W}$, 
and the restriction of $s_{\mf{n}}$ on 
$J_{W}/I_HJ_{W}$
is the identity map.
\item Let $X$ be a $\bb{Z}_p[H]$-module.
Then for any $x \in X$, the map 
\[
\mathrm{id}_X \otimes s_{\mf{n}} \colon 
X \otimes_{\bb{Z}_p} \bb{Z}_p[H] \longrightarrow 
X \otimes_{\bb{Z}_p} \bb{Z}_p[H]
\] 
sends the element
$\sum_{\sigma \in H} \sigma x \otimes \sigma^{-1}$
to 
\[
(-1)^{\epsilon({\mf{n})}}D_{\mf{n}} x 
\otimes \prod_{j=1}^i (\sigma_{\mf{l}_j}-1)
\in X \otimes_{\bb{Z}_p[H]} 
J^{\circ}_{W}/I_HJ^{\circ}_{W},
\]
where $D_{\mf{n}}$ is the Kolyvagin operator. 
\end{enumerate}
\end{lem}

Burns--Kurihara--Sano proved that
the $p$-part of the leading term conjecture
implies the $p$-part of 
the Mazur--Rubin--Sano conjecture.

\begin{thm}[\cite{BKS} Theorem 5.15]
If the conjecture $\boldsymbol{\mathrm{LT}}(L/K)_p$
holds, then the conjecture
$\boldsymbol{\mathrm{MRS}}(L/M/K,T,W)_p$
also holds for any $T$.
\end{thm}

\subsection{Conjectures over imaginary quadratic fields}\label{ssETNC/K}

In \cite{Bl}, Bley proved 
the equivariant Tamagawa number conjecture
over imaginary quadratic fields in certain cases.
Here, we review this work briefly.

Keep $K$ to denote an imaginary quadratic field.
Bley proved the following theorem.

\begin{thm}[\cite{Bl} Theorem 4.2]\label{etnc/K}
Let $p$ be a prime number splitting in $K/\bb{Q}$,
and not dividing $\# \mathrm{Cl}(K)$.
Then, for any abelian extension $M/K$
the conjecture
$\boldsymbol{\mathrm{LT}}(M/K)_p$
holds. 
\end{thm}

Let $p$ be a prime number, and
$K_\infty /K$ be the-extension fixed in \S 1.
For each non-zero ideal $\mf{a}$ of $\mca{O}_K$,
we define 
$w(\mf{a}):=\# \{ \zeta \in 
(\mca{O}_K^\times)_{\mathrm{tor}} 
\mathrel{\vert} \zeta \equiv 1 \mod \mf{a} \}$.
We have an explicit description of
Stark units, namely 
the first rank Rubin--Stark element
with $V=\{ \infty \}$, 
in terms of elliptic units
in certain situations.

\begin{thm}\label{thmellRS}
Let $\psi \in \widehat{\Delta}$ 
be a character satisfying 
$\bar{\psi}\vert_{D_{\Delta,\mf{p}}} \ne 1$
for any prime ideal $\mf{p}$ of $\mca{O}_K$ above $p$.
We assume that 
the conductor $\mf{f}$ of $K_0/K$
satisfies $w (\mf{f})=1$.
Take any element $F \in \mca{IF}$, 
and put $M:=K_0F$.
Let $S$ be the set of places of $\mca{O}_K$
consisting of all places dividing $p\infty$
and all unramified places in $K_0/K$.
Let $\mf{l}$ be a prime of $\mca{O}_K$
not contained in $S$, and 
satisfying $w(\mf{l})=1$.
Let $\mf{n}$ be a square-free ideal of $\mca{O}_K$
prime to all elements of $S \cup \{ \mf{l} \}$, 
and splitting completely in $K/\bb{Q}$. 
We put $S' := S \cup \mathrm{Prime}_K(\mf{n})$.
Then, we have
\[
\epsilon^{\{ \infty \}}_{M 
\langle \mf{n} \rangle/K,S',\{ \mf{l} \}, \psi}
= u \cdot c^{\mf{l}}_{\mf{f}}(F;\mf{n})
\in (\mca{O}_{M \langle \mf{n} \rangle,
S' ,\{ \mf{l} \} }^\times \otimes \bb{Z}_p)_\psi 
\]
for some $u \in \bb{Z}_p [\Gal(M \langle \mf{n} \rangle/K)]^\times$, 
where the scalar action of $\bb{Z}_p
[\Gal(M \langle \mf{n} \rangle/K)]$
on the unit group 
$\mca{O}^\times_{M \langle \mf{n} 
\rangle ,S',\{ \mf{l} \}} \otimes \bb{Z}_p$
is written in multiplicative manner.
\end{thm}

\begin{rem}
The existence of the Stark units 
(namely the integrality of the Rubin--Stark elements 
for $V=\{ \infty \}$) 
of (arbitrary) abelian extension fields 
of imaginary quadratic fields 
is proved in \cite{St}.
Note that Theorem \ref{thmellRS} follows
form (for instance)  
Proposition \ref{propell} and 
\cite{Bl} the equality (10) in page 90.
\end{rem}

\subsection{Higher Fitting ideals and Rubin--Stark elements}\label{ssFitBKS}

Here, we recall the work on higher Fitting ideals
by Burns, Kurihara and Sano in \cite{BKS}
focused on our setting.
We use the same notation as above.

As in \S \ref{secintro}, we fix an extension
$K_\infty /K$.
Fix $F \in \mca{IF}$, and put $M:=K_0F$.
Suppose that $S$ consists of 
all infinite places and all ramified places in $M/K$.
Let $T$ be a non-empty finite set 
of finite places of $K$
such that 
\[
\mca{O}_{M,S,T}^\times :=\{ 
u \in \mca{O}_{M,S}^\times \mathrel{\vert}
u \equiv 1 \ \mathrm{mod}\ v \
\text{for all $v \in T$}
\}
\] 
is a torsion-free $\bb{Z}$-module.
We define 
$A^T(M):= \mathrm{Cl}^T(M) 
\otimes_{\bb{Z}}\bb{Z}_p$.
As in \S \ref{secintro}, we put 
$A(M):=A^{\emptyset}(M)$.

Let $\psi \in \widehat{\Delta}$
be a non-trivial faithful character 
such that $\psi \vert_{D_{\Delta,v}} \ne 1$
for any finite place $v$ in $S$.
We put $R_\psi:=\mca{O}_\psi[\Gal(F/K)]$, and 
\(
\mca{H}_\psi:= \Hom_{R_\psi}((\mca{O}_{M,S,T}^\times 
\otimes_{\bb{Z}} \bb{Z}_p)_\psi, R_\psi).
\)
By the assumption on $\psi$, 
we have $R_\psi$-isomorphisms
$A^T(M)_\psi \simeq (\mathrm{Cl}_S^T(M) 
\otimes_{\bb{Z}}\bb{Z}_p)_\psi$ and 
$(\mca{O}_{M,T}^\times\otimes_{\bb{Z}} \bb{Z}_p)_\psi
\simeq 
(\mca{O}_{M,S,T}^\times\otimes_{\bb{Z}} \bb{Z}_p)_\psi$.

For any $i \in \bb{Z}_{\ge 0}$ and any subset $Q$ of $P_K$,
we define $\mca{V}^{Q}_{i}(M/K)$ to be the set of 
all the subset $W$  of $P_K(M)$
satisfying $\# W=i$ and $W \cap Q = \emptyset$.
We put $U:=\{ \infty \}=S \cap P_K(M)$. 
(Note that by the assumption on $\psi$,
the primes above $p$ 
do not split completely in $M/K$.)

\begin{dfn}\label{defTheta}
For any $i \in \bb{Z}_{\ge 0}$ 
and any subset $\mca{V}'$ of $\mca{V}^{S \cup T}_i(M/K)$, 
we define
\[
\Theta^{\mathrm{RS}}_{S,T,i}(M/K;\mca{V}')
:= \left\{
\Phi \left(
\epsilon^{U \cup W}_{M/K,S,T,\psi}
\right)
\mathrel{\bigg\vert} W \in \mca{V}', \ 
\Phi \in \bigwedge^{i+1}_{R_\psi}\mca{H}_\psi
\right\}.
\]
We write 
$\Theta^{\mathrm{RS}}_{S,T,i}(M/K):=
\Theta^{\mathrm{RS}}_{S,T,i}
(M/K;\mca{V}_i^{S \cup T}(M/K))$.
\end{dfn}

Burns, Kurihara and Sano obtained the following result.

\begin{thm}[\cite{BKS} Corollary 1.7]\label{thmBKSFitt}
Let $\psi \in \widehat{\Delta}$
be a non-trivial faithful character 
such that $\psi \vert_{D_{\Delta,v}} \ne 1$
for any finite place $v$ in $S$.
Suppose that $\boldsymbol{\mathrm{LT}}(M/K)_p$
holds after taking 
$(-)\otimes_{\bb{Z}[\Delta]}\mca{O}_\psi$.
Then, for any $i \in \bb{Z}_{\ge 0}$, we have
\[
\Fitt_{R_\psi,i}(A^T(M)_\psi)
=\Theta^{\mathrm{RS}}_{S,T,i}(M/K)_\psi.
\]
\end{thm}

Suppose that 
$\psi \ne \omega$ as characters of $G_{K}$.
Then, the Chebotarev density theorem implies that
there exists a finite place $v$ of $K$
with $w(v)=1$, 
not contained in $S$ and satisfying 
$\psi\vert_{G_{K_v}} \ne \omega\vert_{G_{K_v}}$.
We can take $T$ as a singleton containing such $v$. 
Then, we obtain $A^T(M)_\psi=A(M)_\psi$.
By combining Theorem \ref{thmBKSFitt} with
Proposition \ref{etnc/K}, 
we obtain the following corollary.

\begin{cor}\label{corFittRS}
Suppose that $p$ splits completely in $K/\bb{Q}$, 
and $p$ is prime to $\# \mathrm{Cl}(K)$.
Let $F \in \mca{IF}$ be any element.
Let $\psi \in \widehat{\Delta}$
be a non-trivial faithful character 
such that $\psi \vert_{D_{\Delta,v}} \ne 1$
for any finite place $v$ in $S$.
Moreover we assume that if $K_0$ contains $\mu_p$, 
then $\psi \ne  \omega$.
$\psi \ne \omega$ as characters of $G_{K}$.
Let $\mf{l}$ be an ideal of $\mca{O}_K$
with $w(v)=1$, 
not contained in $S$ and satisfying 
$\psi\vert_{G_{K_{\mf{l}}}} \ne \omega\vert_{G_{K_{\mf{l}}}}$.
We put $T:=\{ \mf{l} \}$. 
Let $N$ be any positive integer, and put
\(
\mca{V}_{i}^{\mf{l}}(F;N):=
\{ \mathrm{Prime}_K(\mf{n}) \mathrel{\vert} 
\mf{n} \in 
\mca{N}_{F,N}^{\mathrm{wo}}(\chi_{\mathrm{cyc}}\psi^{-1}) \}
\cap \mca{V}^{S \cup \{ \mf{l} \}}_{i}
\). 
Then, for any $i \in \bb{Z}_{\ge 0}$, we have
\[
\Fitt_{R_\psi,i}(A(K_0F)_\psi)
=\Theta^{\mathrm{RS}}_{S,T,i}(K_0F/K; 
\mca{V}_{i}^{\mf{l}}(F;N))_\psi
\]
\end{cor}

\begin{rem}
In the arguments in proof of 
\cite{BKS} Theorem 7.9 using Chebotarev density theorem, 
the sets $V'$ are taken from $\mca{V}'$. 
However, after taking the $\psi$-part, 
more careful arguments implies that it suffices 
to take the sets $V'$ from $\mca{V}_{i}^{\mf{l}}(F;N)$. 
(For instance, see \cite{Oh1} Proposition 3.15, 
which is the ``over $F$ version" of 
Proposition \ref{Cheb}.)
Consequently, we can use 
$\mca{V}_{i}^{\mf{l}}(F;N)$
in the assertion of Corollary \ref{corFittRS}
in stead of $\mca{V}_{i}^{S \cup T}(M/K)$.
\end{rem}

\subsection{Comparison of ideals}\label{sspfC=Theta}

Let $K_\infty /K$ be as in \S \ref{secintro}.
Fix $F \in \mca{IF}$, and put $M:=K_0F$.
We denote by $S$ the subset of $P_K$ consisting of  
all infinite places and all ramified places in $M/K$.
We put $U:= \{ \infty \}=S \cap P_K(M)$.
Here, let us show the following theorem, 
which says that the ideal
$\mf{C}^{\mathrm{ell}}_{i}(F/K)_\psi$ 
coincides with
$\Theta^{\mathrm{RS}}_{S,T,i}(M/K)$
in the situation of Corollary \ref{corFittRS}.

\begin{thm}\label{thmTheta=C}
Suppose that $p$ splits completely in $K/\bb{Q}$, 
and $p$ is prime to $\# \mathrm{Cl}(K)$.
Let $\psi \in \widehat{\Delta}$
be a non-trivial faithful character 
such that $\psi \vert_{D_{\Delta,v}} \ne 1$
for any finite place $v$ in $S$.
Moreover we assume that if $K_0$ contains $\mu_p$, 
then $\psi \ne  \omega$.
Let $\mf{l}$ be an ideal of $\mca{O}_K$
with $w(v)=1$, 
not contained in $S$ and satisfying 
$\psi\vert_{G_{K_{\mf{l}}}} \ne \omega\vert_{G_{K_{\mf{l}}}}$.
We put $T;=\{ \mf{l} \}$. 
Then, for any $i \in \bb{Z}_{\ge 0}$ and 
any $F \in \mca{IF}$, we have
\[
\mf{C}^{\mathrm{ell}}_{i}(F/K)_\psi
=\Theta^{\mathrm{RS}}_{S,T,i}(M/K)_\psi.
\]
\end{thm}

Let us introduce some notation. 
We put $G:=\Gal(M/K)$, and put $R:=\bb{Z}_p[G]$.
Let $\psi \in \widehat{\Delta}$ and 
$T:=\{ \mf{l} \}$ 
be as in Theorem \ref{thmTheta=C}.
We put $R_\psi :=\mca{O}_\psi[\Gal(F/K)]$.
Let $\nu \in \bb{Z}_{>0} \cup \{ \infty \}$
be any element.
We  put 
$R_{\nu}:=R/p^\nu R$, and  
$R_{\nu,\psi}:=R_\psi/p^\nu R_\psi$.
Let $N \in \bb{Z}_{>0}$ be any element, 
and $\mf{n} \in 
\mca{N}_{F,N}^{\mathrm{wo}}(\chi_{\mathrm{cyc}}\psi^{-1})$ 
an element prime to $\mf{l}$.
We put 
\begin{align*}
\mathcal{U}_{\mf{n}, \nu}:=& 
\mca{O}_{M,
S\cup \mathrm{Prime}_K(\mf{n}),T}^\times 
\otimes_{\bb{Z}} \bb{Z}_p/p^\nu \bb{Z}_p, \\
\widetilde{\mathcal{U}}\langle \mf{n} \rangle_{\nu}:=& 
\mca{O}_{M \langle \mf{n} \rangle ,
S,T}^\times 
\otimes_{\bb{Z}} \bb{Z}_p /p^\nu \bb{Z}_p.
\end{align*}
Note that since we assume that $p \mid \# \mathrm{Cl}(\mca{O}_K)$, 
we have 
$\widetilde{\mathcal{U}}\langle \mca{O}_K 
\rangle_{\nu,\psi}=\mathcal{U}_{\mca{O}_K,\nu,\psi}$.
For any $R_{\nu}$-module $A$, 
we write $A^{*_\nu}:=\Hom_{R_{\nu}}
(A,R_{\nu})$.

Let us consider the modulo $p^\nu$ reduction of 
the ideal $\Theta^{\mathrm{RS}}_{S,T,i}(M/K)$. 
We need the following lemma.

\begin{lem}\label{lemhomsurj}
Let $X$ be a $\bb{Z}_p[G]$-module 
which is torsion free as a $\bb{Z}_p$-module.
We put 
$\overline{X}_{\nu}:=U^\times 
\otimes_{\bb{Z}} \bb{Z}_p/p^{\nu}\bb{Z}_p$.
for any $\nu \in \bb{Z}_{>0} \cup \{ \infty \}$.
Then, for any $\nu \in \bb{Z}_{>0}$, 
the modulo $p^\nu$ map
\(
X^{*_\infty}
\longrightarrow 
(\overline{X}_{\nu})^{*_\nu} 
\)
becomes a surjection.
In particular, we have 
an $R_{\nu}$-isomorphism
\(
X^{*_\infty} \otimes_{\bb{Z}}\bb{Z}/p^\nu \bb{Z}
\simeq 
(\overline{X}_{\nu})^{*_\nu}
\).
\end{lem}

\begin{proof}
Let $\nu' \in \bb{Z}_{> 0} \cup \{ \infty \}$
be any element.  
We define a projection map 
\[
\mathrm{pr}_{\nu'} \colon R_{\nu'}
\longrightarrow \bb{Z}_p/p^{\nu'}\bb{Z}_p;\ 
\sum_{g \in G} a_g g \longmapsto a_{e_G},
\]
where $e_G$ denotes the identity element of $G$,
and define 
\[
P_{\nu'} \colon 
(\overline{X}_{\nu'})^{*_{\nu'}} \longrightarrow 
\Hom_{\bb{Z}_p/p^{\nu'}\bb{Z}_p}(\overline{X}_{\nu'}, 
\bb{Z}_p/p^{\nu'}\bb{Z}_p);\ f \longmapsto \mathrm{pr}_{\nu'} \circ f.
\]
Note that $P_{\nu'}$ is a bijection since the map 
\begin{align*}
Q \colon (\overline{X}_{\nu'})^{*_{\nu'}}  & \longrightarrow 
\Hom_{\bb{Z}_p/p^{\nu'}\bb{Z}_p}(\overline{X}_{\nu'},
\bb{Z}_p/p^{\nu'}\bb{Z}_p) \\ 
h & \longmapsto \left(
x \longmapsto \sum_{g} h(g^{-1}x)g 
\right)
\end{align*}
becomes the inverse of $P_{\nu'}$.
Since $X$ is 
a torsion free $\bb{Z}_p$-module, 
we have a commutative diagram 
\[
\xymatrix{
X^{*_\infty}  \ar[d]_{\mod p^\nu}
\ar[r]^(0.3){P_\infty}_(0.3){\simeq} &
\Hom_{\bb{Z}_p}(X,\bb{Z}_p), 
\bb{Z}_p) \ar@{->>}[d]^{\mod p^\nu} \\
(\overline{X}_{\nu})^{
*_{\nu}}\ar[r]^(0.3){P_\nu}_(0.3){\simeq}  & 
\Hom_{\bb{Z}_p/p^{\nu}\bb{Z}_p}(\overline{X}_{\nu}, 
\bb{Z}_p/p^{\nu}\bb{Z}_p),
}
\]
whose right vertical map is surjective.
This implies the assertion of 
Lemma \ref{lemhomsurj}.
\end{proof}

\begin{cor}\label{corEPBmodp}
We have a natural $R_{\nu}$-isomorphism  
\[
\left( 
\bigcap^{i+1}_R \mathcal{U}_{\mf{n}, \infty}
\right) \otimes_{\bb{Z}_p} \bb{Z}_p/p^\nu \bb{Z}_p
\simeq \bigcap^{i+1}_{R_{\nu}} 
\mathcal{U}_{\mf{n}, \nu}.
\]
\end{cor}

\begin{proof}
By the definition of exterior power biduals,
the $R$-module 
$\bigcap^{i+1}_R \mathcal{U}_{\mf{n}, \infty}$ 
is torsion free. 
So, by Lemma \ref{lemhomsurj}, we have
\[
\left( 
\bigcap^{i+1}_R \mathcal{U}_{\mf{n}, \infty}
\right) \otimes \bb{Z}/p^\nu \bb{Z}
\simeq \left( 
\left( 
\bigwedge^{i+1}_R 
(\mathcal{U}_{\mf{n}, \infty})^{*_\infty} 
\right) \otimes \bb{Z}/p^\nu \bb{Z}
\right)^{*_\nu}
\]
Since the exterior power commutes with the base change, 
we have 
\[
\left( 
\left( 
\bigwedge^{i+1}_R 
(\mathcal{U}_{\mf{n}, \infty})^{*_\infty} 
\right) \otimes \bb{Z}/p^\nu \bb{Z}
\right)^{*_\nu}
\simeq 
\left( 
\bigwedge^{i+1}_{R_{\nu}} 
\left( 
(\mathcal{U}_{\mf{n}, \infty})^{*_\infty} 
\otimes \bb{Z}/p^\nu \bb{Z}\right)
\right)^{*_\nu}
\]
Finally, since $\mca{O}_{M,S',T}^\times \otimes \bb{Z}_p$ 
is a torsion free, we obtain
\[
\left( 
\bigwedge^{i+1}_{R_{\nu}}
\left( 
(\mathcal{U}_{\mf{n}, \infty})^{*_\infty} 
\otimes \bb{Z}/p^\nu \bb{Z}\right)
\right)^{*_\nu}
\simeq \bigcap^{i+1}_{R_{\nu}} 
\mathcal{U}_{\mf{n}, \nu},
\]
by Lemma \ref{lemhomsurj}.
Hence we obtain the isomorphism as desired.
\end{proof}

We put $W:=\mathrm{Prime}_K(\mf{n})$.
We define $V:=U \cup W$, and $S':=S \cup W$.
By Corollary \ref{corEPBmodp}, 
the image of the Rubin--Stark element 
$\epsilon^{V}_{M/K,
S',T,\psi}$
by modulo $p^\nu$ reduction
can be naturally regarded as 
an element of 
$\bigcap^{i+1}_{R_{\nu}} 
\mathcal{U}_{\mf{n}, \nu}$.

\begin{dfn}
We define 
\[
\Theta_\nu(\mf{n})
:= \left\{
\overline{\Phi} \left(\epsilon^{V}_{M/K,
S',T,\psi}
\right)
\mathrel{\bigg\vert} 
\overline{\Phi} \in \bigwedge^{i+1}_{R_{\nu,\psi}}
\overline{\mca{H}}_{\mf{n},\nu,\psi}
\right\},
\]
where we write 
\(
\overline{\mca{H}}_{\mf{n},\nu,\psi}:= 
\Hom_{R_{\nu,\psi}}
\left(
\mathcal{U}_{\mf{n}, \nu, \psi}, R_{\nu,\psi}
\right)
\).
We define an ideal $\Theta_{i,N,\nu}$ of 
$R_{\nu,\psi}$
to be the one generated by 
\[
\bigcup_{\mf{n} \in 
\mca{N}_{F,N}^{\mathrm{wo}}(\chi_{\mathrm{cyc}}\psi^{-1})}
\Theta_\nu(\mf{n}).
\]
\end{dfn}

Since 
$\mca{O}_{M,S,T}^\times 
\otimes_{\bb{Z}} \bb{Z}_p$ is a 
torsion free $\bb{Z}_p$-module,
we immediately obtain the following lemma 
from Lemma \ref{lemhomsurj}.

\begin{lem}\label{corhomsurj}
The image of $\Theta_\infty (\mf{n})$ 
in $R_\psi$ coincides with 
$\Theta_\nu (\mf{n})$.
\end{lem}

On the one hand, 
by Corollary \ref{corEPBmodp},
we have 
\[
\Theta^{\mathrm{RS}}_{S,T,i}(M/K; 
\mca{V}_{i}^{\mf{l}}(F;N))_\psi
=\varprojlim_\nu \Theta_{i,N,\nu}
\]
for any $N \in \bb{Z}_{>0}$.
On the other hand, 
by Corollary \ref{corFittRS},
we have 
\[
\Theta^{\mathrm{RS}}_{S,T,i}(M/K)_\psi
=\Theta^{\mathrm{RS}}_{S,T,i}(M/K; 
\mca{V}_{i}^{\mf{l}}(F;N))_\psi
\]
for any $N \in \bb{Z}_{>0}$.
Therefore, we obtain
\[
\Theta^{\mathrm{RS}}_{S,T,i}(M/K)_\psi
= \varprojlim_{\nu,N} \Theta_{i,N,\nu}
= \varprojlim_{N} \Theta_{i,N,N}.
\]
By Remark \ref{remCigen}, we have
\(
\mf{C}_{i,F,N}(Z^{\mathrm{ell}}_\psi )=
\mf{C}_{i,F,N}( \{\mf{c}_{\mf{f},\psi}^{\mf{l}} \})
\)
for any $i \in \bb{Z}_{\ge 0}$ and 
any $N \in \bb{Z}_{> 0}$.
In order to prove 
Theorem \ref{thmTheta=C}, 
it suffices to show that 
\(
\mf{C}_{i,F,N}(\{ 
\boldsymbol{c}^{\mf{l}}_{\mf{f}} \})
=
\Theta_{i,N,N}
\)
for any $i \in \bb{Z}_{\ge 0}$ 
and $N \in \bb{Z}_{>0}$.
So, it is sufficient to prove that 
\begin{align}\label{eqThetanvsCn}
\mf{C}_{F,N}(\mf{n}; 
\boldsymbol{c}^{\mf{l}}_{\mf{f}}) 
& \subseteq 
\Theta_N(\mf{n}), \\
\Theta_{i,N,N} \label{eqThetanvsCn2}
& \subseteq 
\mf{C}_{i,F,N}(\{ 
\boldsymbol{c}^{\mf{l}}_{\mf{f},\psi} \})
\end{align}
for any $N \in \bb{Z}_{>0}$ and 
any $\mf{n} \in 
\mca{N}_{F,N}^{\mathrm{wo}}(\chi_{\mathrm{cyc}}\psi^{-1})$.

\begin{proof}[Proof of Theorem \ref{thmTheta=C}]
First, 
let us show the inequality (\ref{eqThetanvsCn}).
Again, we fix any $N \in \bb{Z}_{>0}$ and 
any $\mf{n} \in 
\mca{N}_{F,N}^{\mathrm{wo}}(\chi_{\mathrm{cyc}}\psi^{-1})$.
We write $\mf{n}=\mf{l}_1 \cdots \mf{l}_i$, 
where $\mf{l}_j$ is a prime of $\mca{O}_K$ 
for each $j$.
Here, we assume that $p \nmid \# \mathrm{Cl}(K)$.
So, we have a natural decomposition
\[
H:=\Gal(M \langle \mf{n} \rangle /M)
\simeq H_{\mf{l}_1} \times \cdots \times H_{\mf{l}_i}.
\]
We put $\widetilde{G}:=\Gal(M \langle \mf{n} \rangle /K)$.
Then, we have $G=\widetilde{G}/H$.
Similarly 
to as above, 
we define $\mathcal{U}_{\mf{n}, \nu}$,
$\widetilde{\mathcal{U}} \langle \mf{n} \rangle_\nu$, 
$W$, $V$ and $S'$. 
We define an integer $N(\mf{n})$ by
\(
\bigotimes_{j=1}^i H_{\mf{l}_j}
\simeq \bb{Z}/p^{N(\mf{n})}\bb{Z}
\).
Note that we have $N(\mf{n}) \ge N$.

By Theorem \ref{etnc/K}, 
the conjecture 
$\boldsymbol{\mathrm{MRS}}
(M \langle \mf{n} \rangle/M/K,T,W)_p$ holds.
The assumption that $\psi \vert_{D_{\Delta,v}} \ne 1$
for any finite place $v$ contained in $S$
implies that if $\psi$ is unramified at $\mf{p}$, 
then $1 - \psi(\Fr_{\mf{p}}) \in \mca{O}_\psi^\times$.
So, by Theorem \ref{thmellRS}, 
there exists a unit $u'_\psi \in R_\psi[H_{\mf{n}}]^\times$
such that 
\(
\epsilon^{U}_{M 
\langle \mf{n} \rangle/K,S',T,\psi}
= u'_\psi \cdot 
c^{\mf{l}}_{\mf{f}}(F;\mf{n})_\psi
\in \widetilde{\mathcal{U}} \langle \mf{n} 
\rangle_{\infty,\psi}
\).
By Lemma \ref{lemsnDn}, we have
\begin{equation}\label{eqsnDne}
\left( \mca{N}_H \left( 
\epsilon^{\{ \infty \}}_{M 
\langle \mf{n} \rangle/K,S',\{ \mf{l} \}}
\right) \right)
= (-1)^{\epsilon({\mf{n})}} u'_\psi D_{\mf{n}} 
c^{\mf{l}}_F(\mf{n})  
\otimes \prod_{j=1}^i (\sigma_{\mf{l}_j}-1)
\end{equation}
Since the conjecture $\boldsymbol{\mathrm{MRS}}
(M \langle \mf{n} \rangle/M/K,T,W)$
holds, the left hand side of 
(\ref{eqsnDne})
is fixed by the action of $H$.

In order to prove (\ref{eqThetanvsCn}), 
we need to study the map $\mathrm{Rec}_W$
appearing  in the statement of 
$\boldsymbol{\mathrm{MRS}}
(M \langle \mf{n} \rangle/M/K,T,W)$.
Let $\mf{l}_j \in W$ be any element, and 
denote by $D_{\mf{l}_j}$
the decomposition subgroup of 
$\widetilde{G}$ at $\mf{l}_j$.
In Definition \ref{dfnsigmaell}, we have fixed 
a generator $\sigma_{\mf{l}_j}$ of 
the cyclic group $H_{\mf{l}_j}$.
The projection map map
$D_{\mf{l}_j} \longrightarrow H_{\mf{l}_j}
\subseteq D_{\mf{l}_j}$ induces the surjection
\(
\xymatrix{
\pi_{\mf{l}_j} \colon
I_{\mf{l}_j}/I_{\mf{l}_j}I_H
\ar@{->>}[r] &
I(H_{\mf{l}_j})/I(H_{\mf{l}_j})^2
}
\).
We have an isomorphism
\[
r_{\mf{l}_j} \colon 
I(H_{\mf{l}_j})/I(H_{\mf{l}_j})^2
\xrightarrow{\ \simeq \ } 
H_{\mf{l}_j};\ \sigma-1 \longmapsto \sigma
\]
of abelian groups.
We define $\overline{\mathrm{Rec}}_{\mf{l}_j,N,\psi}
\colon \mathcal{U}_{\mf{n}, N, \psi} \longrightarrow R_{N,\psi}$
to be the unique map which make the diagram
\[
\xymatrix{
\mathcal{U}_{\mf{n}, N, \psi} \ar[rrr]^(0.4){\mathrm{Rec}_{\mf{l}_j,\psi}
\ \mathrm{mod}\ p^N} 
\ar[rrrd]_{\overline{\mathrm{Rec}}_{\mf{l}_j,N,\psi}
\otimes \sigma_{\mf{l}_j}}
& & & R_{N,\psi} 
\otimes I_{\mf{l}_j}/I_{\mf{l}_j}I_H 
\ar[d]^{\mathrm{id} 
\otimes (r_{\mf{l}_j} \circ \pi_{\mf{l}_j}) } \\
& & & R_{N,\psi}\otimes H_{\mf{l}_j}
}
\]
commute, where $\mathrm{Rec}_{\mf{l}_j,\psi} \ \mathrm{mod}\ p^N$ 
denotes the modulo $p^N$ reduction of 
the $\psi$-component of 
the map $\mathrm{Rec}_{\mf{l}_j}$ defined in \S \ref{sssMRS}.
These maps induce the natural map
\begin{align*}
\bigwedge^i_{j=1} \overline{\mathrm{Rec}}_{\mf{l}_j,N,\psi} 
\colon
\bigcap_{R_{N,\psi}}^{i+1} \mathcal{U}_{\mf{n}, N, \psi} 
\longrightarrow &
\bigcap_{R_{N,\psi} }^{1} 
\mathcal{U}_{\mf{n}, N, \psi} 
= \Hom_{R_{N,\psi}} \left(
\overline{\mca{H}}_{\mf{n},N, \psi}, 
R_{N,\psi} \right)
\end{align*}
by sending an element $u \in 
\bigcap_{R_{N,\psi}}^{i+1} \mathcal{U}_{\mf{n}, N, \psi} $
to the map 
\[
\left((\mathcal{U}_{\mf{n}, N, \psi} )^{*_N} 
\ni \varphi \longmapsto \left(
\bigwedge^i_{j=1} 
\overline{\mathrm{Rec}}_{\mf{l}_j,N,\psi} 
\wedge \varphi\right)(u) \in R_{N,\psi}
\right)  
\]
Let $J_W^\circ$ be as in  \ref{lemsnDn}.
Then, we have an isomorphism
\[
r_\mf{n} \colon 
J_{W}^\circ /J_{W}^\circ I_H 
\xrightarrow{\ \simeq \ } 
\bigotimes_{j=1}^i H_{\mf{l}_j};\ 
(\sigma_1-1)\cdots (\sigma_i-1) 
\longmapsto  \bigotimes_j \sigma_j.
\]
By the definition 
of the $\mathrm{Rec}_{W}$ 
and Corollary \ref{corEPBmodp}, 
we obtain the commutative diagram
\[
\xymatrix{ 
\bigcap_{R_{N,\psi}}^{i+1} \mathcal{U}_{\mf{n}, N, \psi} 
\ar[rrrd]_{(\mf{s}_{\mf{n}} \circ \mathrm{Rec}_{W})\ 
\mathrm{mod}\ p^N \hspace{10mm}\ } 
\ar[rrr]^(0.4){\bigwedge^i_{j=1} 
\overline{\mathrm{Rec}}_{\mf{l}_j,N,\psi}} & & & 
\left(
\bigcap_{R_{N,\psi}}^{1}
\mathcal{U}_{\mf{n}, N, \psi}
\right)
\otimes_{\bb{Z}} 
\bigotimes_{j=1}^i H_{\mf{l}_j} \\
& & &
\left(
\bigcap_{R_{N,\psi}}^{1}
\mathcal{U}_{\mf{n}, N, \psi} \right)
\otimes_{\bb{Z}} J_{W}^\circ /J_{W}^\circ I_H 
\ar[u]^{\simeq}_{\mathrm{id}\otimes r_{\mf{n}}}. 
}
\]

Now, let us complete the proof of (\ref{eqThetanvsCn}).
Let $f\colon H^1(F, (\bb{Z}/p^N\bb{Z})\otimes 
\chi_{\mathrm{cyc}}\psi^{-1}) \longrightarrow 
R_{N,\psi}$
be any $R_{N,\psi}$-linear map, 
and denote by $f_0$ the restriction of $f$
to $\mathcal{U}_{\mf{n}, N, \psi}$.
By Hochschild--Serre spectral sequence,
we have an natural injection  
\(
\iota_{\mf{n},N,\psi} 
\colon \mathcal{U}_{\mf{n}, N, \psi}
\hookrightarrow 
\widetilde{\mathcal{U}} 
\langle \mf{n} \rangle_{N,\psi}^H
\).
Recall that we have 
an $R_{N,\psi}[H]$-isomorphism
\[
\nu_H \colon R_{N,\psi} 
\xrightarrow{\ \simeq \ } R_{N,\psi}[H];\ 
x \longmapsto N_Hx . 
\]
By the injectivity of the 
$R_{N,\psi}[H]$-module $R_{N,\psi}[H]$, 
there exists an $R_{N,\psi}[H]$-linear map 
$g \colon \widetilde{\mathcal{U}} 
\langle \mf{n} \rangle_{N,\psi} \longrightarrow 
R_{N,\psi}[H]$ which makes the diagram
\[
\xymatrix{
\widetilde{\mathcal{U}} 
\langle \mf{n} \rangle_{N,\psi}
\ar@{-->}[r]^(0.5){g} & R_{N,\psi}[H] \\
\mathcal{U}_{\mf{n}, N, \psi}
\ar@{^{(}->}[u]^{\iota_{\mf{n},N,\psi}} 
\ar[ru]_{\nu_H \circ f_0}&
}
\]
commute.
Then, we obtain
\begin{align*}
\nu_H \left(
\bigwedge_{j=1}^i 
\overline{\mathrm{Rec}}_{\mf{l}_j,N,\psi}
\wedge f_0
\left(
\epsilon^{V }_{M/K,S',T,\psi}
\right)
\right)
&=g \left(
\mca{N}_H \left(
\epsilon^{U }_{M 
\langle \mf{n} \rangle/K,S',T}
\right)
\right) \\
& =g \left(
 (-1)^{\epsilon({\mf{n})}} u'_\psi D_{\mf{n}} 
c^{\mf{l}}_{\mf{f}}(F;\mf{n})_\psi 
\right) \\
&= \bar{u}'_\psi \cdot
(\nu_H \circ f_0) \left( 
\kappa_{F,N}
(\mf{n}; 
\mf{c}^{\mf{l}}_{\mf{f}})
\right),
\end{align*}
where $\bar{u}'_\psi$ denotes the image of 
$u'_\psi \in R_{\psi}[H]^\times$ 
in $R_{N,\psi}^\times$.
Hence we obtain 
\[
f_0 \left( 
\kappa_{F,N}
(\mf{n}; 
\mf{c}^{\mf{l}}_{\mf{f},\psi})
\right)
= (\bar{u}'_\psi)^{-1} \cdot 
\left(
\bigwedge_{j=1}^i 
\overline{\mathrm{Rec}}_{\mf{l}_j,N,\psi}
\wedge f_0
\right)
\left(
\epsilon^{V }_{M/K,S',T,\psi}
\right)
\in \Theta_N(\mf{n}).
\]
This implies the inequality (\ref{eqThetanvsCn}).

Next, let us show (\ref{eqThetanvsCn2}).
Again, we fix $\mf{n}=\mf{l}_1 \cdots \mf{l}_i \in 
\mca{N}_{F,N}^{\mathrm{wo}}(\chi_{\mathrm{cyc}}\psi^{-1})$.
We take arbitrary maps 
$\overline{\Phi}_\nu \in \overline{\mca{H}}_{\mf{n},N, \psi}$
with $1 \le \nu \le i+1$.

We follow the arguments in the proof of 
\cite{BKS} Theorem 9.6. 
By the Chebotarev density theorem, 
we can take an ideal $\mf{r}= \mf{q}_1 
\cdots \mf{q}_{i+1} \in 
\mca{N}_{F,N}^{\mathrm{wo}}(\chi_{\mathrm{cyc}}\psi^{-1})$, 
which is prime to $\mf{n}$ satisfying the following properties.
\begin{enumerate}[(i)]
\item For any $\nu \in \bb{Z}$ with $1 \le \nu \le i$,
the class $[ \mf{l}'_{\nu} ]_\psi$
in $A(M)_\psi$ coincides with  $[ \mf{l}_{\nu} ]_\psi$.
\item There exists a sequence 
$\{ z_\nu \}_{\nu=1}^i \subseteq 
M^\times$ 
satisfying the following
properties. 
\begin{itemize}
\item We have $(z_\nu \mca{O}_M)_{\psi}
=(\mf{q}_\nu\mf{l}_\nu^{-1})_\psi \in 
(I_F \otimes _{\bb{Z}} \bb{Z}_p)_\psi$
for any $\nu \in \bb{Z}$ with $1 \le \nu \le i$, 
where $I_F$ denotes the group of 
the fractional ideals of $\mca{O}_F$.
\item We have $\overline{\mathrm{Rec}}_{\mf{q}_\mu,N,\psi}(z_\nu)=0$
for any distinct $\mu$ and $\nu$.
\end{itemize} 
\item For any $\nu \in \bb{Z}$ with $1 \le \nu \le i+1$, 
we have
\(
\overline{\Phi}_\nu (x)= 
\overline{\mathrm{Rec}}_{\mf{q}_\nu,N,\psi}(x)
\)
for any $x \in {U}_{\mf{n}, N, \psi}$.
\end{enumerate}
(See \cite{Oh1} Proposition 3.15.)

For each $j \in \bb{Z}$ with $1\le j \le i$, we put 
$\mf{n}_j;=\prod_{\nu=1}^j \mf{q}_\nu 
\prod_{\mu=j+1}^i \mf{l}_\mu$, and 
define  $W_j:= \mathrm{Prime}_K(\mf{n}'_j)$.
We put $S'_j:=S \cup W_j$, and 
$V_j:=U \cup W_j$.
We also put $\mf{n}_0:=\mf{n}$, 
and define $S'_0:=S'$ and $V_0:=V$.

\begin{lem}\label{lemVV}
For any $j \in \bb{Z}$ with $0 \le j \le i$, we have
\begin{equation}\label{eqindepsilonPhiRec}
\left(
\bigwedge_{\nu=1}^i \overline{\Phi}_\nu
\right)
\left(
\epsilon^{V }_{M/K,S',T}
\right)
\equiv 
\left(
\bigwedge_{\nu=1}^{i+1} 
\overline{\mathrm{Rec}}_{\mf{q}_\nu,N,\psi}
\right)
\left(
\epsilon^{V_j }_{M/K,S'_j,T}
\right)
\mod \Theta_{i,N,N}.
\end{equation}
\end{lem}

\begin{proof}[Proof of Lemma \ref{lemVV}]
Let us prove 
the congruence (\ref{eqindepsilonPhiRec})
by induction on $j$.
When $j=0$, the congruence (\ref{eqindepsilonPhiRec}) 
follows from the property (c) 
of $\mf{r}$.

Let $j_0 \in \bb{Z}$ with $1 \le j \le i$, and 
suppose the congruence (\ref{eqindepsilonPhiRec}) 
holds for $j=j_0-1$.
Then, let us prove the congruence 
(\ref{eqindepsilonPhiRec}) for $j=j_0$.
We denote by $m$
the order of the ideal class 
$[\mf{l}_{j_0}]$ in 
$\mathrm{Cl}^{T}(M)$, 
and let $a \in 
\mca{O}_{M,S'_{j_0}, T}^\times$ 
be such that 
$[\mf{l}_{j_0}]^m=(a)$.
We put $W_{j_0}^{\circ}:=W_{j_0}
\setminus \{ \mf{q}_{j_0} \}=W_{j_0-1}
\setminus \{ \mf{l}_{j_0} \}$.
We set $S^{\circ}_{j_0}
:=S \cup W_{j_0}^{\circ}$, 
and $V^{\circ}_{j_0}
:=U \cup W_{j_0}^{\circ}$.
Then, by the characterization of 
the Rubin--Stark elements, 
we have
\[
\epsilon^{{V_{j_0-1}} }_{M/K,S'_{j_0-1},T}
= \frac{(-1)^{j_0}}{m}\cdot 
a \wedge \epsilon^{V^{\circ}_{j_0}} 
_{M/K,S^{\circ}_{j_0}, T}
\in \bb{Q} \bigcap^{1+i}_{R}
\mca{O}_{M,S'_{j_0-1}, T}^\times.
\]
By using this presentation, we obtain
the equality
\begin{align}\label{leftepsilongocha}
&\epsilon^{ V_{j_0-1} }_{M/K,S'_{j_0-1},T}
+(-1)^{j_0}\cdot z_{j_0} \wedge 
\epsilon^{V^{\circ}_{j_0} }_{M/K,S^\circ _{j_0},T} \\
=& \label{rightepsilongocha}
(-1)^{j_0}\left(
\frac{1}{m}\cdot a+z_{j_0} \right) 
\wedge \epsilon^{V^\circ_{j_0} }_{
M/K,S^\circ_{j_0},T},
\end{align}
in $\bb{Q} \bigcap^{1+i}_{\bb{Z}[G]}
\mca{O}_{M, S^\circ_{j_0} \cup 
\widetilde{W}, T}^\times$, 
where $\widetilde{W}$ is a finite subset of $P_K$ 
such that $a$ and $z_{j_0}$
belong to
$\mca{O}_{M, S^\circ_{j_0} \cup 
\widetilde{W}, T}^\times$, 
and satisfying $S^\circ_{j_0} \cap \widetilde{W}=\emptyset$.  
Fix an isomorphism $\bb{C} \simeq \overline{\bb{Q}}_p$, 
and regard $\psi$ as a $\bb{C}$-valued character.
The $\psi$-component of 
the right hand side (\ref{rightepsilongocha})
belongs to 
\[
(\bb{C} \bigcap^{i+1}_{\bb{Z}[G]}
\mca{O}_{M,S'_{j_0}, T}^\times)_\psi
=\bigcap^{i+1}_{\bb{C}[\Gal(F/K)]}
(\bb{C}\mca{O}_{M,S'_{j_0}, T}^\times)_\psi
\]
since we have
\[
\left(
\frac{1}{m}\cdot a+x  
\right)_\psi
\in \Ker \left(
\bigoplus_{w \mid v \in \widetilde{W}^\circ} 
\mathrm{ord}_w
\colon \bb{C}
\mca{O}_{M,S^\circ_{j_0} \cup 
\widetilde{W}, T}^\times \rightarrow 
\bb{C} X_{M,\widetilde{W}}
\right)_\psi
=\left(
\bb{C}\mca{O}_{M,S'_{j_0}, T}^\times
\right)_\psi,
\]
where we put $\widetilde{W}^\circ:=\widetilde{W} 
\setminus \{ \mf{q}_{j_0} \}$.
By computing the image of (\ref{rightepsilongocha}) 
via the regulator,
we deduce that the $\psi$-component of 
(\ref{rightepsilongocha})
coincides with that of the Rubin--Stark element
\[
\epsilon^{{V_{j_0}} }_{M/K,S'_{j_0},T,\psi}
\in \left(
\bb{C} \bigcap^{i+1}_{\bb{Z}[G]}
\mca{O}_{M,S'_{j_0}, T}^\times
\right)_\psi.
\]
On the other hand, the left hand side 
(\ref{leftepsilongocha})
belongs to 
$\bb{Z}_p\bigcap^{i+1}_{\bb{Z}[G]}
\mca{O}_{M,S^\circ_{j_0} \cup \widetilde{W},T}^\times$, so we can 
apply the map 
$\left(
\bigwedge_{\nu=1}^{i+1} 
\overline{\mathrm{Rec}}_{\mf{q}_\nu,N,\psi}
\right)$ to (\ref{leftepsilongocha}).
By the property (b) of $\mf{r}$, 
we obtain 
\begin{align*}
&\left(
\bigwedge_{\nu=1}^{i+1} 
\overline{\mathrm{Rec}}_{\mf{q}_\nu,N,\psi}
\right)
\left(
\epsilon^{V_{j_0-1}}_{M/K,S'_{j_0-1},T,\psi}
+(-1)^{j_0}\cdot z_{j_0,\psi} \wedge \epsilon^{
V^{\circ}_{j_0}}_{M/K,S^{\circ}_{j_0},T,\psi}
\right) \\
=&\left(
\bigwedge_{\nu=1}^{i+1} 
\overline{\mathrm{Rec}}_{\mf{q}_\nu,N,\psi}
\right)
\left(
\epsilon^{V_{j_0-1} }_{M/K,S'_{j_0-1},T,\psi}
\right) \\
& \pm 
\overline{\mathrm{Rec}}_{\mf{q}_{j_0},N,\psi}
 (z_{j_0,\psi}) 
\cdot 
\left(
\bigwedge_{\mu=1}^{j_0-1} 
\overline{\mathrm{Rec}}_{\mf{q}_{\mu},N,\psi}
\wedge
\bigwedge_{\nu=j_0+1}^{i+1}
\overline{\mathrm{Rec}}_{\mf{q}_{\nu},N,\psi}
\right) \left(
\epsilon^{
V^{\circ}_{j_0}}_{M/K,S^{\circ}_{j_0},T,\psi}
\right) \\
\equiv &
\left(
\bigwedge_{\nu=1}^{i+1} 
\overline{\mathrm{Rec}}_{\mf{q}_\nu,N,\psi}
\right)
\left(
\epsilon^{V_{j_0-1} }_{M/K,S'_{j_0-1},T,\psi}
\right)
\mod \Theta_{i,N,N}.
\end{align*} 
This congruence and 
the induction  hypothesis  imply 
(\ref{eqindepsilonPhiRec}) 
for $j=j_0$.
\end{proof}

Let us show (\ref{eqThetanvsCn2}) 
by induction on $i$. 
When $i=0$, the inequality 
(\ref{eqThetanvsCn2})  clearly holds.

Let $i_0 \in \bb{Z}_{\ge 1}$, and 
suppose that (\ref{eqThetanvsCn2}) for $i=i_0-1$ holds.
We take any element $\mf{n} \in 
\mca{N}_{F,N}^{\mathrm{wo}}(\chi_{\mathrm{cyc}}\psi^{-1})$
with $\epsilon (\mf{n})=i_0$.
In order to prove (\ref{eqThetanvsCn2}) for $i=i_0$,
it suffices to show that 
\(
\Theta_{N}(\mf{n}) \subseteq 
\mf{C}_{i_0,F,N}(\{ 
\boldsymbol{c}^{\mf{l}}_{\mf{f},\psi} \})
\).

Let $\mf{r}$ be as in Lemma \ref{lemVV} 
(corresponding to present $\mf{n}$),
and put $H':=H_{\mf{r}}$.
Similarly to the proof of 
(\ref{eqThetanvsCn}), 
there exists an $R_{N,\psi}[H']$-linear map 
$h \colon \widetilde{\mathcal{U}} 
\langle \mf{r} \rangle_{N,\psi}
\longrightarrow 
R_{N,\psi}[H']$ which makes the diagram
\[
\xymatrix{
\widetilde{\mathcal{U}} 
\langle \mf{r} \rangle_{N,\psi}
\ar@{-->}[r]^(0.5){h} & R_{N,\psi}[H'] \\
\mathcal{U}_{ \mf{r}, N,\psi}
\ar@{^{(}->}[u]^{\iota_{\mf{r},N,\psi}} \ar[ru]_{\nu_{H'} \circ 
\overline{\mathrm{Rec}}_{\mf{q}_{i+1},N,\psi} }
}
\]
commute.
So, by By Lemma \ref{lemsnDn} and By Lemma \ref{lemVV}, 
we obtain
\begin{align*}
\left(
\bigwedge_{\nu=1}^{i+1} \overline{\Phi}_\nu
\right)
\left(
\epsilon^{V }_{M/K,S',T,\psi}
\right)
& \equiv 
\left(
\bigwedge_{\nu=1}^{i+1}
\overline{\mathrm{Rec}}_{\mf{q}_\nu,N,\psi}
\right)
\left(
\epsilon^{V_i }_{M/K,S'_i,T}
\right) 
\mod \mf{C}_{i,F,N}(\{ 
\boldsymbol{c}^{\mf{l}}_{\mf{f},\psi} \}) \\
& \equiv
\overline{\mathrm{Rec}}_{\mf{q}_{i+1},N,\psi} 
\left( 
\kappa_{F,N}
(\mf{r}; 
\mf{c}^{\mf{l}}_{\mf{f},\psi})
\right) 
\mod \mf{C}_{i,F,N}(\{ 
\boldsymbol{c}^{\mf{l}}_{\mf{f},\psi} \}) \\
&  \equiv 0
\mod \mf{C}_{i,F,N}(\{ 
\boldsymbol{c}^{\mf{l}}_{\mf{f},\psi} \}).
\end{align*}
Hence we obtain (\ref{eqThetanvsCn2}).
This completes the proof of Theorem \ref{thmTheta=C}.
\end{proof}


\begin{thebibliography}{BKS}

\bibitem[Bl]{Bl}
Bley, W., \textit{Equivariant Tamagawa number conjecture 
for abelian extensions of a quadratic imaginary field}, 
Documenta Mathematica \textbf{11} (2006), 73--118.


\bibitem[BF]{BF}
Burns, D, and Flach, M., 
\textit{Tamagawa numbers for motives with (non-commutative) coefficients}, 
Documenta Mathematica \textbf{6} (2001), 501--570.

\bibitem[BG]{BG}
Burns, D.\ and Greither, C., 
\textit{On the Equivariant Tamagawa Number Conjecture 
for Tate motives}, Invent.\ math.\ 
\textbf{153} (2003) 303--359.


\bibitem[BK]{BK}
Bloch, S. and Kato, K., 
\textit{L-functions and Tamagawa numbers of motives}, 
The Grothendieck Festschrift, Vol.\ I, 333--400, Progr. Math. 
\textbf{86}, Birkh\"user Boston, Boston, MA, 1990.


\bibitem[BKS]{BKS} Burns, D., Kurihara, M., and Sano, T., 
\textit{On zeta elements for Gm}, 
Documenta Mathematica \textbf{21} (2016), 555--626.


\bibitem[BS]{BS} Burns, D.\ and Sano, T., 
\textit{On the theory of higher rank Euler, 
Kolyvagin and Stark systems}, 
preprint, arXiv:1612.06187, [v1].



\bibitem[dS]{dS} de Shalit, E., 
\textit{Iwasawa Theory of Elliptic Curves with Complex Multiplication}, 
Perspect. Math. 3, Academic Press, Boston, 1987.



\bibitem[Fl]{Fl} Flach, M., 
\textit{On the cyclotomic main conjecture for the prime $2$}, 
J.\ reine angew.\ Math. \textbf{661} (2011) 1--36.


\bibitem[Ka]{Ka} Kato, K., 
\textit{$p$-adic Hodge theory and values of 
zeta functions of modular forms}, 
in Cohomologies $p$-adiques et applications arithm\'etiques, III,
Ast\'erisque \textbf{295}, Soc.\ Math.\ France, 
Montrouge, 2004, 117--290. 


\bibitem[KM]{KM}
Knudsen, F., and Mumford, D., 
\textit{The projectivity of the moduli space of stable
curves I: Preliminaries on `det' and `Div'}, 
Math.\ Scand.\ \textbf{39} (1976) 19--55.


\bibitem[Ku1]{Ku} Kurihara, M.,
  \textit{Refined Iwasawa theory and Kolyvagin systems of Gauss sum type}, 
   Proceedings of the London Mathematical Society (3) 
\textbf{104} (2012), 728--769. 


\bibitem[Ku2]{Ku2} Kurihara, M.,
\textit{Refined Iwasawa theory for $p$-adic 
representations and the structure of Selmer groups}, 
M\"unster J.\ Math.\ \textbf{7} (2014), 149--223. 



\bibitem[MR]{MR} Mazur, B., and Rubin, K.,
  \textit{Kolyvagin systems}, 
Memoirs of the AMS Vol \textbf{168}, Number \textbf{799} (2004). 






\bibitem[Oc]{Oc}
Ochiai, T. {\it Euler system for Galois deformations}, 
Ann.\ Inst.\ Fourier (Grenoble), \textbf{55} (1), 
113--146, (2005).




\bibitem[Oh1]{Oh1}
Ohshita, T., \textit{On higher Fitting ideals of 
Iwasawa modules of ideal class groups 
over imaginary quadratic fields and Euler systems of elliptic units}, 
Kyoto Journal of Mathematics \textbf{53} (4), 845--887 (2013).


\bibitem[Oh2]{Oh2}
Ohshita, T., \textit{On the higher Fitting ideals of 
Iwasawa modules of ideal class groups over real abelian fields}, 
J.\ Number Theory \textbf{135} (2014), 67--138.




\bibitem[Ru1]{Ru1} 
Rubin, K., \textit{Kolyvagin's system of Gauss sums}, 
Arithmetic algebraic geometry, 309--324, 
Progr.\ Math., \textbf{89}, Birkh\"auser, 1991.


\bibitem[Ru2]{Ru2} 
Rubin, K., \textit{The ``main conjectures" 
of Iwasawa theory for imaginary quadratic fields}, 
Invent.\ Math.\ \textbf{103} (1991), 25--68.


\bibitem[Ru3]{Ru3}  Rubin, K., \textit{More ``main conjectures" 
for imaginary quadratic fields} 
in Elliptic Curves and Related Topics, ed.\ H.\ Kisilevsky and R.\ Murty, 
CRM Proc.\ Lecture Notes \textbf{4}, Amer.\ Math.\ Soc., Providence, 
1994, 23--28.

\bibitem[Ru4]{Ru4}
Rubin, K., \textit{A Stark Conjecture `over $\bb{Z}$' 
for abelian $L$-functions with multiple zeros}, 
Ann.\ Inst.\ Fourier \textbf{46} (1996), 33--62.



\bibitem[Ru5]{Ru5}
Rubin, K.,
{\it Euler systems},
{Hermann Weyl lectures, Ann.\ of Math.\ Studies, vol.\ {\bf 147}, 
Princeton Univ.\ Press} (2000).


\bibitem[St]{St}
Stark, H.\ M., \textit{L-functions at $s=1$. IV. 
First Derivatives at $s = 0$}, 
Adv.\ Math.\ \textbf{35} (1980), 197--235.



\bibitem[Ta]{Ta}
Tate, J.,
{\it Relation between $K_2$ and Galois cohomology} 
Invent.\ Math.\ \textbf{36} (1976), 257--274.






\end{thebibliography}
\end{document}